\documentclass[preprint]{elsarticle}
\usepackage{lineno}

 \usepackage{booktabs}
\usepackage{slashbox}
\usepackage{amsmath}
\usepackage{amssymb}
\usepackage{amsthm}
\usepackage{graphicx}
\usepackage{epic,eepic,epsfig}
\usepackage{color}
\usepackage{subfigure}  
\usepackage{placeins}   
\usepackage{multirow}
\usepackage{epstopdf}
\usepackage{varwidth}
\usepackage{float}
\usepackage[table]{xcolor}
\usepackage{bm}

\newtheorem{thm}{Theorem}[section]
\newtheorem{exm}{Example}[section]
\newtheorem{lemma}{Lemma}[section]
\newtheorem{remark}{Remark}[section]

\definecolor{tabclr}{cmyk}{0,0,1,0}

\allowdisplaybreaks[4]

 \usepackage{caption2}
\usepackage{algorithm}
\usepackage{algorithmic}

\allowdisplaybreaks[4]

\begin{document}

\title{Double-activation neural network for solving parabolic equations with time delay}
\author[bjut]{Qiumei Huang}
\ead{qmhuang@bjut.edu.cn}
\author[bjut]{Qiao Zhu\corref{cor}}
\ead{qiaozhubjut@163.com}
\cortext[cor]{Corresponding author}
\address[bjut]{School of Mathematics, Statistics and Mechanics, Beijing University of Technology, Beijing 100124, China}

\begin{abstract} This paper presents the double-activation neural network (DANN), a novel network architecture designed for solving parabolic equations with time delay.   In DANN, each neuron is equipped with two activation functions to augment the network's nonlinear expressive capacity.  Additionally, a new parameter is introduced for the construction of the quadratic terms in one of two activation functions, which further  enhances the network's ability to capture complex nonlinear relationships. 
To address the issue of low fitting accuracy caused by the discontinuity of solution’s derivative, a piecewise fitting approach is proposed by dividing the global solving domain into several subdomains. The convergence of the loss function is proven. Numerical results are presented to demonstrate the superior accuracy and faster convergence of DANN compared to the traditional physics-informed neural network (PINN).
\end{abstract}
 \maketitle
\bigskip
\section{Introduction}\label{sec1}
Delay differential equations (DDEs) and delay partial differential equations (DPDEs) are important mathematical models that describe various real-life phenomena such as biology, control and climate prediction.

In this paper, we consider the following semilinear parabolic DPDEs
	\begin{eqnarray}
    \begin{cases}
          \partial_{t}u(t,x)=\alpha \Delta u(t,x)+h(t,x,u(t-\tau(t,u(t,x)),x),u(t,x)), \quad (t,x) \in \Omega_T:=[0,T] \times \Omega, \\
        u(t,x)=g(t,x), \quad (t,x) \in \Omega_0:=[-\tau_0,0] \times \Omega,\\ 
        u(t,x)=b(t,x), \quad (t,x) \in \partial\Omega_T:=[0,T] \times \partial\Omega ,
     \end{cases}
\label{equation}
\end{eqnarray}
where $\Omega \subset \mathbb{R}^{d}$ is a bounded domain, $\alpha$ is a constant, and $\Delta$ is the Laplacian operator on $\Omega$. The delay function $\tau(t,u(t,x)) > 0$, and $ t-\tau(t,u(t,x))$ is monotonic with respect to $t$.  Additionally, $\tau_0$ is defined as $\tau_0 = -\min_{(t,x) \in \Omega_T} \{t- \tau(t,u(t,x))\}$. Disregarding spatial factors and  setting $\alpha=0$, ($\ref{equation}$) degenerates into DDEs.

The exact solutions of DDEs are not available in general, and therefore one has to rely on numerical methods to find approximate solutions. There are many numerical methods for DDEs including linear multistep methods \cite{hong1996numerical}, Runge--Kutta methods \cite{bellen2013numerical, zhang2019impulsive}, $\theta$-methods \cite{rihan2022analysis,xu2004h}, collocation methods \cite{brunner2004collocation, brunner2001geometric, engelborghs2001collocation},
spectral methods \cite{yang2013spectral, zayernouri2014spectral} and finite element methods \cite{brunner2010discontinuous,huang2019postprocessing, huang2011superconvergence, huang2013hp, huang2016continuous}. DPDEs not only take into account the temporal delay in the evolution of a system's state but also incorporate spatial variations, making them ideal for modeling practical problems. The numerical methods of DPDEs mainly include finite difference methods \cite{ansari2007parameter}, spectral methods \cite{li2011ldg, liu2015spectral} and finite element methods \cite{Dai2023ExponentialTD, xu2023discontinuous}.

For DDEs with state-dependent delay, the delay term depends on the unknown solution, which makes it difficult to obtain the discontinuous points in advance. The current common method involves initially obtaining an approximate solution to identify the discontinuous points. This is followed by a procedure to refine these points, which then serve as the basis for solving a more accurate approximation. Subsequently, the discontinuous points are further corrected based on this improved solution, with the entire process being iterated to enhance precision progressively. This method has a very high computational cost and storage requirements. Therefore, finding efficient numerical methods to solve DDEs with state-dependent delay is of great importance.

In recent years, as a very simple and efficient method to solve PDEs, deep learning methods has attracted more and more attention \cite{weinan2021dawning,han2018solving,karniadakis2021physics}. Physics-informed neural network (PINN) \cite{guo2022normalizing, raissi2019physics} is one of the popular deep learning methods for solving PDEs with deep neural networks. Compared to the traditional mesh-based methods, such as the finite difference
method and the finite element method, PINN could be a mesh-free approach by taking advantage of the automatic differentiation. PINN is widely used for solving a variety of forward and inverse PDEs problems \cite{guo2022monte, lu2021deepxde, tang2023pinns, wang2021understanding}.

Deep learning methods have been successfully applied to solve DDEs. Fang et al. \cite{fang2020neural} adopted a very simple feedforward neural network with a hidden layer for solving the first-order DDEs with constant delay. Khan et al. \cite{khan2020design} presented a novel computing paradigm by exploiting the strength of feedforward
neural networks with Levenberg-Marquardt method, and Bayesian regularization method
 based backpropagation for solving DDEs of pantograph type. Panghal et al. \cite{panghal2022neural} applied feedforward neural network to solve DDEs of pantograph type,
where the neural networks are trained using the extreme learning machine algorithm. Liu et al. \cite{liu2021legendre} proposed a novel Legendre neural network combined with the extreme learning machine algorithm to solve variable coefficients linear delay differential-algebraic equations with weak discontinuities.
However, deep learning methods have not yet been applied to DDEs with state-dependent delay and DPDEs. In this paper, PINN for DDEs with state-dependent delay and parabolic DPDEs are obtained. Moreover, considering discontinuous points may influence the approximate accuracy, we propose a piecewise fitting method for parabolic DPDEs with nonvanishing delay by dividing the whole domain into several subdomains according to the primary discontinuous points.

Some researchers have explored various network structures to improve the learning capabilities of PINNs. Yang et al. \cite{yang2020physics} proposed a physics-informed generative adversarial network (GAN) by combining GAN with PINN, which improved the modeling and solving capabilities for stochastic processes. Jagtap et al. \cite{jagtap2020adaptive} introduced a PINN with an adaptive activation function (APINN), by incorporating a scalable hyper-parameter through a scaling factor in the activation function to accelerate the training convergence of PINNs. Bu et al. \cite{bu2021quadratic} proposed quadratic residual networks (QRES) by adding a quadratic residual term to the weighted sum of inputs for solving forward and inverse PDE problems, and made comparisons with identity shortcut (ISC) and quadratic shortcut (QSC), which are two baseline architectures sharing similar ideas with ResNet \cite{he2016deep}. Inspired by these different neural architectures, we propose a novel neural network called the double-activation neural network (DANN), which enhances the network’s nonlinear expressive capacity by equipping each neuron with two activation functions and introducing an additional parameter to construct quadratic terms. We compare the experimental results of four numerical examples obtained by PINN, APINN, QRES, ISC, QSC, and DANN models, which show that DANN can achieve the highest fitting accuracy with a small number of training points.

This paper is organized as follows. Section 2 discusses deep learning algorithms for solving equation ($\ref{equation}$) and introduces two improved strategies for PINN, namely DANN and piecewise fitting. Section 3 describes the convergence result of the loss function. Section 4 presents numerical examples to demonstrate the effectiveness of these improvement strategies.

\section{Deep learning for semilinear parabolic DPDEs}
In this section, we present main algorithm of this paper. We first describe PINN for solving equation ($\ref{equation}$), including training data, loss function and optimization algorithms. Then, DANN and piecewise fitting are proposed.
\subsection{PINN for semilinear parabolic DPDEs}
PINN is a type of neural network that incorporates physical equations as constraints for solving partial differential equations. By adding the residual of the physical equations into the loss function of the neural network, PINN allows the physics to be integrated into the training process. First, we give the expression of the approximate solution $\hat{u}(t, x ; \theta)$ for ($\ref{equation}$). The approximate solution is defined as
\begin{eqnarray}
\hat{u}(t, x ; \theta)=b(t, x)+B(x)  \mathcal{N}(t, x ; \theta).
\label{approximatesolution}
\end{eqnarray}
Here, $\mathcal{N}(t, x ; \theta)$ represents a neural network with parameters $\theta$, and $B(x)$ ensures that the approximate solution accurately satisfies the boundary conditions.

For notational convenience, we  denote
$$\mathcal{G}[u](t,x)=\alpha \partial_{t}u(t,x)-\Delta u(t,x)-h(t,x,u(t-\tau(t,u(t,x)),x),u(t,x)).$$ In the PINN framework, the approximate solution $\hat{u}(t,x;{\theta})$ of ($\ref{equation}$) is obtained by minimizing the following loss function
\begin{eqnarray}
\begin{aligned}
   \mathrm{Loss}(\hat{u})&=\Vert \mathcal{G}[\hat{u}](t,x;{\theta})\Vert_{2,\Omega_T}^{2}
   +\Vert \hat{u}(t,x;{\theta})-g(t,x)\Vert_{2,\Omega_0}^{2},
\end{aligned}
\end{eqnarray}
where $\Vert \cdot \Vert_{2,\Omega_T}^{2}=\iint_{\Omega_T} | \cdot |^2dtdx$ and $\Vert \cdot \Vert_{2,\Omega_0}^{2}=\iint_{\Omega_0} | \cdot |^2dtdx$. The loss function $\mathrm{Loss}(\hat{u})$ measures how well the approximate solution satisfies the differential operator and initial condition. To discretize the loss function, Latin hypercube sampling is employed. The interior collocation points are denoted as $S_{\Omega_{T}}=\{{t_i},{x_i}\}_{i=1}^{N_f}$, and the initial training data points are denoted as $S_{\Omega_0}=\{{t_j},{x_j}\}_{j=1}^{N_{I}}$. $N_f$ and $N_I$ represent the number of interior collocation points and initial training data respectively. By combining these selected training points, the discretized loss function is obtained as follows
\begin{eqnarray}
\begin{aligned}
   \mathrm{Loss}_N(\hat{u})&=\Vert \mathcal{G}[\hat{u}](t,x;{\theta})\Vert_{N_f,S_{\Omega_{T}}}^{2}
   +\Vert \hat{u}(t,x;{\theta})-g(t,x)\Vert_{N_{I},S_{\Omega_0}}^{2},
\end{aligned}
\end{eqnarray}
where
\begin{align*}
&\|\mathcal{G}[\hat{u}](t, x ; \theta)\|_{N_f, S_{\Omega_T}}^2 = \frac{1}{N_f} \sum_{i=1}^{N_f}\big(\alpha \partial_t \hat{u}(t_i, x_i ; \theta)-\Delta \hat{u}(t_i, x_i ; \theta) \\
&\quad\quad\quad\quad\quad\quad\quad\quad\quad\quad-h(t_i, x_i, \hat{u}(t_i-\tau(t_i, \hat{u}(t_i, x_i ; \theta)), x_i ; \theta), \hat{u}(t_i, x_i ; \theta)) \big)^2, \\
&\|\hat{u}(t, x ; \theta)-g(t, x)\|_{N_I, S_{\Omega_0}}^2 = \frac{1}{N_I} \sum_{j=1}^{N_I}\left(\hat{u}\left(t_j, x_j ; \theta\right)-g\left(t_j, x_j\right)\right)^2 .
\end{align*}

To minimize the loss function and obtain the optimal parameters for the neural network, we employ the Limited-memory Broyden-Fletcher-Goldfarb-Shanno with Bounds (L-BFGS-B) optimization algorithm \cite{zhu1997algorithm}, which combines the BFGS method with line search techniques, making it a robust and efficient optimization algorithm for constrained optimization problems. Moreover, we adopt the Xavier initialization method \cite{glorot2010understanding} to initialize the weights of $\hat{u}$, which can better initialize the network and help accelerate convergence speed.
\subsection{Double-activation neual network (DANN)}	
Most existing work in PINN predominantly utilizes deep neural network (DNN) architectures.  Each neuron structure in DNN can be represented as shown in Figure $\ref{DNN}$. Mathematically, a DNN layer can be expressed as $y^{\text{DNN}} = \sigma(Wx+b)$, where $(W, b)$ are the learnable parameters and $\sigma$ represents a nonlinear activation function. It is important to note that the linear transformation $Wx+b$ only introduces linearity to the outputs, while the nonlinearity is introduced through the activation function $\sigma$. As a consequence, in order to capture an adequate level of nonlinearity with satisfactory accuracy, a significant number of DNN layers with reasonable widths are required.
\begin{figure}[h]
	\centering
		\includegraphics[width=0.5\linewidth]{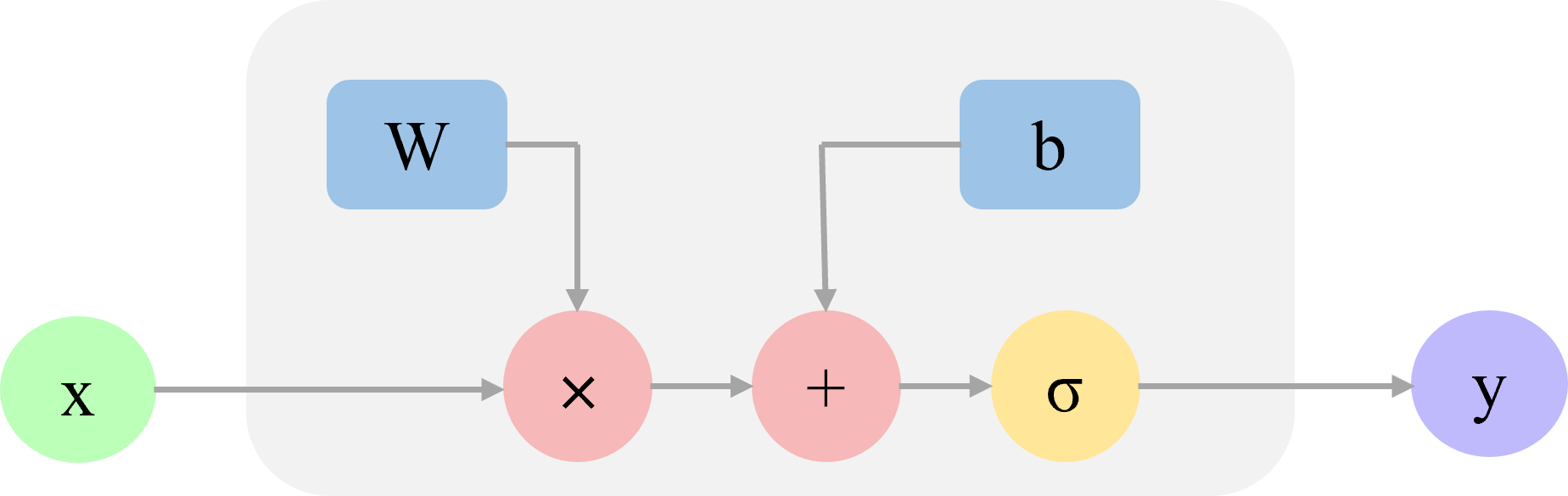}
		\caption{The neuron structure in DNN.}
 \label{DNN}
\end{figure}
\begin{figure}[h]	
    \centering
		\includegraphics[width=0.5\linewidth]{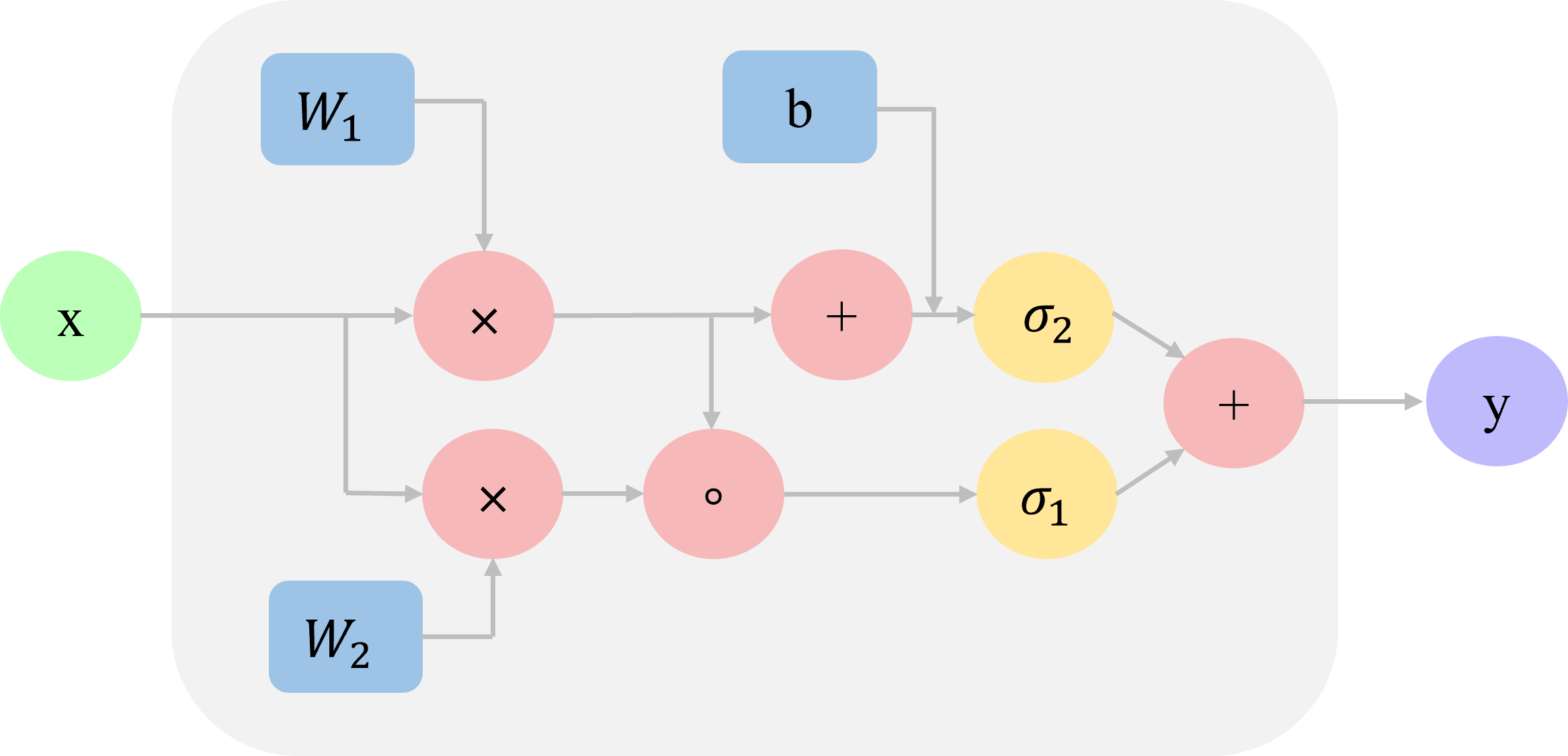}
		\caption{The neuron structure in DANN.}
 \label{DANN}
 \end{figure}

We introduce quadratic terms $W_2x \circ W_1x$ and add two activation functions into each neuron at every layer of DANN to contribute additional nonlinearity, where $\circ$ denotes the Hadamard product. Specifically, a single layer of DANN can be expressed as $y^{\text{DANN}} = \sigma_1(W_2x \circ W_1x) + \sigma_2(W_1x + b)$, where $\widetilde{{\theta}}=(W_1, W_2, b)$ are the learnable parameters. $\sigma_1$ and  $\sigma_2$ represent nonlinear activation functions, which can be the same or different. The term $W_2x \circ W_1x$ is added to the weighted sum $W_1x + b$ before passing through two activation functions. The neuron structure of DANN is illustrated in Figure $\ref{DANN}$. When $\sigma_1$ is linear, DANN degenerates to QSC. These modifications allow DANN to capture stronger nonlinearity and improve the network's ability to approximate solutions accurately compared to QRES and QSC. The experimental results presented later in the paper will provide empirical evidence regarding the advantages of DANN over the other models mentioned. The algorithm of DANN is summarized in Algorithm 1.

To provide a comprehensive comparison, we compared results obtained by several other popular models in Section 4, including the APINN, QRES, ISC, and QSC models. The network structures used in these models are summarized as follows

\noindent
\begin{minipage}[t]{0.5\textwidth}
\centering
\begin{itemize}
\small
    \item PINN: $y^{\text{PINN}} = \sigma(Wx+b)$.
    \item APINN: $y^{\text{APINN}} = \sigma(na(Wx+b))$.
    \item ISC: $y^{\text{ISC}} = \sigma(Wx +b)+x$.
\end{itemize}
\end{minipage}%
\begin{minipage}[t]{0.5\textwidth}
\centering
\begin{itemize}
\small
    \item QRES: $y^{\text{QRES}} = \sigma(W_2x \circ W_1x+ W_1x + b)$.   
    \item QSC: $y^{\text{QSC}} = W_2x \circ W_1x + \sigma(W_1x + b)$.
    \item DANN: $y^{\text{DANN}} = \sigma_1(W_2x \circ W_1x) + \sigma_2(W_1x + b)$.
\end{itemize}
\end{minipage}

\begin{table}[h]
\centering
\scalebox{0.98}{
\begin{tabular}{ p{15cm}} 
\hline
   \textbf{Algorithm 1:} DANN \\
   \hline
   \textbf{Step 1:} Generate training data using the Latin hypercube sampling method. We have
   $$
    S_{\Omega_T}=\{t_i, x_i\}_{i=1}^{N_f}, S_{\Omega_0}=\{t_j, x_j\}_{j=1}^{N_I}.   
   $$
   \textbf{Step 2:} Choose two activation functions $\sigma_1$ and  $\sigma_2$ for constructing a double-activation neural network $\mathcal{N}(t,x;\widetilde{{\theta}})$. Initialize the parameters $\widetilde{{\theta}}$ (including two sets of weights $\{W_1, W_2\}$ and biases $b$) using the Xavier method. Define the approximate solution $\hat{u}(t, x ; \widetilde{{\theta}})$.\\
    
   \textbf{Step 3:} Construct the residual neural network by substituting $\hat{u}(t, x ;\widetilde{{\theta}})$ into the governing equations using automatic differentiation.\\   
   
   \textbf{Step 4:} Define the loss function
   \begin{eqnarray*}
   \operatorname{Loss}_N(\hat{u})=\|\mathcal{G}[\hat{u}](t, x ; \widetilde{\theta})\|_{N_f, S_{\Omega_T}}^2+\|\hat{u}(t, x ; \widetilde{\theta})-g(t, x)\|_{N_I, S_{\Omega_0}}^2.
   \end{eqnarray*}
  
   \textbf{Step 5:} Minimize the loss function using the L-BFGS-B optimizer to find the optimal parameters $\boldsymbol{\widetilde{\theta}}^*$, where
   \begin{equation*}
    \boldsymbol{\widetilde{\theta}}^*=\arg\min_{\widetilde{\theta}} \operatorname{Loss}_N(\hat{u}).
   \end{equation*}
  \\
\hline
\end{tabular}}
\label{algorithm1}
\end{table}

\subsection{piecewise fitting}
When the neural network approximates the solution in the whole domain, we call it global fitting. For parabolic DPDEs with nonvanishing delay, one important characteristic is primary discontinuous points occuring in the solution domain, which will result in the decrease of convergence for the numerical method if the relevant partition doesn't take those discontinuous points as the mesh points. In order to improve the accuracy of PINN for solving parabolic DPDEs with nonvanishing delay $t-\tau(t)$, we propose a piecewise fitting method, which divides the whole domain into several subdomains by the discontinuity points. First of all, we divide interval $[0, T]$ into $\left[0, \xi_1\right],\left[\xi_1, \xi_2\right], \ldots,\left[\xi_{s-1}, \xi_s\right]$, where $\xi_0=0, \xi_{i+1}-\tau\left(\xi_{i+1}\right)=\xi_i, i=0, \ldots, s-1$, and $\xi_{s-1} \le T \le \xi_s$. The algorithm steps of piecewise fitting are as follows:

Step 1: We first approximate ($\ref{equation}$) at the first subinterval $\left[0, \xi_1\right]$. That is, we use neural network $\mathcal{N}^{(1)}\left(t, x ; \theta^{(1)}\right)$ to approximate $u(t, x)$ of the following parabolic PDEs
\begin{eqnarray}
    \begin{cases}
\partial_t u(t, x)=\alpha\Delta u(t, x)+h(t, x, g(t-\tau(t), x), u(t, x)),\quad (t, x) \in \Omega_{\xi_1}:= [0, \xi_1 ] \times \Omega, \\
u(0, x)=g(0, x),\quad x \in \Omega, \\
u(t, x)=b(t, x),\quad (t, x) \in \partial \Omega_{\xi_1}:= [0, \xi_1 ] \times \partial \Omega.
\end{cases}
\label{fenduan1equation}
\end{eqnarray}
Here, due to the fact that when $t-\tau(t) \leq \xi_1-\tau(\xi_1)=0, t \in[0, \xi_1]$, we have $u(t-\tau(t), x)=g(t-\tau(t), x)$ in \eqref{fenduan1equation} and \eqref{equation} degenerates into a classical parabolic PDEs without delay. Additionally, substitute $u(t-\tau(t), x)$ by $g(t-\tau(t), x)$ and the initial condition can be substitute by $u(0, x)=g(0, x)$.
Define the approximate solution of $u(t, x)$ in ($\ref{fenduan1equation}$) as $\hat{u}^{(1)}(t, x ; \theta^{(1)})$ 
$$
\hat{u}^{(1)}(t, x ; \theta^{(1)})=b(t, x)+B(x) \mathcal{N}^{(1)}(t, x ; \theta^{(1)}).
$$
The loss function is defined as
$$
\begin{aligned}
& \operatorname{Loss}^{(1)}(\hat{u}(t, x ; \theta^{(1)})) \\
&=\|\partial_t \hat{u}^{(1)}(t, x ; \theta^{(1)})-\alpha\Delta \hat{u}^{(1)}(t, x ; \theta^{(1)})-h(t, x, g(t-\tau(t), x), \hat{u}^{(1)}(t, x ; \theta^{(1)}))\|_{2, \Omega_{\xi_1}}^2 \\
& \quad +\|\hat{u}^{(1)}(0, x ; \theta^{(1)})-g(0, x)\|_{2, \Omega}^2,
\end{aligned}
$$
where $B(x)$ already guarantees that the approximate solution is accurate at the boundary condition. Similar with  global fitting, we first discretize the $\operatorname{Loss}^{(1)}(\hat{u}(t, x ; \theta^{(1)}))$ through the training data to obtain the empirical $\operatorname{loss} \operatorname{Loss}_N^{(1)}(\hat{u}(t, x ; \theta^{(1)}))$. Then, L-BFGS-B optimization algorithm is used to minimize the loss function to obtain the optimized parameters $\boldsymbol{\theta}_*^{(1)}$, where 

$$
\boldsymbol{\theta}_*^{(1)}=\underset{\theta}{\operatorname{argmin}} \operatorname{Loss}_N^{(1)}(\hat{u}(t, x ; \theta^{(1)})) \text {. }
$$

The approximate solution of equation ($\ref{equation}$) on $[0, \xi_1] \times \Omega$ can be expressed as
$
\hat{u}^{(1)}(t, x ; \boldsymbol{\theta}_*^{(1)}) \text {. }
$

Step 2: Suppose that the approximate solution $\hat{u}^{(i)}(t, x ; \boldsymbol{\theta}_*^{(i)})$ of equation ($\ref{equation}$) in subinterval $t \in[\xi_{i-1}, \xi_i], i=1,2, \ldots, k, k<s$ has been obtained. In order to obtain the approximate solution in subdomain $[\xi_k, \xi_{k+1}] \times \Omega$, take the approximate solution $\hat{u}^{(k)}(t, x ; \boldsymbol{\theta}_*^{(k)})$ obtained in the previous subinterval as the initial function on the $k^{\text {th }}$ subdomain, then the equation can be expressed as
$$
\left\{\begin{array}{l}
\partial_t u(t, x)=\alpha\Delta u(t, x)+h(t, x, \hat{u}^{(k)}(t-\tau(t), x ; \boldsymbol{\theta}_*^{(k)}), u(t, x)), \quad(t, x) \in[\xi_k, \xi_{k+1}] \times \Omega, \\
u(\xi_k, x)=\hat{u}^{(k)}(\xi_k, x ; \boldsymbol{\theta}_*^{(k)}), \quad x \in \Omega, \\
u(t, x)=b(t,x),\quad (t, x) \in \partial \Omega_{\xi_k}:=[\xi_k, \xi_{k+1}] \times \partial \Omega .
\end{array}\right.
$$
Similar with Step 1, we can obtain the approximate solution $\hat{u}^{(k+1)}(t, x ; \boldsymbol{\theta}_*^{(k+1)})$ of equation ($\ref{equation}$) on $[\xi_k, \xi_{k+1}] \times \Omega$.

Repeating the algorithm of Step 2 until the approximate solution $\hat{u}^{(s)}(t, x ; \boldsymbol{\theta}_*^{(s)})$ of the last subinterval $[\xi_{s-1}, \xi_s]$ is obtained.

\section{Convergence theorem of the neural network solution}
In this section, we give the main results to illustrate the approximation ability of the neural network method proposed in Section 2. In the convergence theory, to apply conviently the universal approximation theorem \cite{hornik1989multilayer, hornik1990universal} of neural networks, we denote the neural network $\mathfrak{S}^{r}$ to be a class of neural networks with a single hidden layer and $r$ hidden units.
\subsection{Main results}
Theorem 3.1 shows that under reasonable assumptions, the loss function $\mathrm{Loss}(\hat{u})$
tends to zero.
\begin{thm}\label{副定理}
Assume

(\romannumeral1) the solution of equation ($\ref{equation}$) is unique and bounded on domain $\Omega_T$,

(\romannumeral2) the initial function $g(t,x)$ of equation ($\ref{equation}$) is bounded on domain ${\Omega_0}$,

(\romannumeral3) the nonlinear terms $h(t,x,u,v)$ is locally Lipschitz continuous in $(u, v)$ with Lipschitz constant
that can have at most polynomial growth in $u$ and $v$, uniformly with respect to $t$, $x$. This means that
$$
\left|h(t, x, u, v)-h(t, x, w, s)\right| \le (\left|u \right|^{q_1/2}+\left|v \right|^{q_2/2}+\left|w \right|^{q_3/2}+\left|s \right|^{q_4/2})(\left|u-w \right|+\left|v-s \right|),
$$
for some constants $0  \le q_1, q_2, q_3, q_4 < \infty$. 

Then, for any $\epsilon	\ge 0$, there exists a positive constant $C\ge 0$ that depend on $\sup_{\Omega_T}\left| u(t,x) \right|$ and $\sup_{\Omega_T}\left| u(t-\tau(t,u(t,x)),x) \right|$ such that there exists a neural network $\hat{u}\in\mathfrak{S}^{r}$ that satisfies
		$$\mathrm{Loss}(\hat{u}) \leq C\epsilon.$$	
 	\end{thm}
\subsection{Preliminaries}
The results stated in the following lemma plays a key role in the proof of Theorem 3.1. In the following theorem, $\mathbb{C}(U)$ denotes the set of continuous functions. $\mathbb{C}^{1,2}\left(\Omega_T\right)$ is the space of all functions that are continuously differentiable for $t$ of first order and continuously differentiable for $x$ of second order.

\begin{lemma}
    For any $a>0$ and $b>0$, the inequality $(a+b)^q \leq C(a^q+b^q)$ holds for $q \geq 0$, where $C$ is a constant.
\end{lemma}
\begin{proof}
    For $a>0$ and $b>0, a+b \leq 2 \max \{a, b\}$, so we have $(a+b)^q \leq(2 a)^q+(2 b)^q$ for $q \geq 0$. Take $C=2^q$, we have $(a+b)^q \leq C(a^q+b^q)$ for $q \geq 0$.
This completes the proof of Lemma 3.1.
\end{proof}
\begin{lemma}( \cite{hornik1991approximation}) Let $\mathbb{C}^m(U)=\left\{u \in \mathbb{C}(U)~|~ \forall |\beta| \leq m, D^\beta u \in \mathbb{C}(U)\right\}$, for $\forall u \in \mathbb{C}^m(U), \forall \epsilon>0$, there exists a single hidden layer feedforward neural network $\mathcal{N}(x ; \theta) \in \mathfrak{S}^r$, such that
$$
\max _{|\beta| \leq m} \sup _{x \in J}|D^\beta u(x)-D^\beta(\mathcal{N}(x ; \theta))|<\epsilon,
$$
where $J$ is a compact set in $U, x=(x_1, \cdots, x_d),|\beta|=\beta_1+\cdots+\beta_d$, and
$$
D^\beta=D_x^\beta=\frac{\partial^{|\beta|}}{\partial x_1^{\beta_1} \cdots \partial x_d^{\beta_d}}.
$$
\end{lemma}

\subsection{Proof of Theorem 3.1}
\begin{proof}
We take $U:=\Omega_T$. By Lemma 3.2 , there is a single hidden layer feedforward neural network $\hat{u}(t, x ; \theta) \in \mathfrak{S}^r$ such that
$$
\sup _{(t, x) \in \Omega_T}\left|\partial_t u(t, x)-\partial_t \hat{u}(t, x ; \theta)\right|+\max _{|\beta| \leq 2} \sup _{(t, x) \in \bar{\Omega}_T}\big|\partial_x^{(\beta)} u(t, x)-\partial_x^{(\beta)} \hat{u}(t, x ; \theta)\big|<\epsilon,
$$
where $\hat{u} \in \mathbb{C}^{1,2}\left(\Omega_T\right)$, and $\beta$ represents the order of differentiation of $u(t, x)$ with respect to $x$. Next, we prove that $\operatorname{Loss}(\hat{u}) \leq C \epsilon$. By locally Lipschitz continuous condition, Lemma 3.1 and Lemma 3.2, we have
$$
\begin{aligned}
& \iint_{\Omega_T}|h(t, x, \hat{u}(t-\tau(t, \hat{u}(t, x ; \theta)), x ; \theta), \hat{u}(t, x ; \theta))-h(t, x, u(t-\tau(t, u(t, x)), x), u(t, x))|^2 d t d x \\
& \leq \iint_{\Omega_T}\big(|\hat{u}(t-\tau(t, \hat{u}(t, x ; \theta)), x ; \theta)|^{\frac{q_1}{2}}+|\hat{u}(t, x ; \theta)|^{\frac{q_2}{2}}+|u(t-\tau(t, u(t, x)), x)|^{\frac{q_3}{2}}+|u(t, x)|^{\frac{q_4}{2}}\big)^2\\
& \quad (|\hat{u}(t-\tau(t, \hat{u}(t, x ; \theta)), x ; \theta)-u(t-\tau(t, u(t, x)), x)|+|\hat{u}(t, x ; \theta)-u(t, x)|)^2 d t d x \\
&\leq  8 \iint_{\Omega_T}\big(|\hat{u}(t-\tau(t, \hat{u}(t, x ; \theta)), x ; \theta)|^{q_1}+|\hat{u}(t, x ; \theta)|^{q_2}+|u(t-\tau(t, u(t, x)), x)|^{q_3}+|u(t, x)|^{q_4}\big) \\
&  \quad \big(|\hat{u}(t-\tau(t, \hat{u}(t, x ; \theta)), x ; \theta)-u(t-\tau(t, u(t, x)), x)|^2+|\hat{u}(t, x ; \theta)-u(t, x)|^2\big) d t d x \\
&\leq  C \iint_{\Omega_T}\big(|\hat{u}(t-\tau(t, \hat{u}(t, x ; \theta)), x ; \theta)-u(t-\tau(t, u(t, x)), x)|^{q_1}+|u(t-\tau(t, u(t, x)), x)|^{q_1}  \\
& \quad  +|\hat{u}(t, x ; \theta)-u(t, x)|^{q_2}+|u(t, x)|^{q_2}+|u(t-\tau(t, u(t, x)), x)|^{q_3}+|u(t, x)|^{q_4}\big) \\
&  \quad\big(|\hat{u}(t-\tau(t, \hat{u}(t, x ; \theta)), x ; \theta)-u(t-\tau(t, u(t, x)), x)|^2+|\hat{u}(t, x ; \theta)-u(t, x)|^2\big) d t d x \\
&\leq  C\big(\epsilon^{q_1}+|u(t-\tau(t, u(t, x)), x)|^{q_1}+\epsilon^{q_2}+|u(t, x)|^{q_2}+|u(t-\tau(t, u(t, x)), x)|^{q_3}+|u(t, x)|^{q_4}\big) \epsilon^2,
\end{aligned}
$$
where $C$ is a constant.

Because $u$ is the solution of equation ($\ref{equation}$), $\mathcal{G}[u](t, x)=0$. By Lemma 3.2 we have
$$
\begin{aligned}
&\operatorname{Loss}(\hat{u}) \\
&=  \|\mathcal{G}[\hat{u}](t, x ; \theta)\|_{\Omega_T}^2+\|\hat{u}(t, x ; \theta)-g(t, x)\|_{\Omega_0}^2 \\
&=  \|\mathcal{G}[\hat{u}](t, x ; \theta)-\mathcal{G}[u](t, x)\|_{\Omega_T}^2+\|\hat{u}(t, x ; \theta)-g(t, x)\|_{\Omega_0}^2 \\
&\leq  3 \iint_{\Omega_T}\big|\alpha\partial_t \hat{u}(t, x ; \theta)-\alpha\partial_t u(t, x)\big|^2 d t d x+3 \iint_{\Omega_T}|\Delta \hat{u}(t, x ; \theta)-\Delta u(t, x)|^2 d t d x \\
& \quad +3 \iint_{\Omega_T}|h(t, x, \hat{u}(t-\tau(t, \hat{u}(t, x ; \theta)), x ; \theta), \hat{u}(t, x ; \theta))-h(t, x, u(t-\tau(t, u(t, x)), x), u(t, x))|^2 d t d x\\
&\quad +\iint_{\Omega_0}|\hat{u}(t, x ; \theta)-g(t, x)|^2 d t d x \\
& \leq  C \epsilon^2 .
\end{aligned}
$$
The above $C$ represents different values in different places. 
This completes the proof of Theorem 3.1.
\end{proof}

\begin{remark}
    Theorem 3.1 provides the convergence of the loss function for PINN based on global fitting. For parabolic DPDEs with primary discontinuous points, the domain can be decomposed at the discontinuous points, and the convergence of the loss function for PINN based on piecewise fitting can be analyzed in each subdomain.
\end{remark} 
\begin{remark}
The convergence conclusion of the loss function for PINN can be extended to deep neural network $(r\ge2)$ which has stronger fitting performance.
\end{remark}  
\section{Numerical results}
In this section, we present the numerical results of our study, where we employ DANN and several variants of PINN, including APINN, QRES, ISC, and QSC, to solve DDEs with state-dependent delay and parabolic DPDEs. For parabolic DPDEs, our investigation focuses on equations with primary discontinuities, and we explore two distinct fitting strategies: global fitting and piecewise fitting. The hyperparameters for each example are presented in Table $\ref{Hyperparameters}$, and the scale factor $n$ in APINN is set to 5. All computations were conducted on a workstation equipped with a NVIDIA GeForce RTX 3080 GPU, utilizing Tensorflow 1.15. The accuracy of the approximation is measured by the relative $L^2$ error and the absolute error, where the exact solution $u(t_i, x_j)$ and the neural network approximate solution $\hat{u}(t_i, x_j ; \boldsymbol{\widetilde{\theta}}^*)$ are obtained at the data points $(t_i,x_j)$.

The calculation formula of the absolute error of $\hat{u}(t, x)$ at $(t_i, x_j)$ is
$$
|\hat{u}(t_i ; x_j ; \boldsymbol{\widetilde{\theta}}^*)-u(t_i, x_j)|,
$$
and the calculation formula of the relative $L^2$ error is
$$
\frac{\sqrt{\sum_{j}\sum_{i}|\hat{u}(t_i, x_j ; \boldsymbol{\widetilde{\theta}}^*)-u(t_i, x_j)|^2}}{\sqrt{\sum_{j}\sum_{i}|u(t_i, x_j)|^2}} .
$$

\begin{table*}[t]
    \centering
    \caption{Hyperparameters for the 4 Examples tested in this study.}
    \scalebox{0.85}{
    \begin{tabular}{lcccccrrr}
    \toprule
        \textbf{Example} & \textbf{Depth} & \textbf{Width} & \textbf{$N_f$} & \textbf{$N_I$} & \textbf{$\sigma(\sigma_1) + \sigma_2$} & \textbf{Optimizer} & \textbf{Iterations}  \\ \midrule
        Example 1 & 3 & 8 & 400 & 150 &softplus+tanh &L-BFGS-B   &20 000   \\ 
        Example 2 & 3 & 6 & 300 & 100  &sigmoid+tanh&L-BFGS-B   &20 000   \\
        Example 3 & 3 & 6 &  800 & 500  &sigmoid+tanh&L-BFGS-B   &10 000   \\
        Example 4 (d=3)&  4 & 15 &  6000 & 2000  &tanh+tanh&L-BFGS-B   &50 000   \\
        Example 4 (d=8) & 4 & 35 &  8000 & 3000  &tanh+tanh &L-BFGS-B   &50 000   \\
    \bottomrule
    \end{tabular}}
\label{Hyperparameters}
\end{table*}

\begin{exm}
    We first consider the Mackey-Glass-type equation as
$$
u_t=u_{x x}-a u(t, x)+\frac{b u(t-\tau, x)}{1+u^m(t-\tau, x)},
$$
which is used to simulate the change of the number of a single species population with time. Take $a=2, b=$ $1, m=2, \tau=1$, and the homogeneous Dirichlet boundary condition is selected as the boundary condition, we have
\begin{eqnarray}
\begin{cases}
u_t=u_{x x}-2 u+\frac{u(t-1, x)}{1+u^2(t-1, x)}+f(t, x),\quad (t, x) \in[0,2] \times[0, \pi], \\
u(t, x)=e^{-t} \sin x, \quad(t, x) \in[-1,0] \times[0, \pi], \\
u(t, 0)=u(t, \pi)=0,\quad  t \in[0,2],
 \end{cases}
\end{eqnarray}
where we choose an appropriate $f(t, x)$ so that the exact solution is
$$
u(t, x)=e^{-t} \sin x,\quad (t, x) \in[0,2] \times[0, \pi] .
$$
\end{exm}

In this study, we focus on a problem without primary discontinuous points and test the fitting performance of different models using global fitting approach. In order to guarantee the accurate satisfaction of the boundary conditions, we represent the approximate solution in the following form:
$$\hat{u}(t, x)=x(x-\pi)\mathcal{N}(t,x;\theta),$$
where $\mathcal{N}(t,x;\theta)$ represents a neural network with parameter $\theta$.

\begin{table*}[t]
    \centering
    \caption{The relative $L^2$ errors of $u$ obtained by different models.}
    \label{tab:univ-compa}
    \begin{tabular}{ccccccc}
    \toprule
        \textbf{Model} & \textbf{PINN} & \textbf{APINN} & \textbf{QRES} & \textbf{ISC} & \textbf{QSC}  & \textbf{DANN} \\ \midrule
        \textbf{Error $u$} & 5.44e-04 & 1.24e-03 &8.08e-04 &1.07e-03   &3.56e-04 & 1.26e-04 \\
    \bottomrule
    \end{tabular}
    \label{delay1table}
\end{table*}

\begin{figure}[H]
	\centering	
	\includegraphics[width=0.46\linewidth]{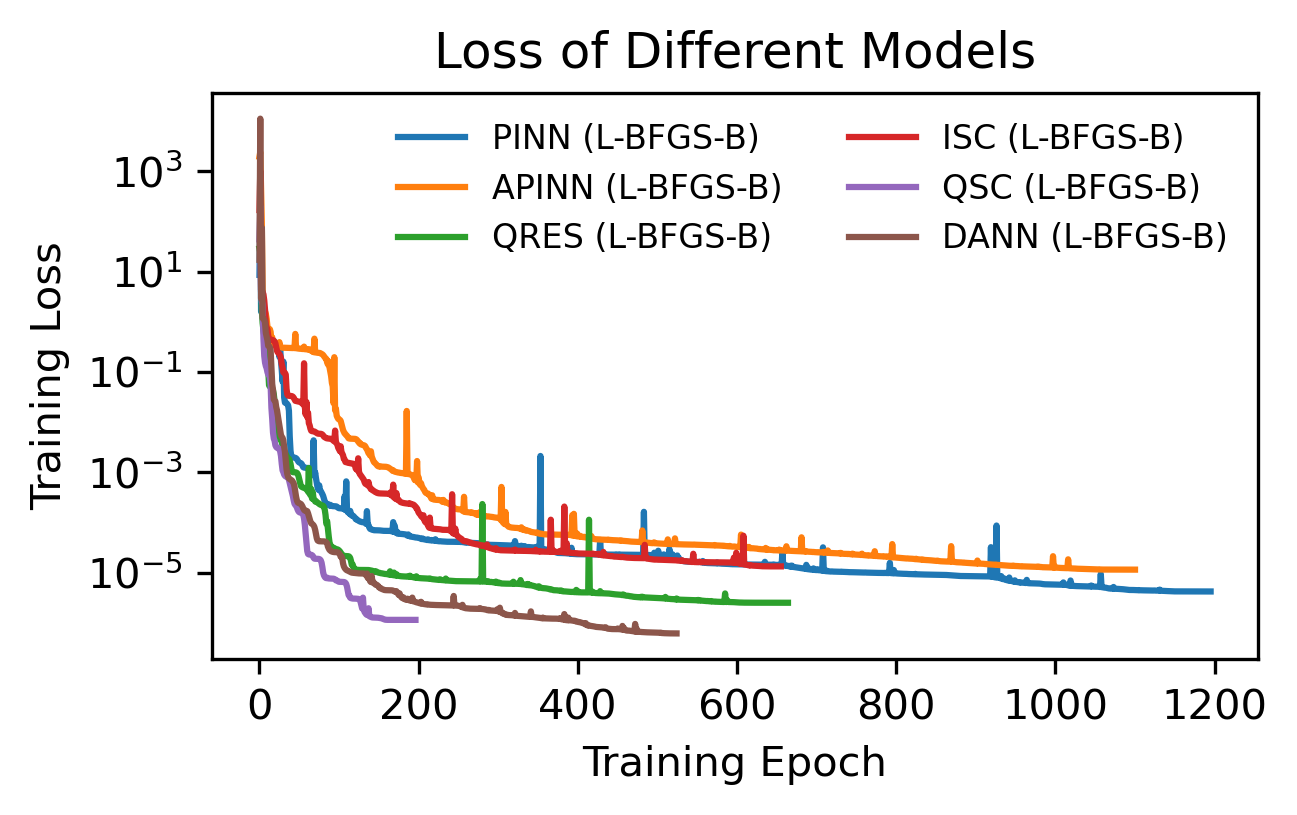}\\
	\caption{Training loss curves for solving Example 1.}
 \label{delay1}
\end{figure}

\begin{figure}[H]
	\centering
 \begin{minipage}{0.36\linewidth}
		\centering
		\includegraphics[width=\linewidth]{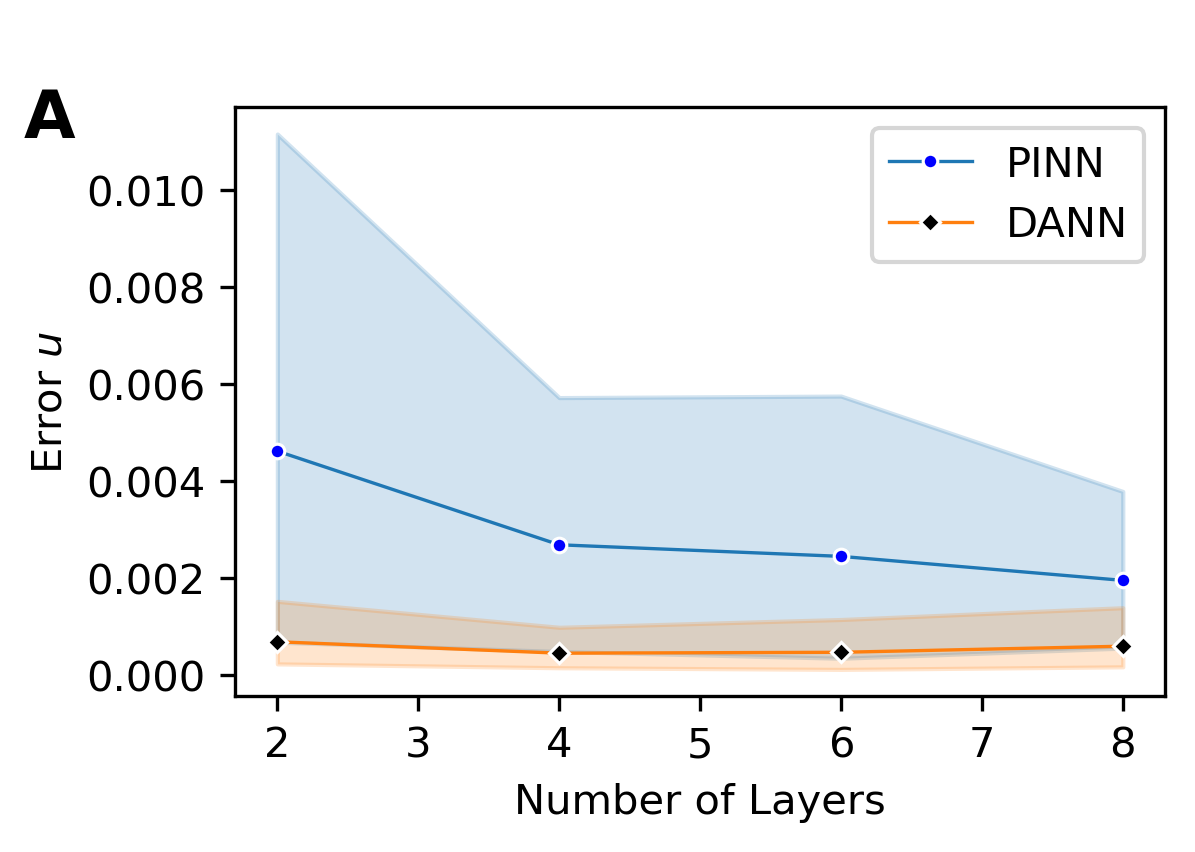}
		
	\end{minipage}
 \begin{minipage}{0.36\linewidth}
		\centering
		\includegraphics[width=\linewidth]{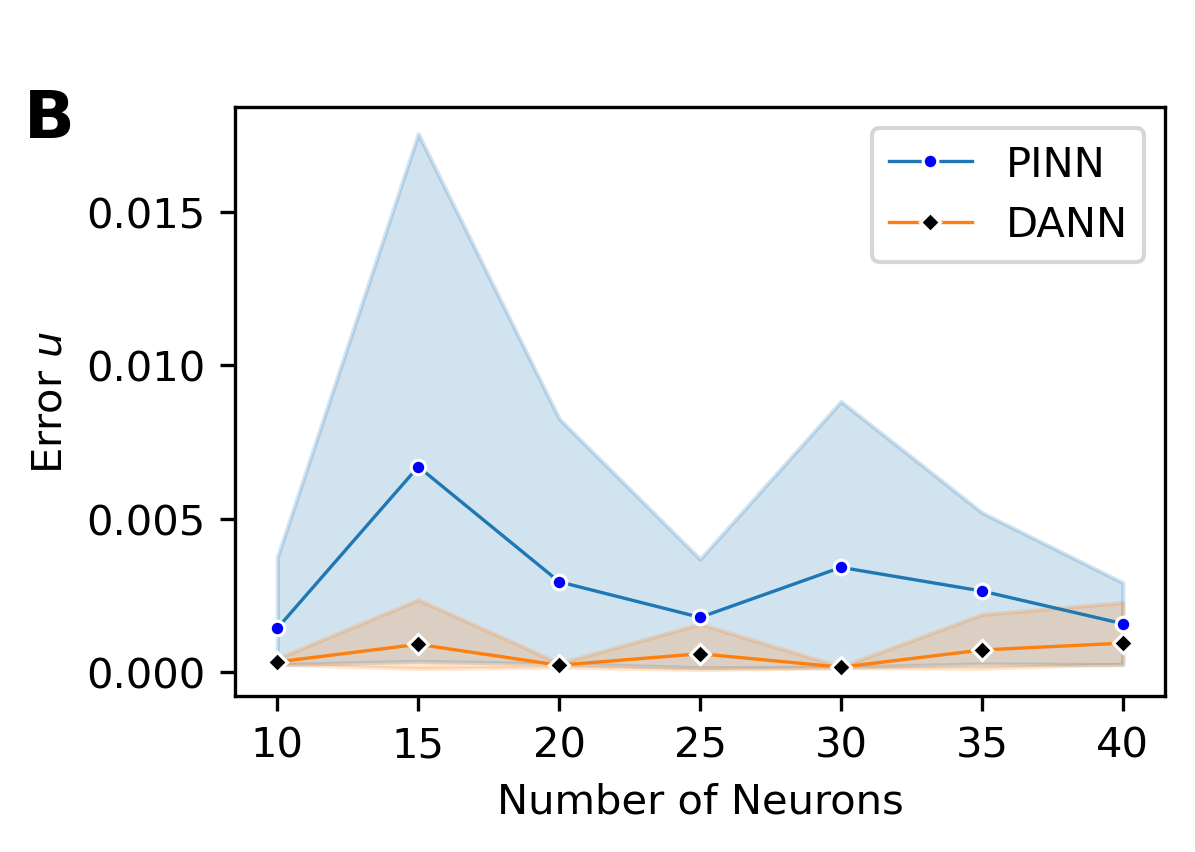}
		
	\end{minipage}

	\begin{minipage}{0.36\linewidth}
		\centering
		\includegraphics[width=\linewidth]{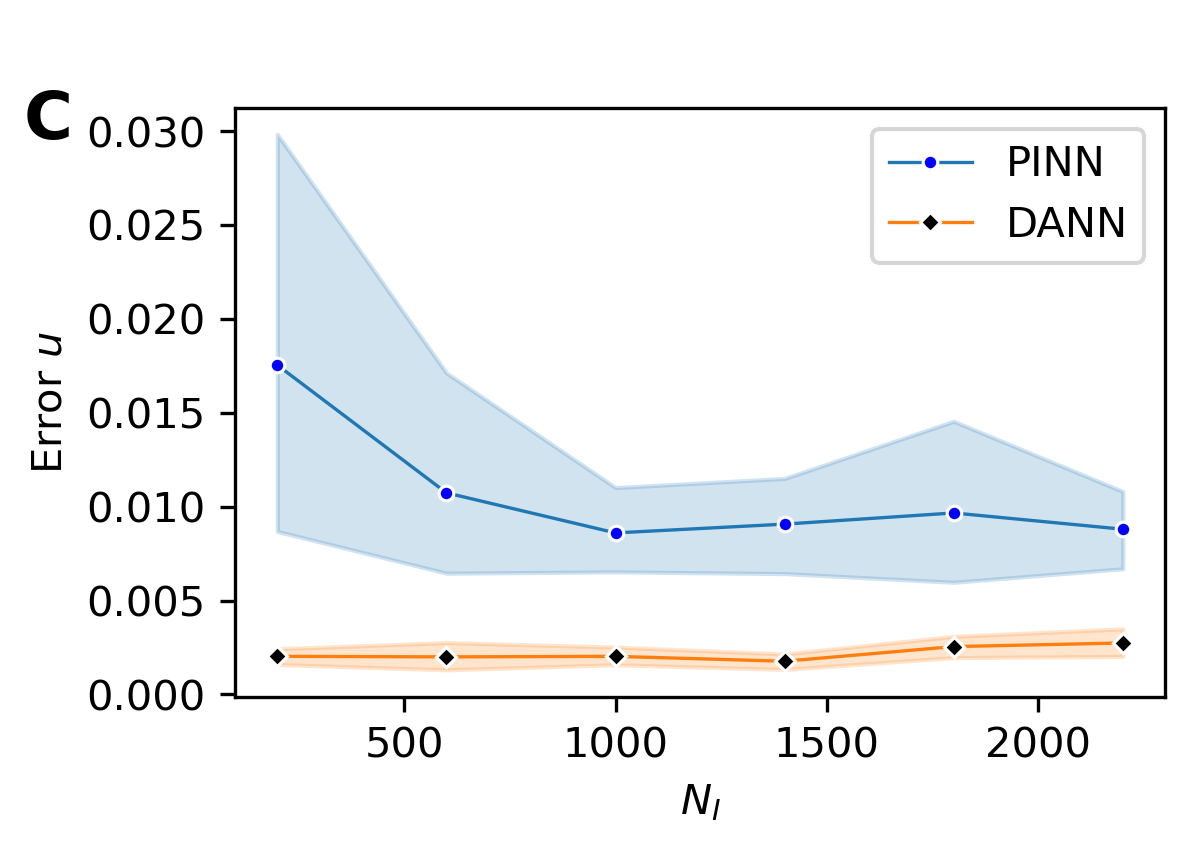}
		
	\end{minipage}
 \begin{minipage}{0.36\linewidth}
		\centering
		\includegraphics[width=\linewidth]{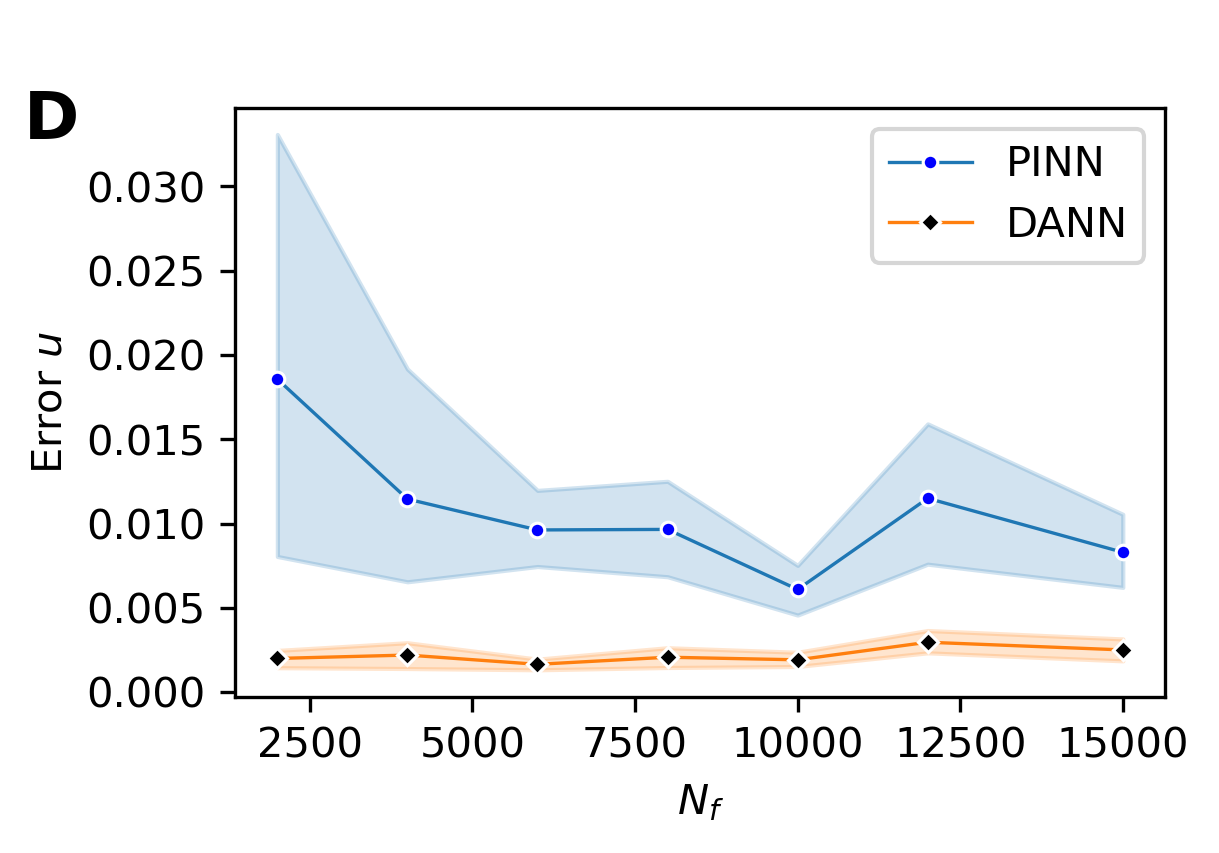}
		
	\end{minipage}
	\caption{(A) and (B): The figures show the relative $L^2$ errors of $u$ for 28 different architectures across four depth values (2, 4, 6, 8) and seven widths (10, 15, 20, 15, 30, 35, 40) for both PINN and DANN. (C) and (D): The figures show the relative $L^2$ errors of $u$ achieved by varying numbers of $N_I$ (200, 600, 1000, 1400, 1800, 2200) and $N_f$ (2000, 4000, 6000, 8000, 10000, 12000, 15000) for PINN and DANN. Lines indicate mean values and shades represent variances.}
	\vspace{-0.2cm}
     \label{delay2xiaolv}
\end{figure}

To demonstrate the superiority of DANN, Table $\ref{delay1table}$ presents the relative $L^2$ errors of $u$ obtained by different models. The results presented in Table $\ref{delay1table}$ demonstrate that, in this case without primary discontinuity points, DANN exhibits superior accuracy compared to other models when all models use the same parameter settings. In terms of convergence speed, Figure $\ref{delay1}$ shows that DANN is faster than PINN, APINN, QRES and ISC. While the convergence speed of QSC is slightly faster than that of DANN, the approximate solution obtained by DANN has higher accuracy.

To further explore the advantages of DANN over PINN, we conduct experiments on different network widths, depths, and varying numbers of training points $N_f$ and $N_I$ to address this issue. Activation function $\sigma$, $\sigma_1$ and $\sigma_2$ all set to tanh. Figures $\ref{delay2xiaolv}$ (A) and (B) illustrate how the relative $L^2$ errors vary with different network widths and depths, with $N_I=2200$ and $N_f=15000$. From Figures $\ref{delay2xiaolv}$ (A) and (B), we can conclude that DANN achieves higher accuracy than PINN with fewer network layers and neurons. Figures $\ref{delay2xiaolv}$ (C) and (D) demonstrate how the relative $L^2$ errors change with varying numbers of training points $N_f$ and $N_I$, with a network depth of 8 layers and 40 neurons per layer. From Figures $\ref{delay2xiaolv}$ (C) and (D), we observe that DANN achieves higher accuracy with fewer training points compared to PINN. These results highlight the significant advantage of DANN in terms of parameter efficiency.

\begin{exm}
Consider the following parabolic DPDE with nonvanishing delay
\begin{eqnarray}
        \begin{cases}
u_t=u_{x x}+u+u(t^3-3 t^2+3 t-1, x),\quad(t, x) \in[0,2] \times[0, \pi], \\
u(t, x)=t \sin x,\quad(t, x) \in[-1,0] \times[0, \pi], \\
u(t, 0)=u(t, \pi)=0,\quad t \in[0,1] .
     \end{cases}
		\end{eqnarray}
  The exact solution is
\begin{eqnarray}
u(t,x) =
        \begin{cases}
        
\left(\frac{1}{4} t^4-t^3+\frac{3}{2} t^2-t\right) \sin x,\quad (t, x) \in[0,1] \times[0, \pi], \\
\left(\frac{(t-1)^{13}}{52}-\frac{(t-1)^{10}}{10}+\frac{3}{14}(t-1)^7-\frac{(t-1)^4}{4}-\frac{1}{4}\right) \sin x,\quad (t, x) \in[1,2] \times[0, \pi] .
     \end{cases}
		\end{eqnarray}
\end{exm}

This problem involves two primary discontinuity points at $t=0$ and $t=1$. To ensure the accurate satisfaction of the boundary conditions, we employ an approximate solution in the following form
$$\hat{u}(t, x)=x(x-\pi)\mathcal{N}(t,x;\theta),$$
where $\mathcal{N}(t,x;\theta)$ represents a neural network with parameter $\theta$.

First, we use PINN, APINN, QRES, ISC, QSC, and DANN based on global fitting to solve this problem. The exact solution and the approximate solutions obtained by different models at $t=0.5,1.0,1.5$ are shown in Figure $\ref{delay2nihe}$. We can see that the approximate solution obtained by DANN is closest to the exact solution.

Next, we employ piecewise fitting approach to solve this problem. We divide $(t, x) \in[0,2] \times[0, \pi]$ into two subdomains, $(t, x) \in[0,1] \times[0, \pi]$ (Subdomain 1) and $(t, x) \in[1,2] \times[0, \pi]$ (Subdomain 2). Figure $\ref{delay2jie}$ presents the approximate solutions obtained  by different models based on global and piecewise fitting. It can be observed from Figure $\ref{delay2jie}$ that, regardless of the model chosen for global fitting among the six models, the approximate solution consistently displays noticeable discrepancies compared to the exact solution. However, when piecewise fitting is performed with any of the models, the resemblance between the approximate and exact solutions significantly improves. This observation leads to the conclusion that, in cases where the exact solution contains primary discontinuity points, piecewise fitting approach yields more accurate results.

Table $\ref{delay2table}$ provides the relative $L^2$ errors obtained by different models based on global and piecewise fitting. In global fitting, the approximate solutions accuracy obtained by these six models ranks from high to low as follows: DANN $>$ APINN $>$ QRES $>$ QSC $>$ ISC $>$ PINN, which means DANN has
\begin{figure}[H]
    \centering
    \subfigure{\label{Fig:num1a}
        \includegraphics[width=12cm,height=3cm]{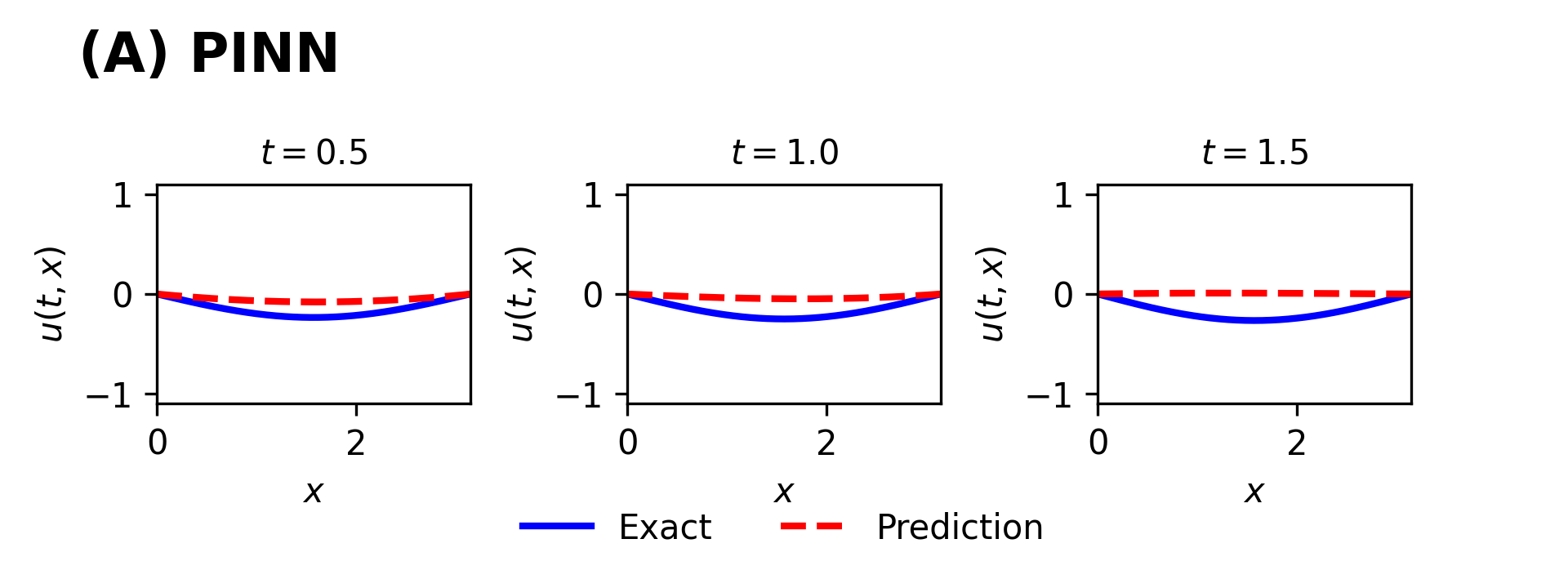}
    }

    \subfigure{\label{Fig:num1b}
        \includegraphics[width=12cm,height=3cm]{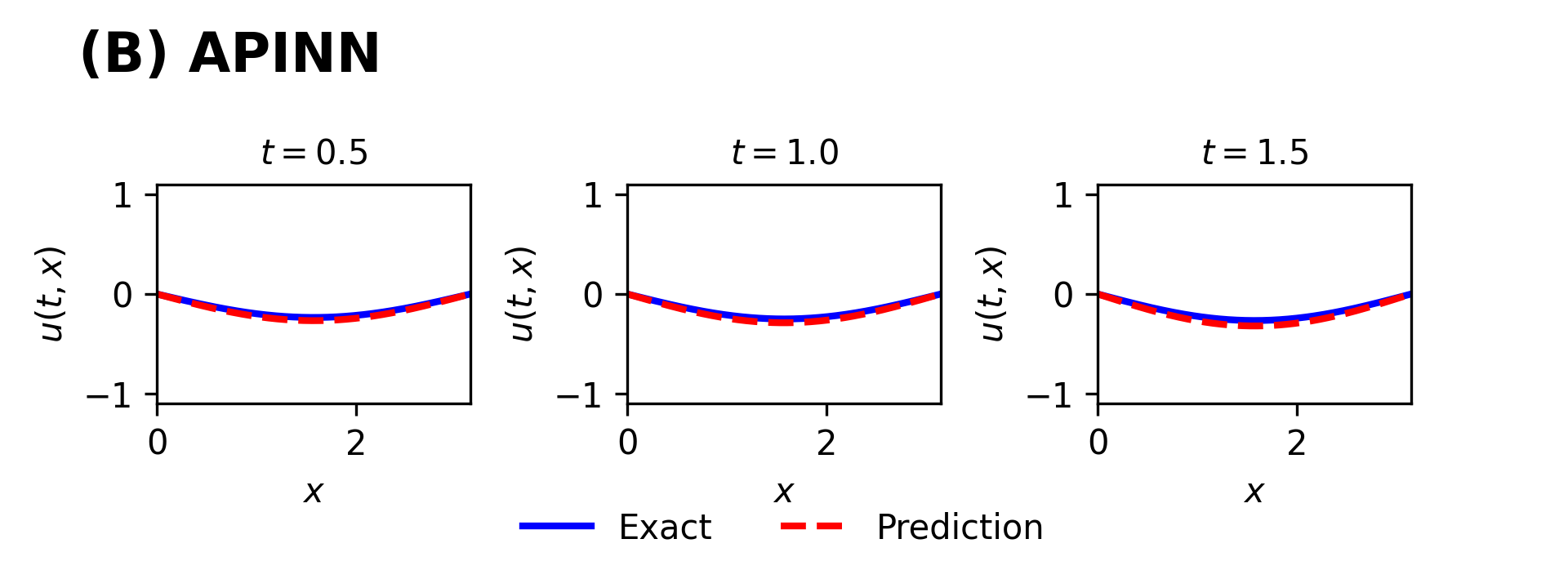}
    }

    \subfigure{\label{Fig:num1c}
        \includegraphics[width=12cm,height=3cm]{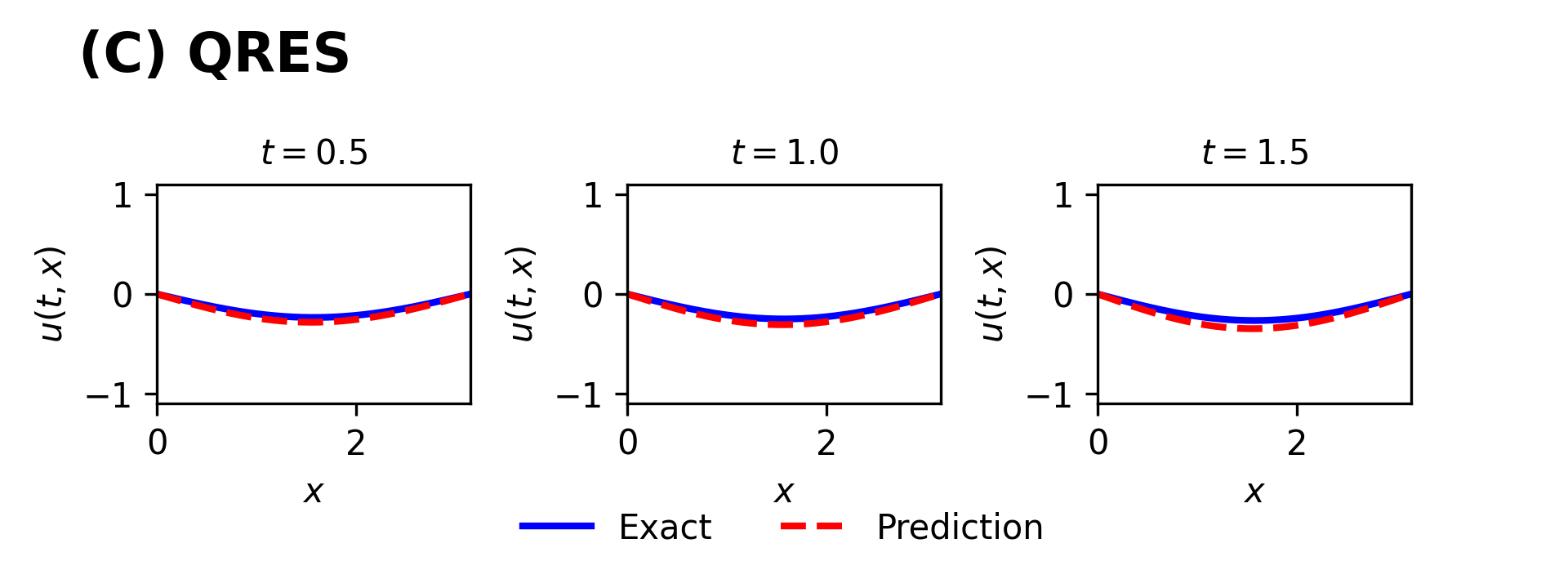}
    }

    \subfigure{\label{Fig:num1d}
        \includegraphics[width=12cm,height=3cm]{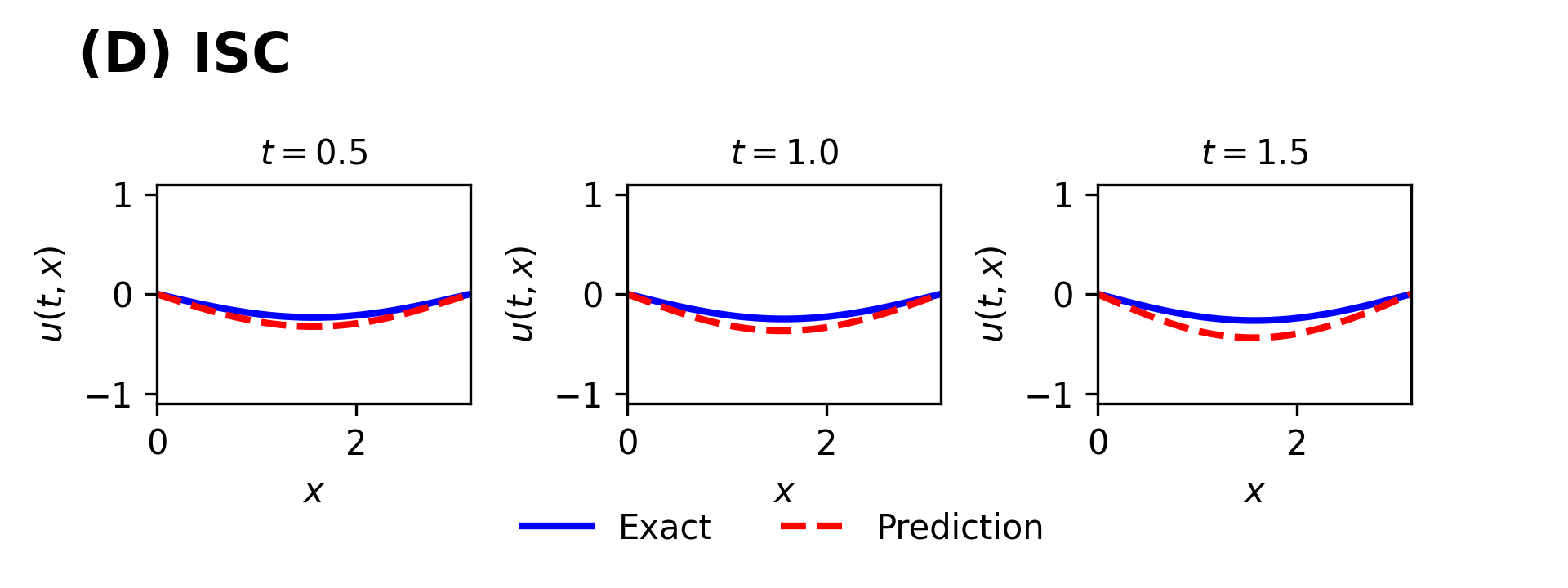}                
    }

    \subfigure{\label{Fig:num1d}
        \includegraphics[width=12cm,height=3cm]{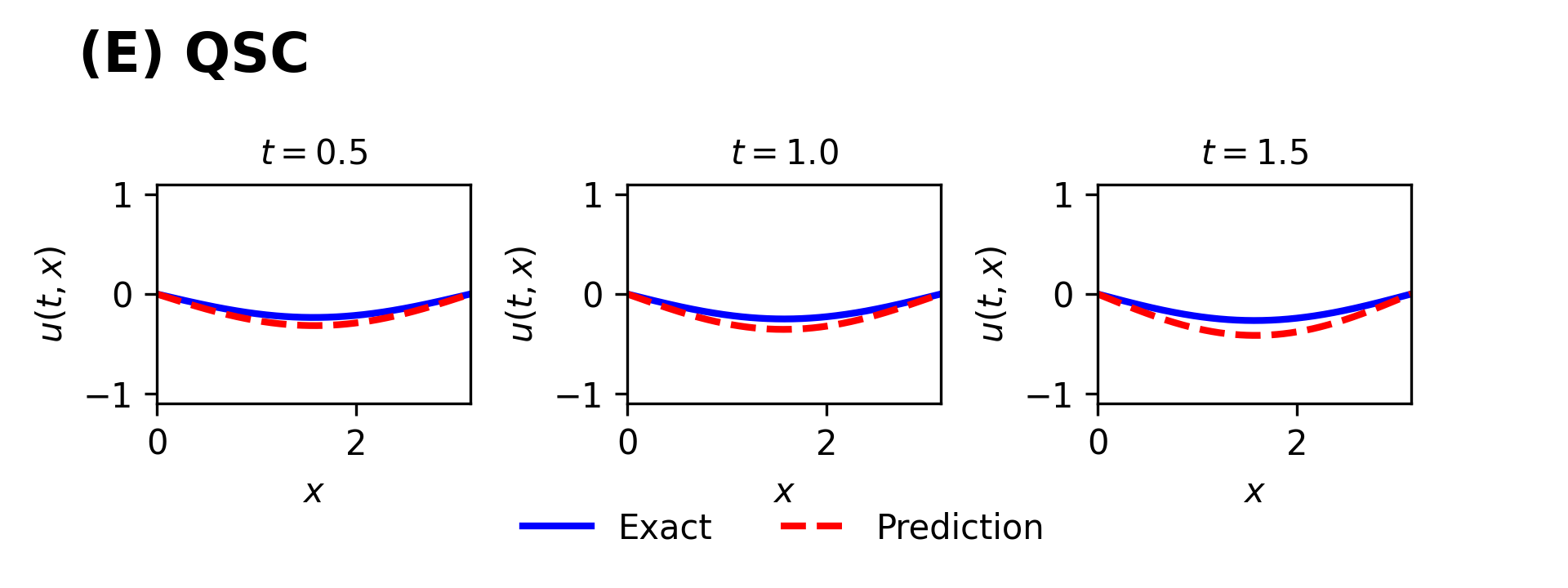}                
    }

    \subfigure{\label{Fig:num1d}
        \includegraphics[width=12cm,height=3cm]{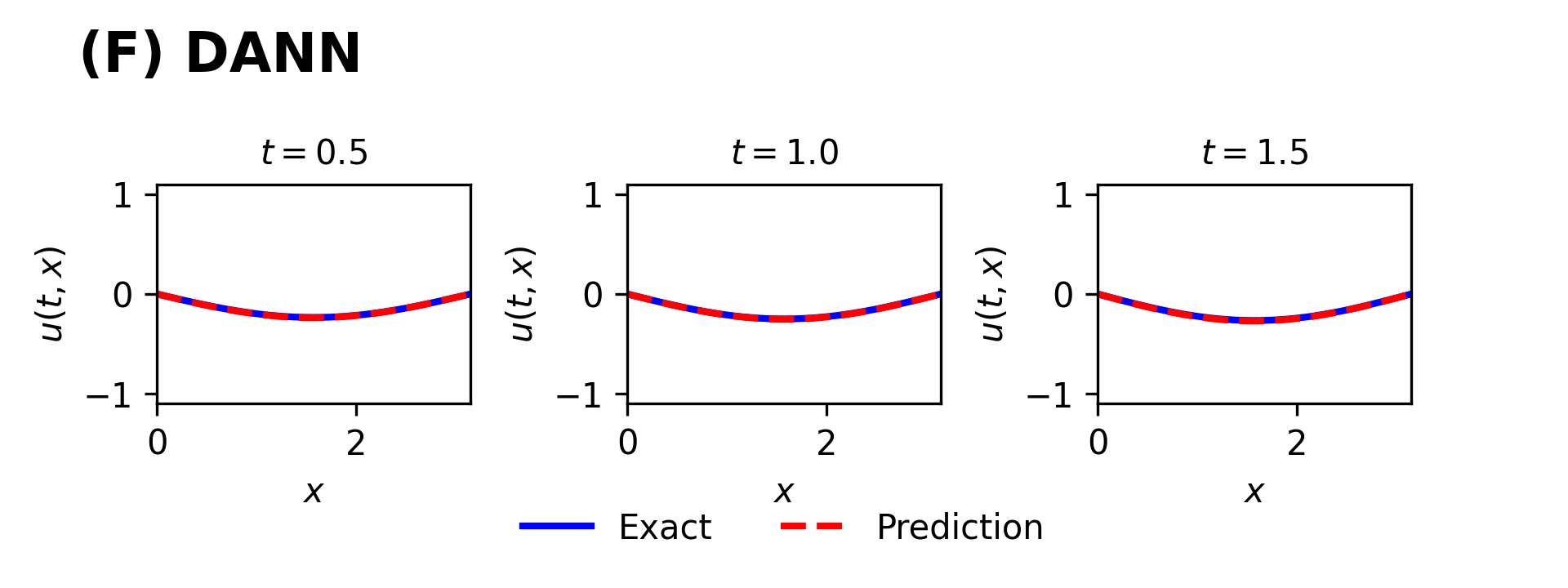}                
    }
    \caption{The exact solution (blue) and the approximate solutions (red) obtained by different models at $t=0.5, 1, 1.5$.}
    \label{delay2nihe}
\end{figure}
\begin{figure}[H]
	\centering
	\begin{minipage}{0.32\linewidth}
		\centering
		\includegraphics[width=0.9\linewidth]{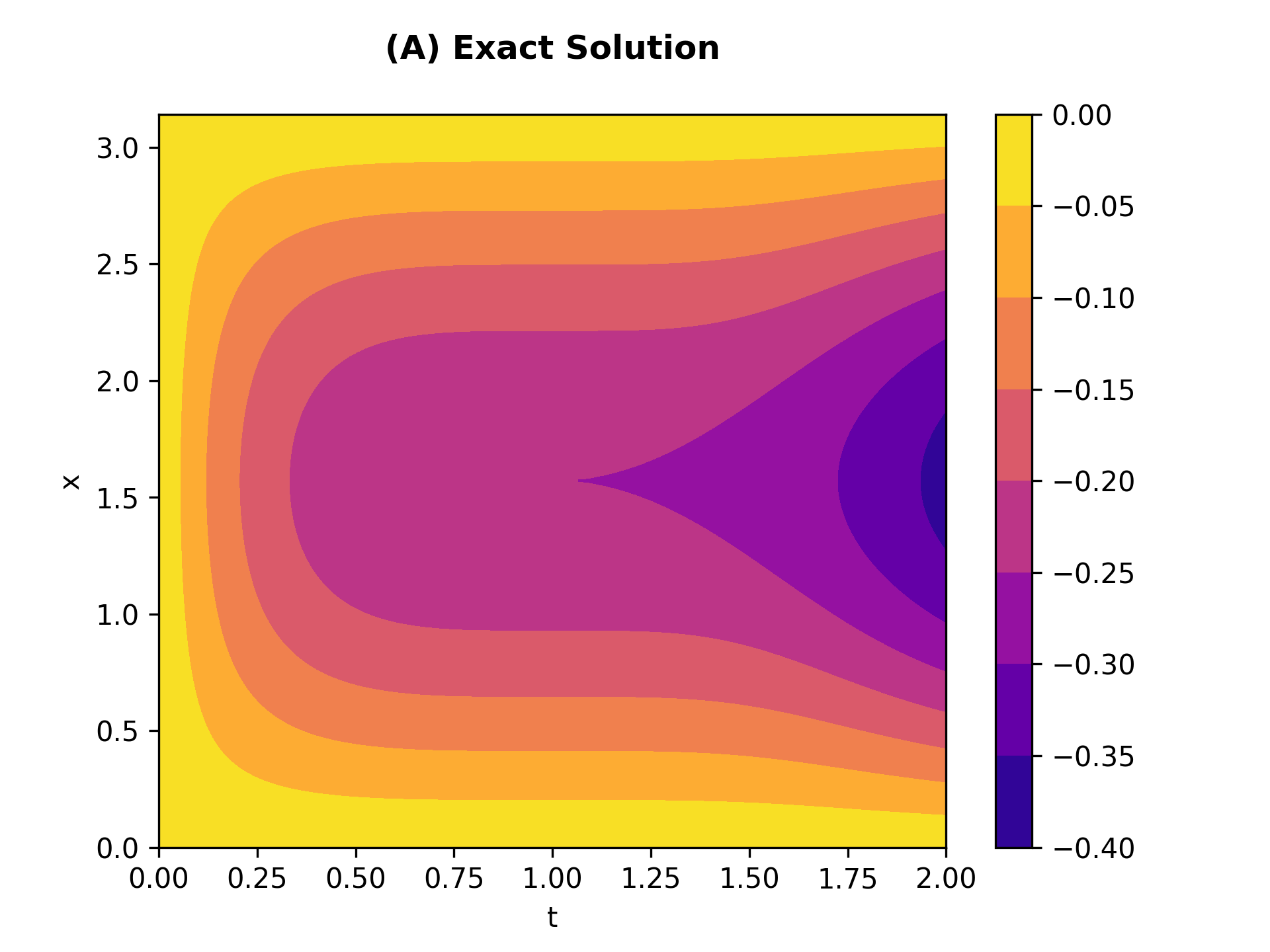}
		
		\label{chutian1}
	\end{minipage}
	
	\begin{minipage}{0.24\linewidth}
		\centering
		\includegraphics[width=\linewidth]{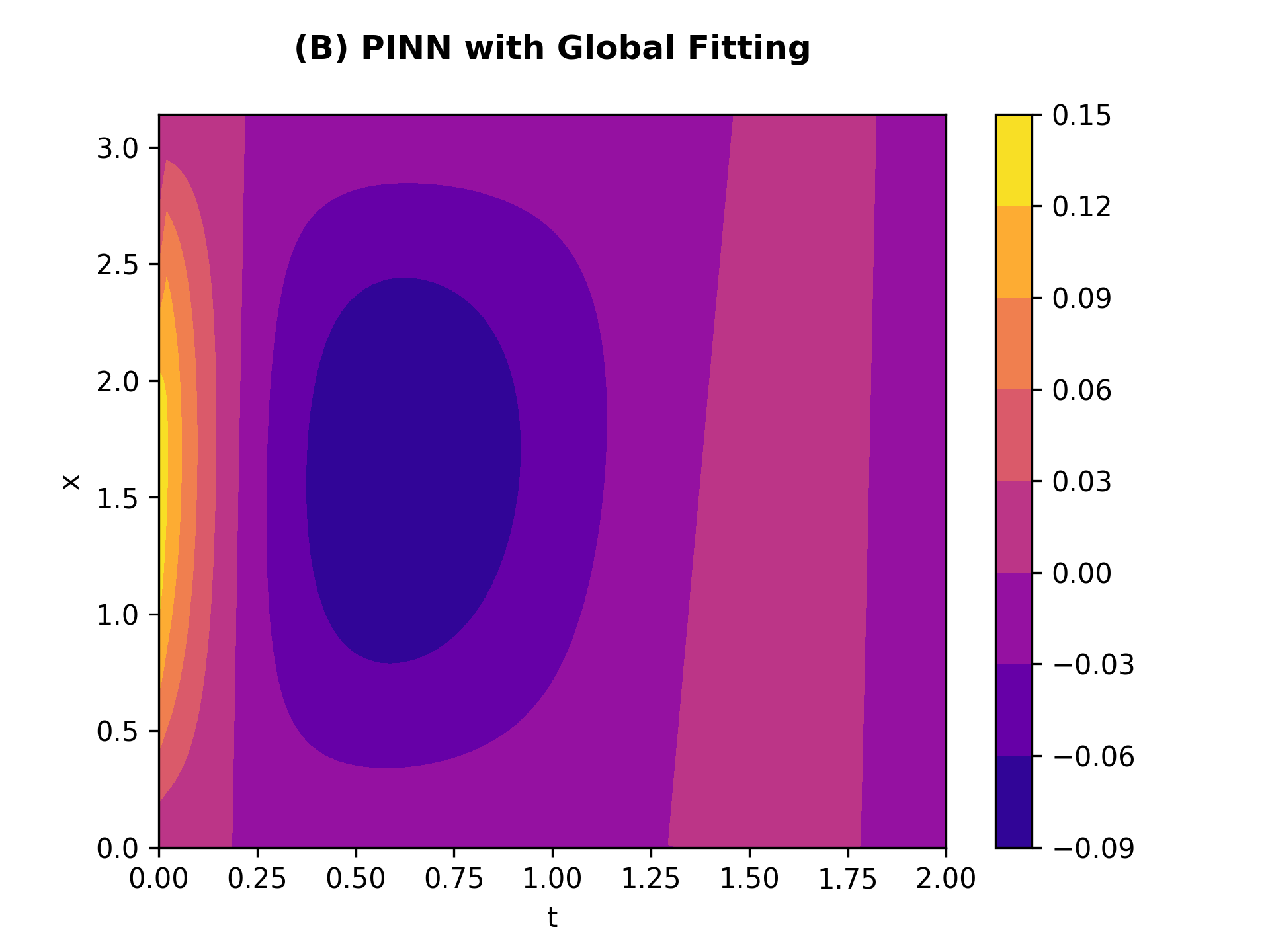}
		
		\label{chutian3}
	\end{minipage}
	\begin{minipage}{0.24\linewidth}
		\centering
		\includegraphics[width=\linewidth]{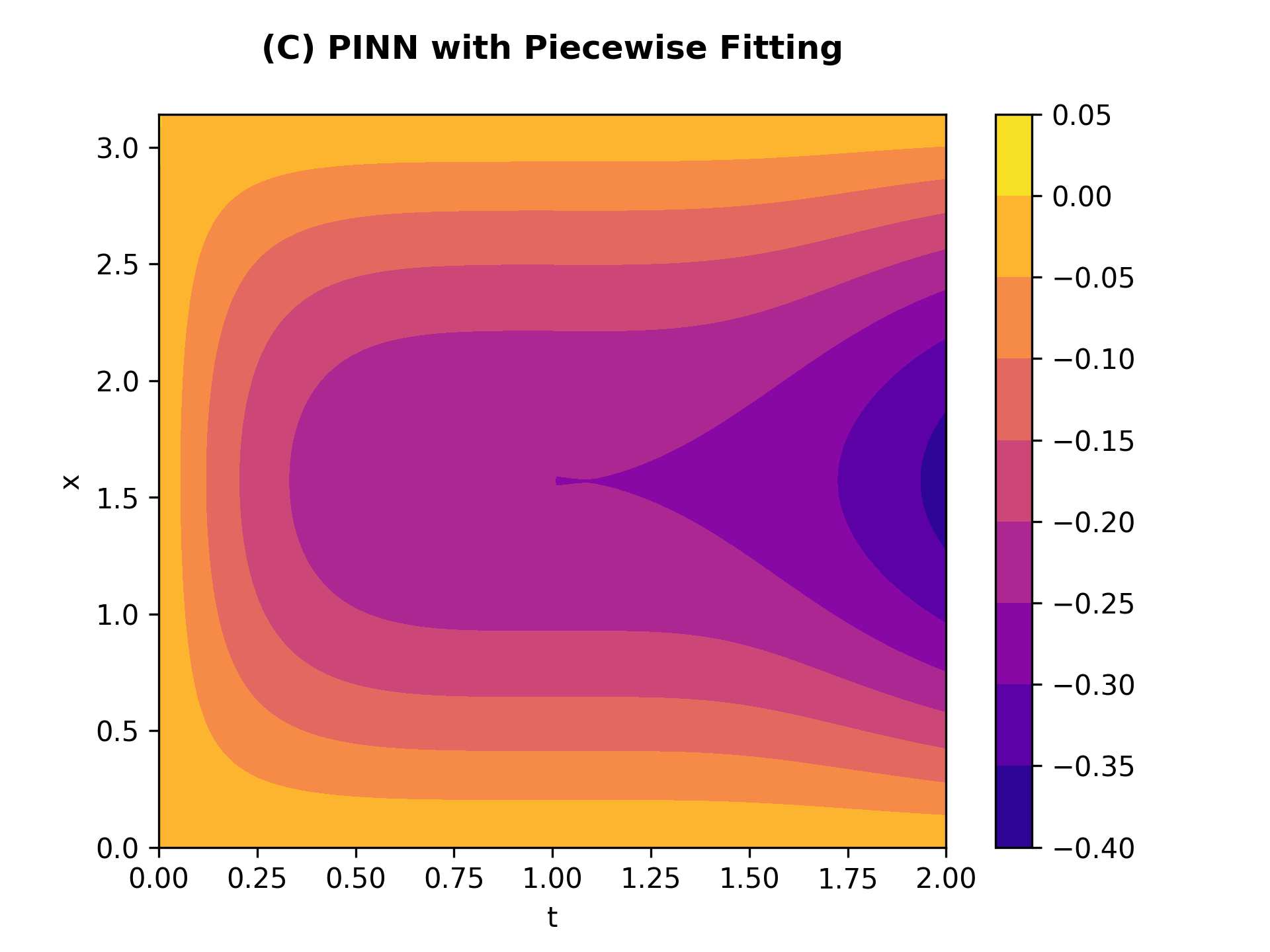}
		
		\label{chutian4}
	\end{minipage}
 \begin{minipage}{0.24\linewidth}
		\centering
		\includegraphics[width=\linewidth]{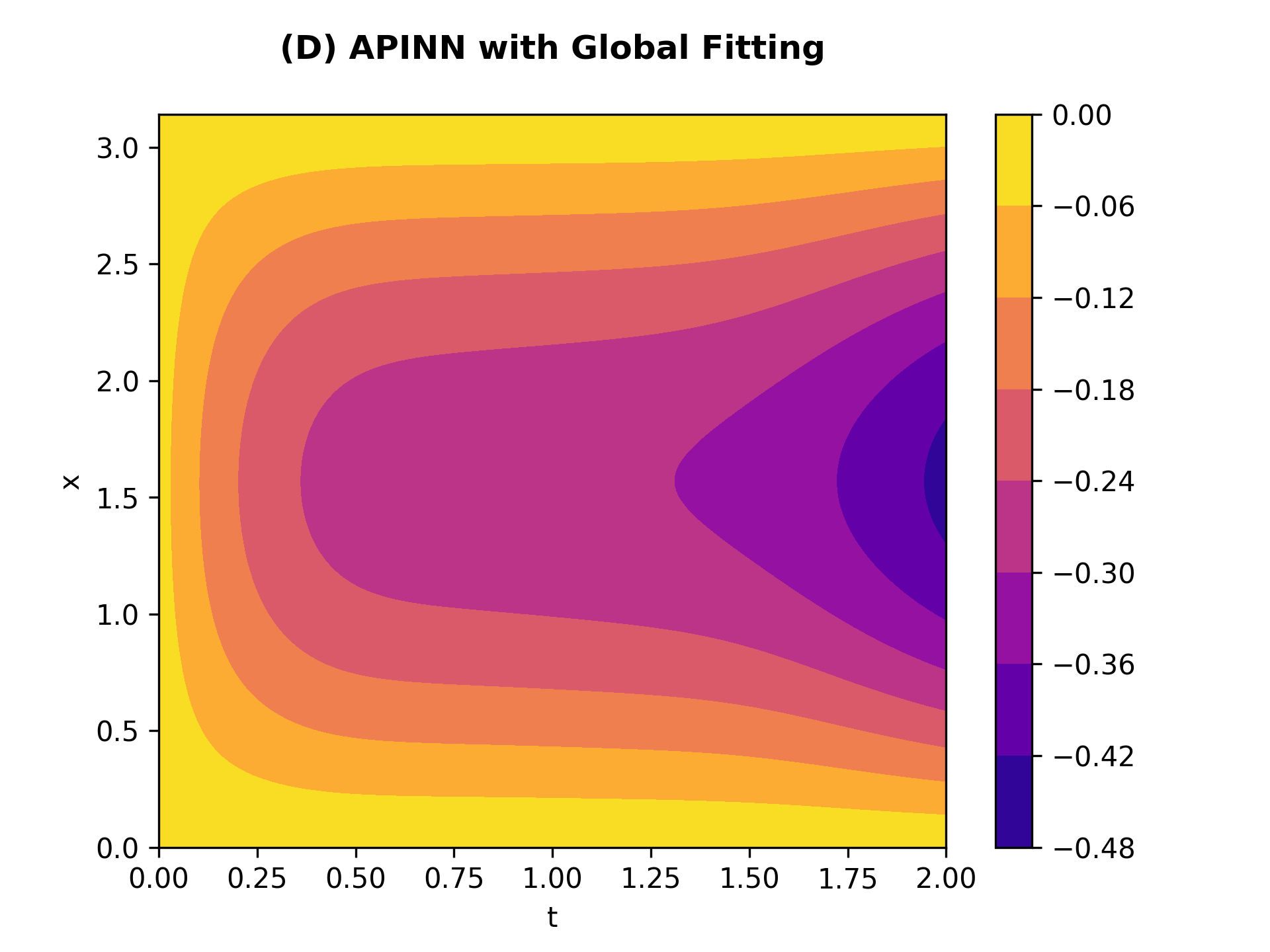}
		
		\label{chutian2}
	\end{minipage}
 \begin{minipage}{0.24\linewidth}
		\centering
		\includegraphics[width=\linewidth]{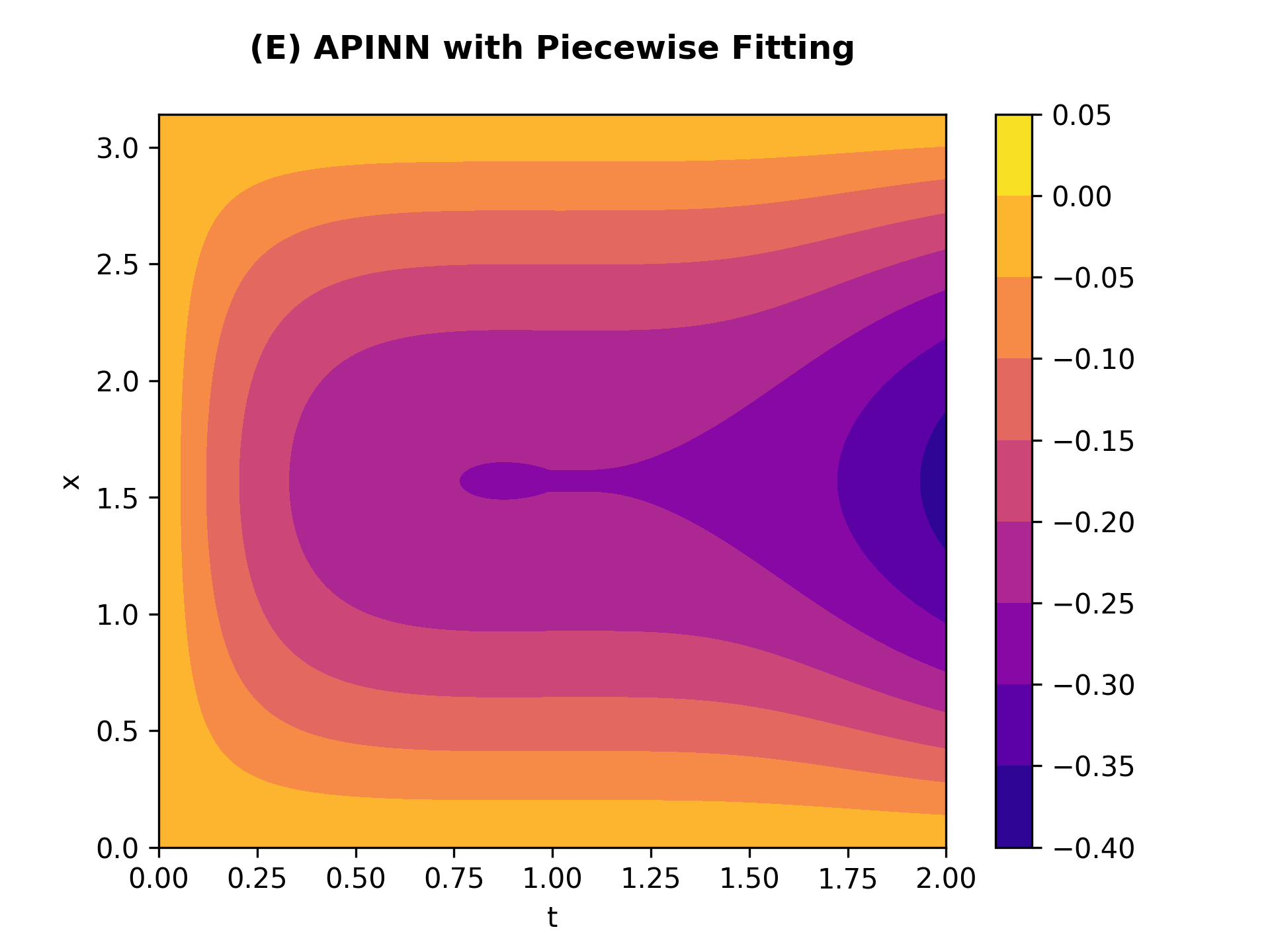}
		
		\label{chutian2}
	\end{minipage}

 \begin{minipage}{0.24\linewidth}
		\centering
		\includegraphics[width=\linewidth]{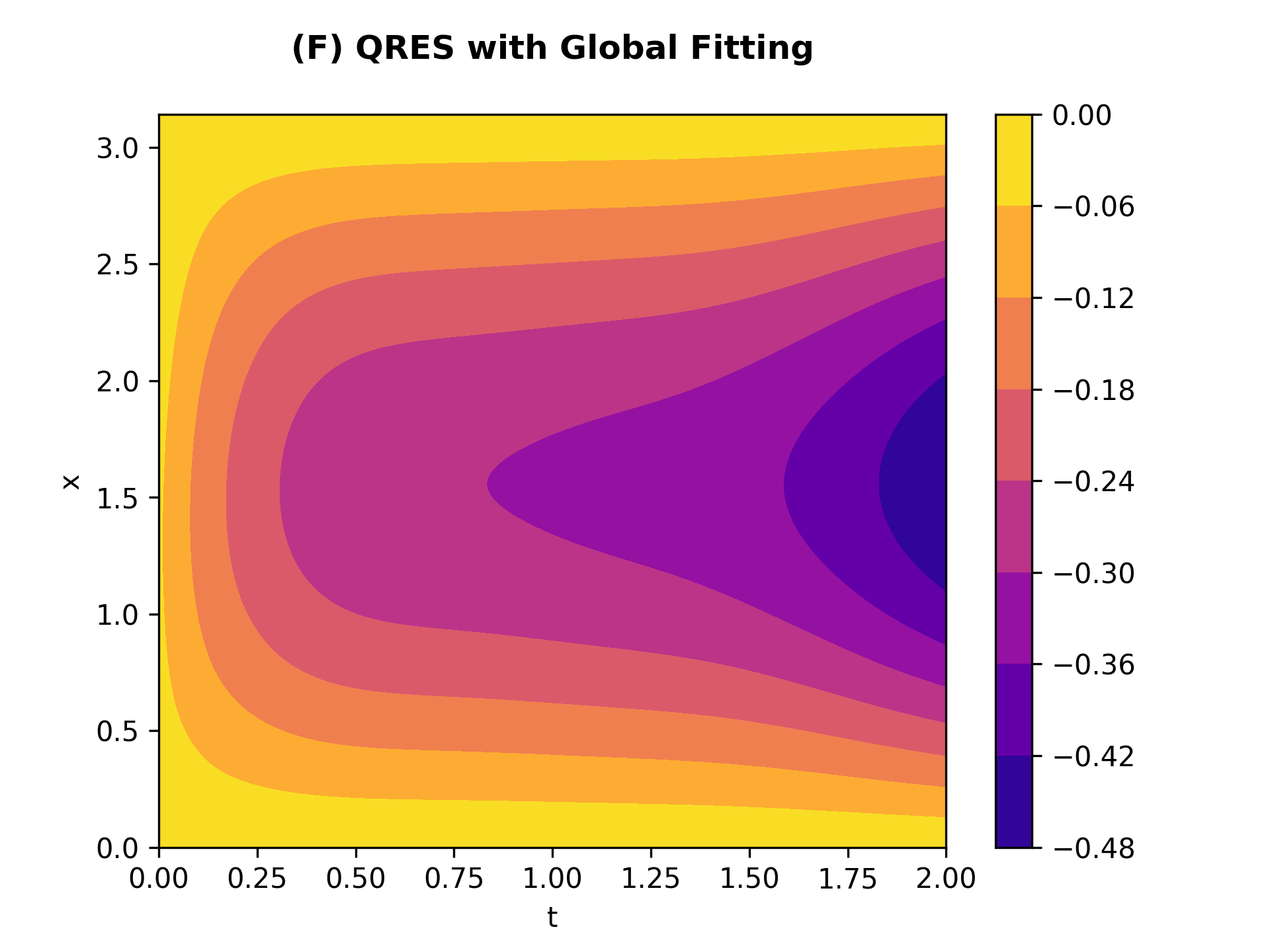}
		
		\label{chutian3}
	\end{minipage}
	\begin{minipage}{0.24\linewidth}
		\centering
		\includegraphics[width=\linewidth]{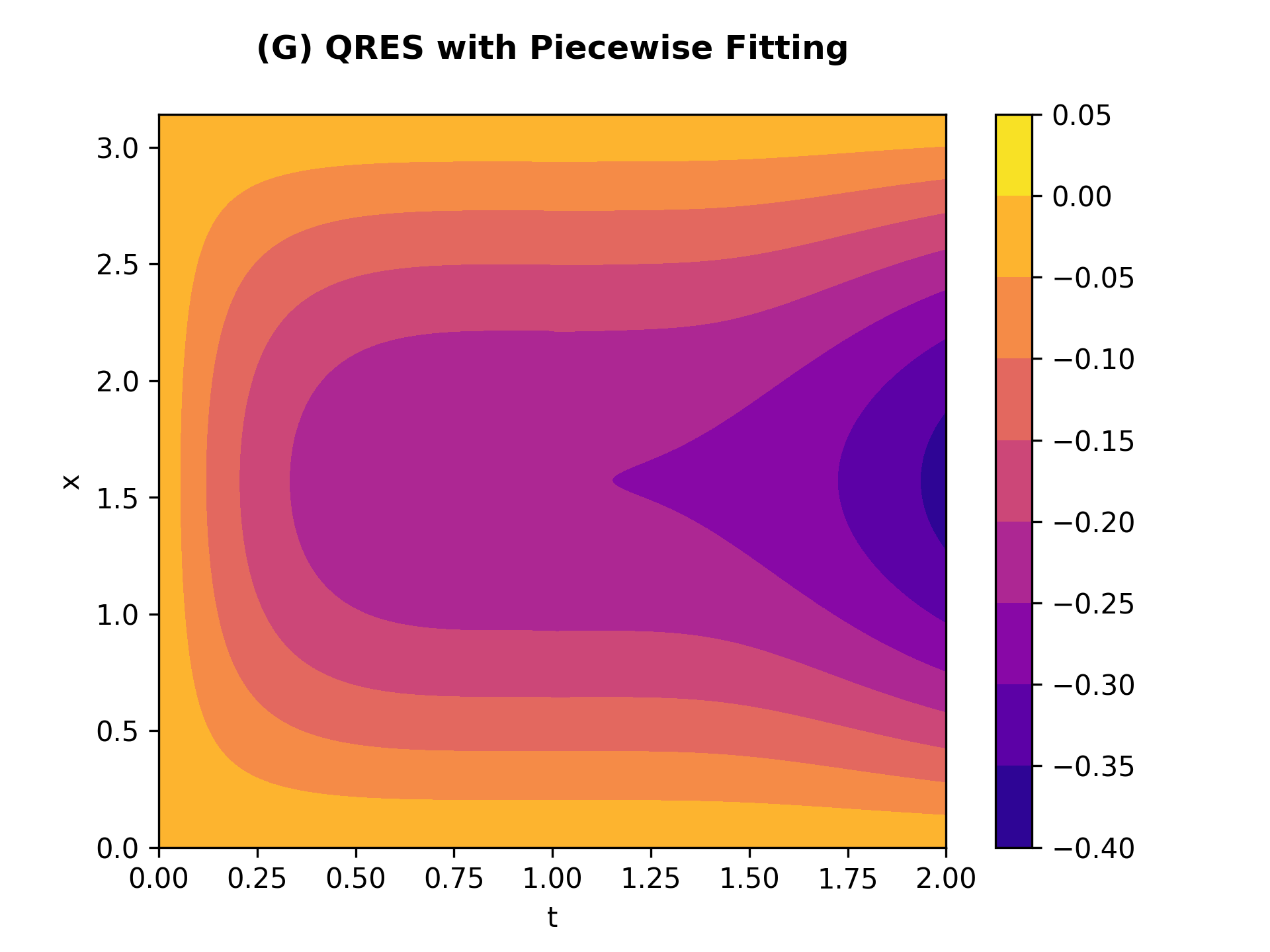}
		
		\label{chutian4}
	\end{minipage}
 \begin{minipage}{0.24\linewidth}
		\centering
		\includegraphics[width=\linewidth]{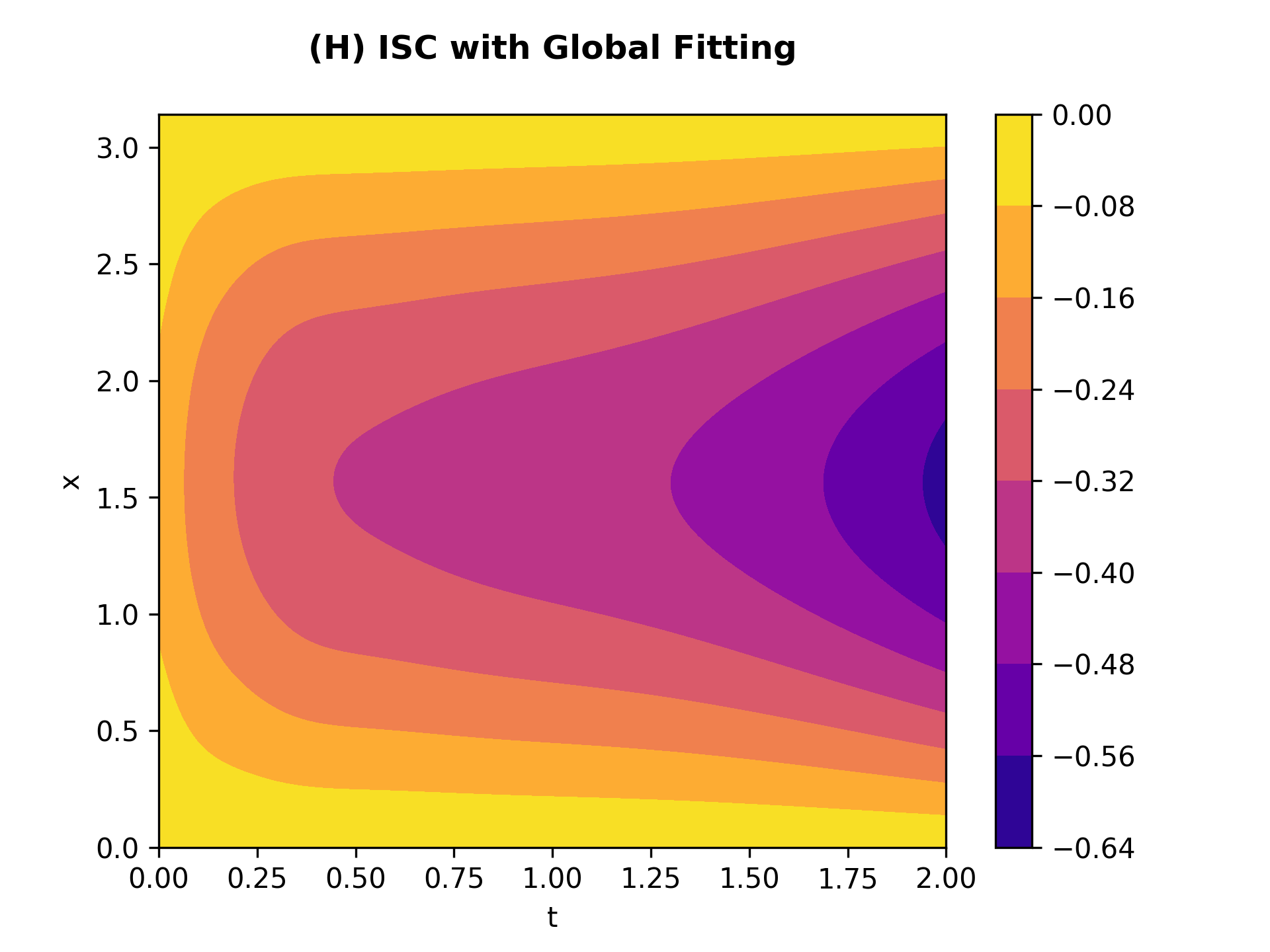}
		
		\label{chutian2}
	\end{minipage}
 \begin{minipage}{0.24\linewidth}
		\centering
		\includegraphics[width=\linewidth]{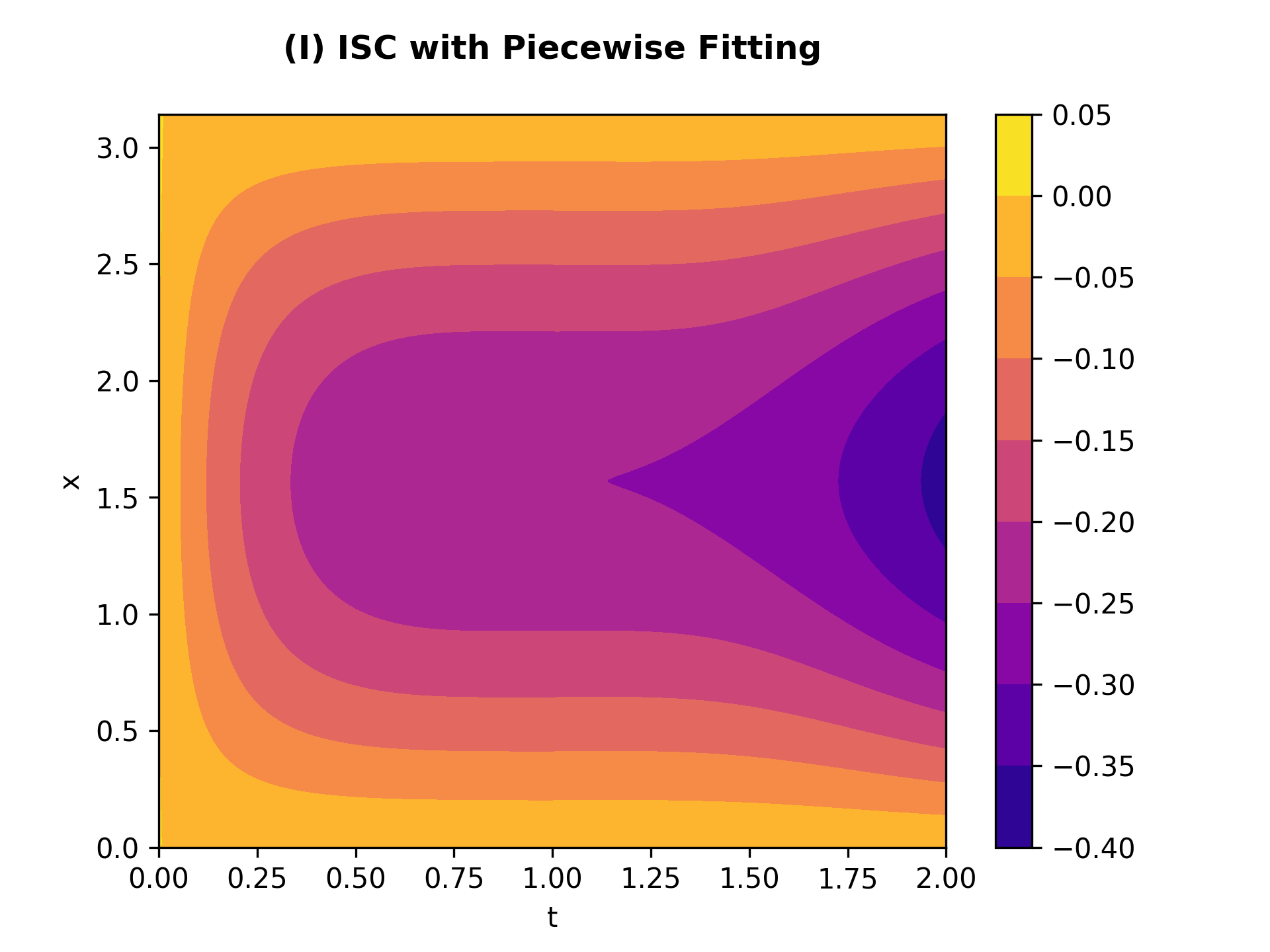}
		
		\label{chutian2}
	\end{minipage}

 \begin{minipage}{0.24\linewidth}
		\centering
		\includegraphics[width=\linewidth]{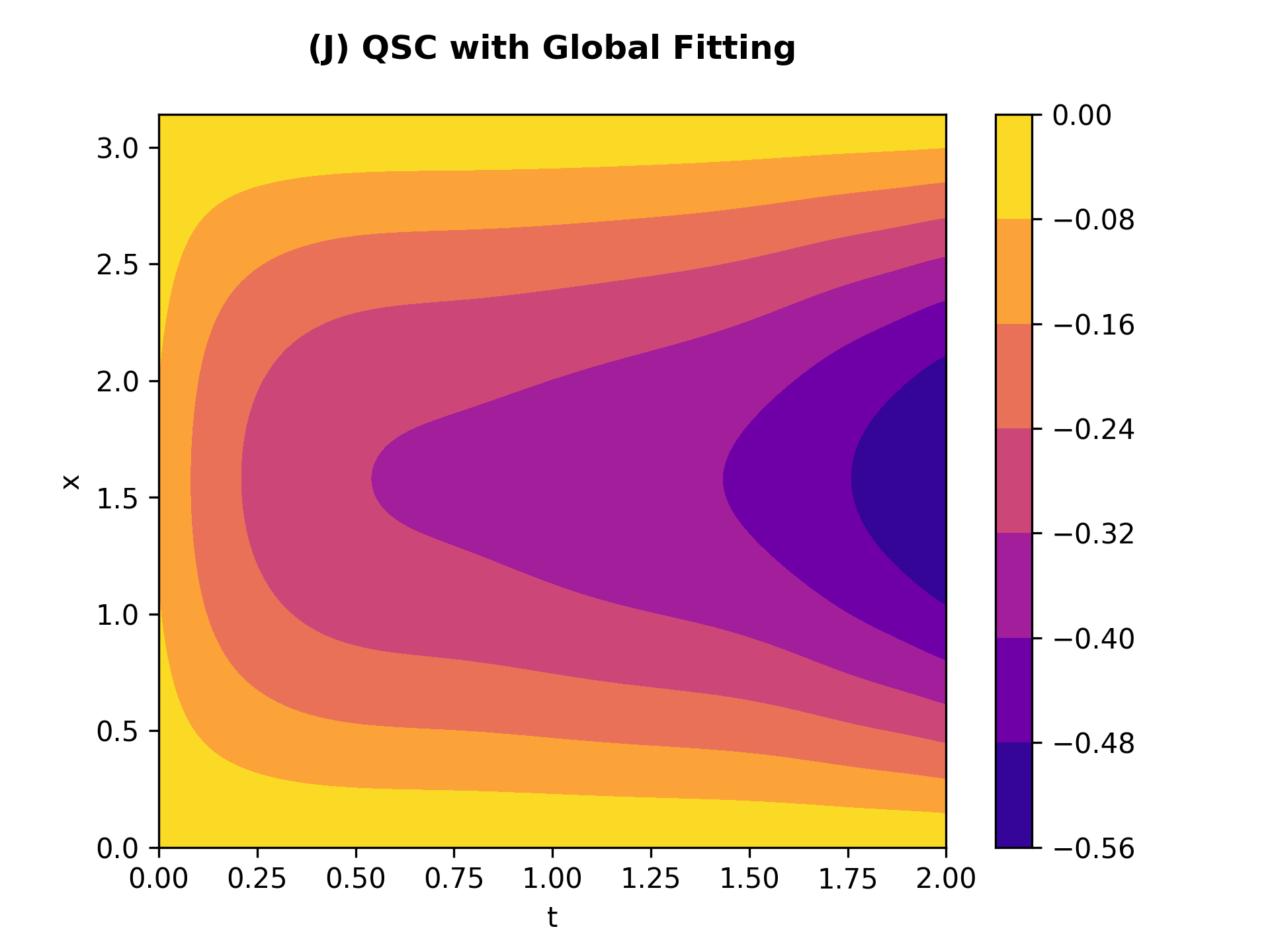}
		
		\label{chutian3}
	\end{minipage}
	\begin{minipage}{0.24\linewidth}
		\centering
		\includegraphics[width=\linewidth]{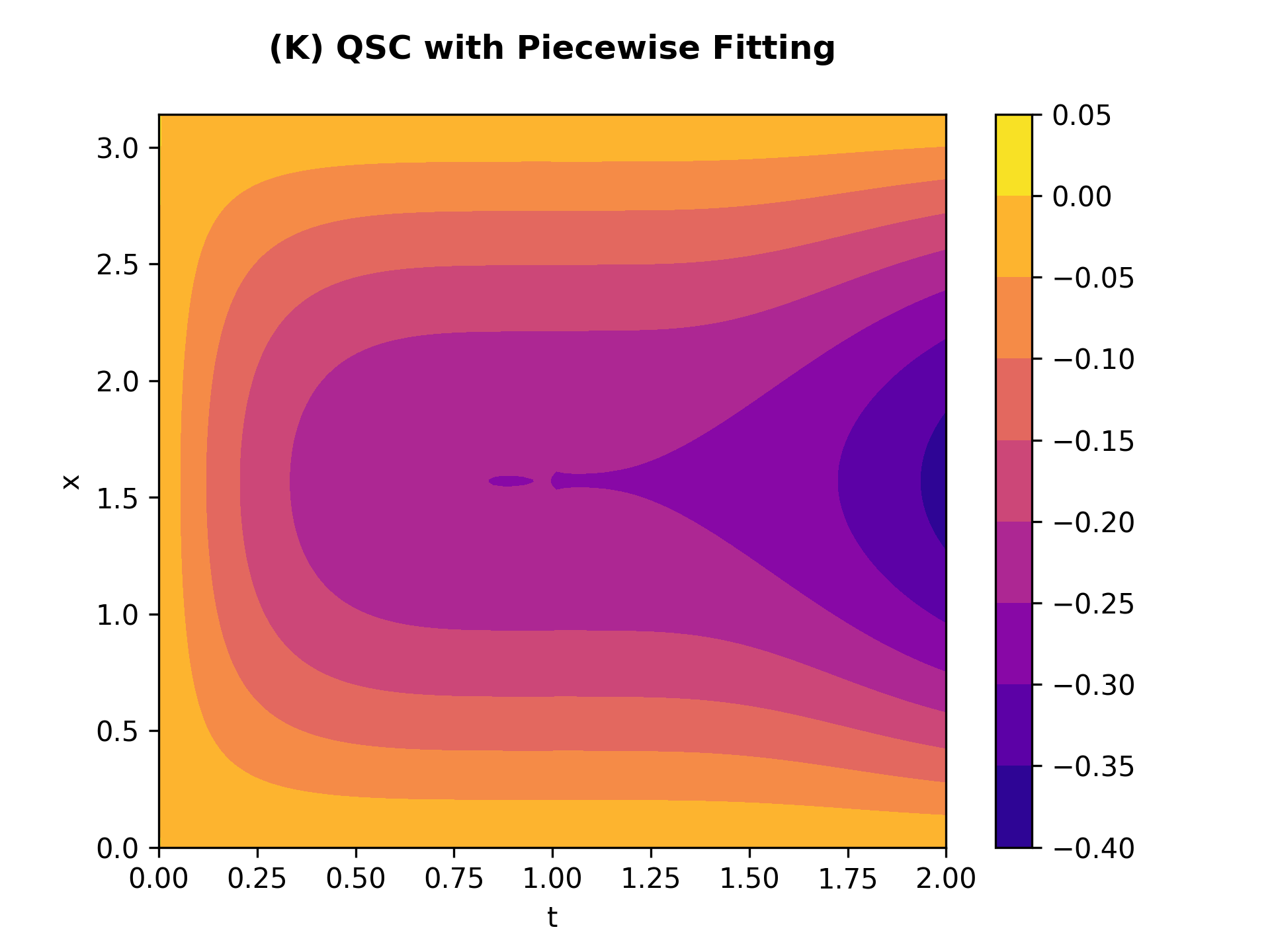}
		
		\label{chutian4}
	\end{minipage}
 \begin{minipage}{0.24\linewidth}
		\centering
		\includegraphics[width=\linewidth]{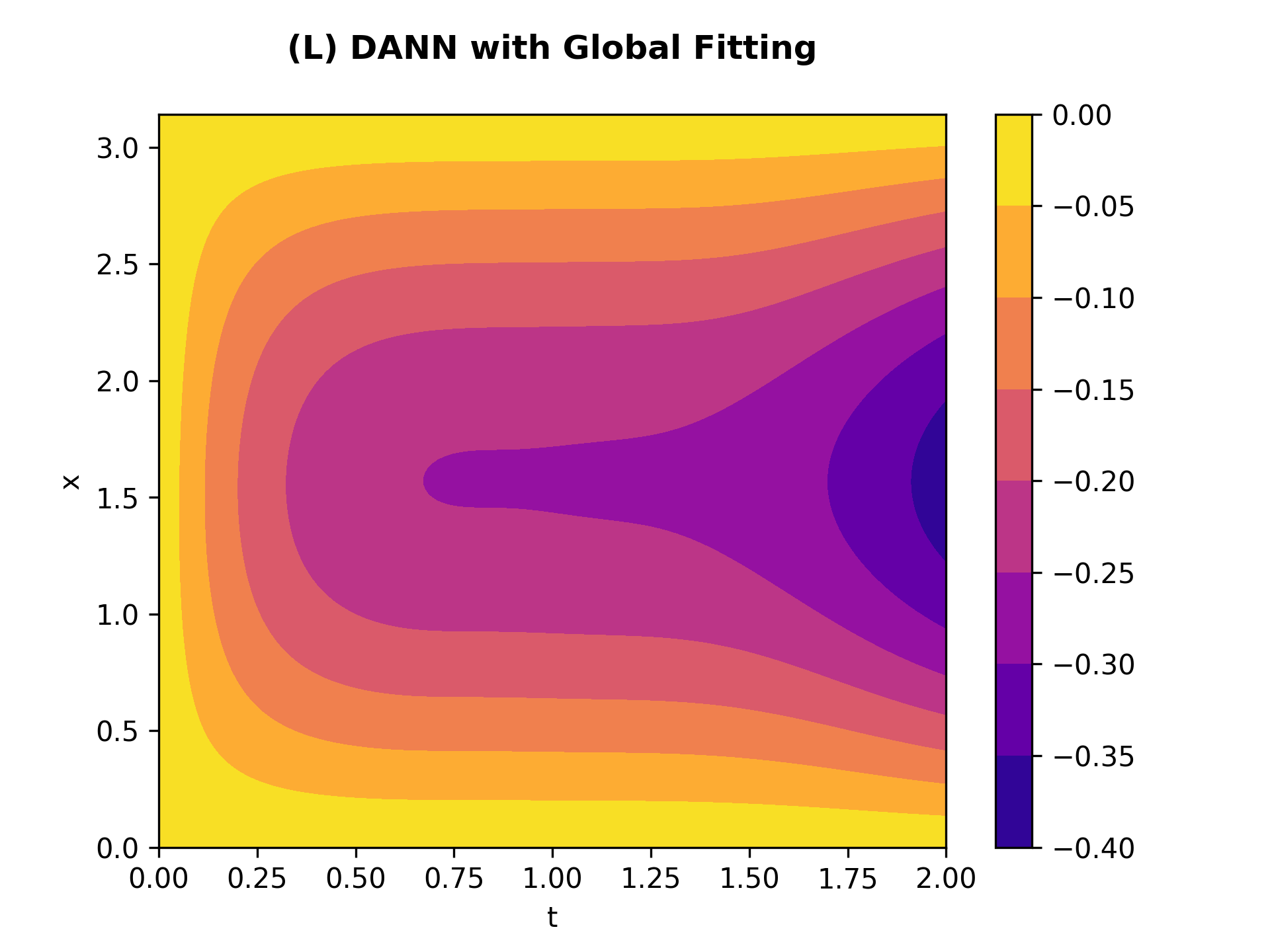}
		
		\label{chutian2}
	\end{minipage}
 \begin{minipage}{0.24\linewidth}
		\centering
		\includegraphics[width=\linewidth]{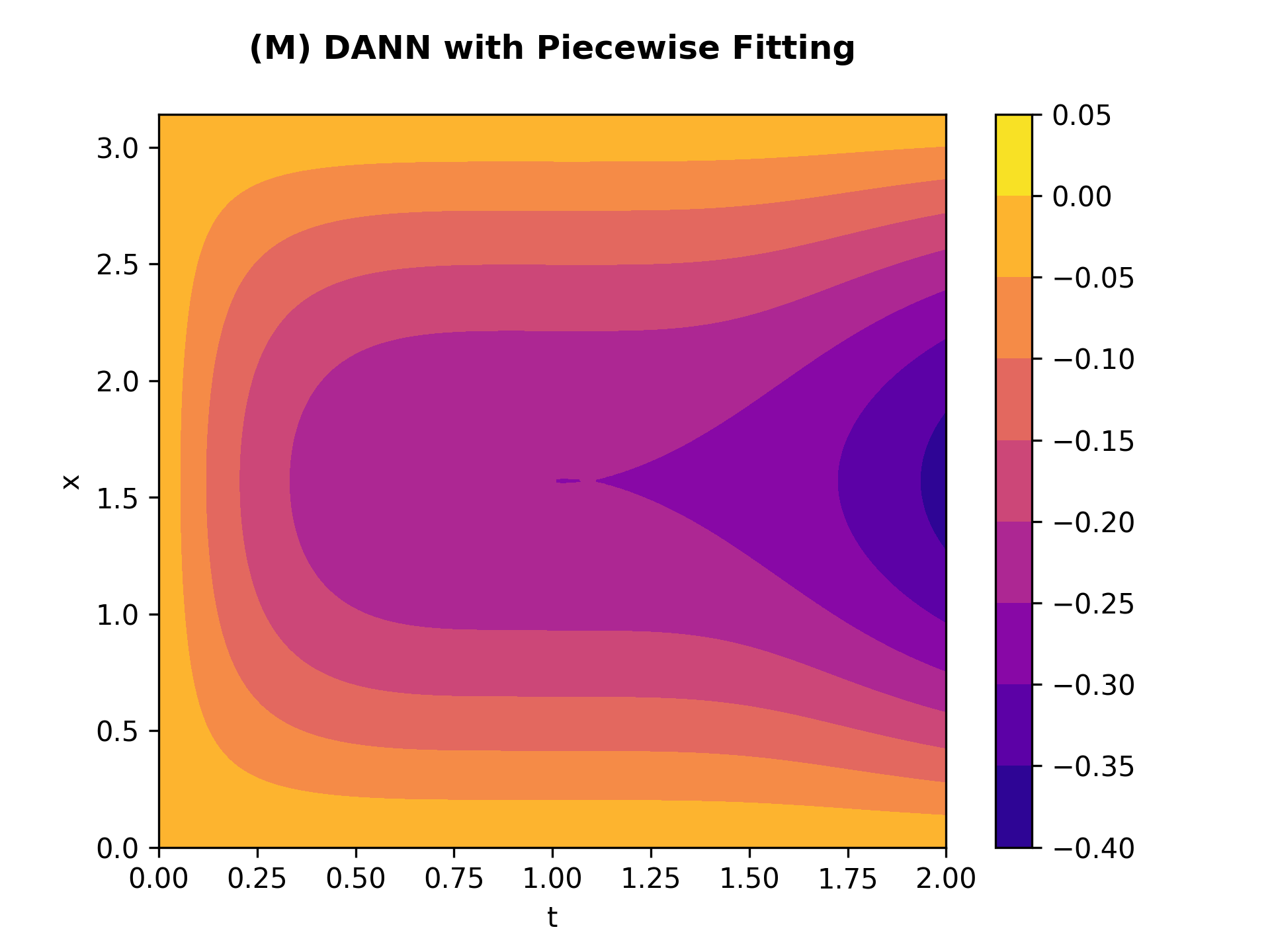}
		
		\label{chutian2}
	\end{minipage}
 \caption{The exact solution (A) and the approximate solutions obtained by different models based on global and piecewise fitting. }
 \label{delay2jie}
\end{figure}
 \begin{figure}[H]
	\centering
	\begin{minipage}{0.32\linewidth}
		\centering
		\includegraphics[width=\linewidth]{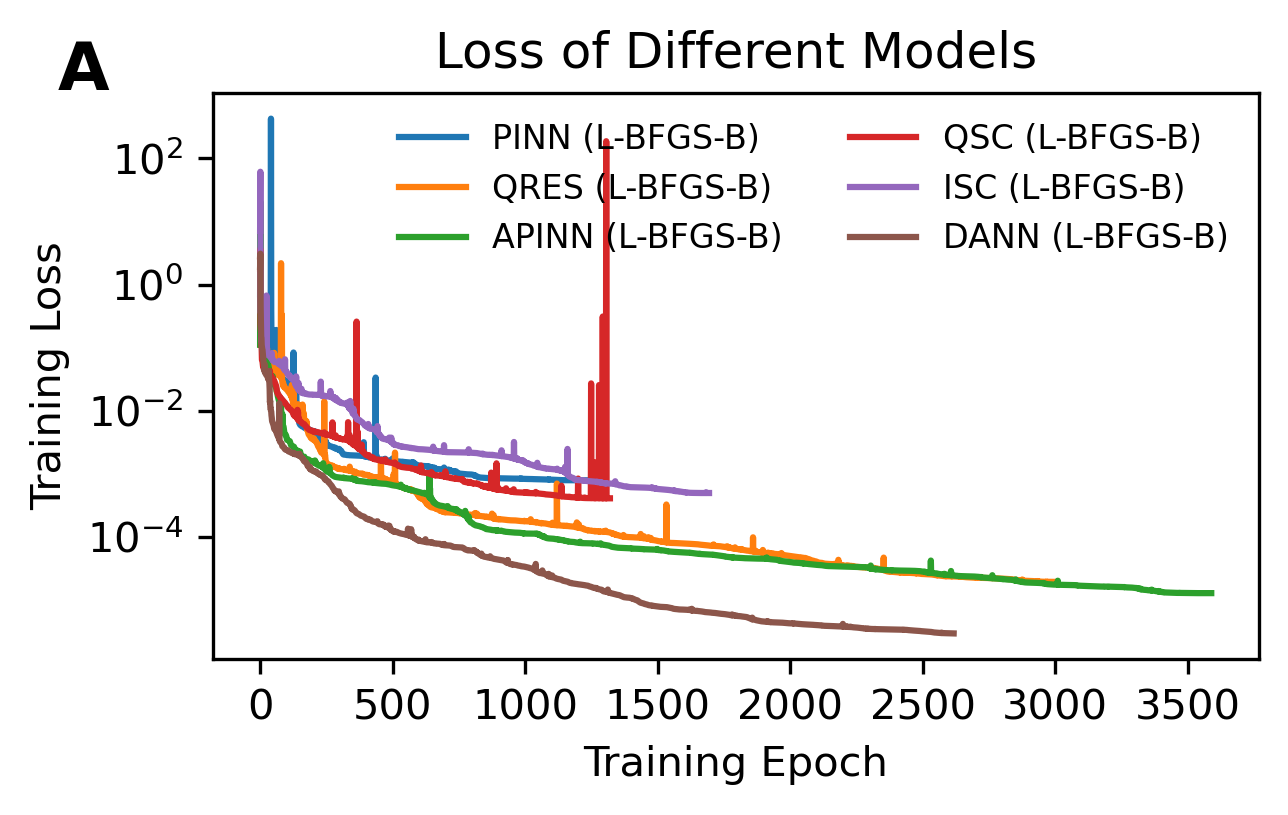}
	\end{minipage}
	\begin{minipage}{0.32\linewidth}
		\centering
		\includegraphics[width=\linewidth]{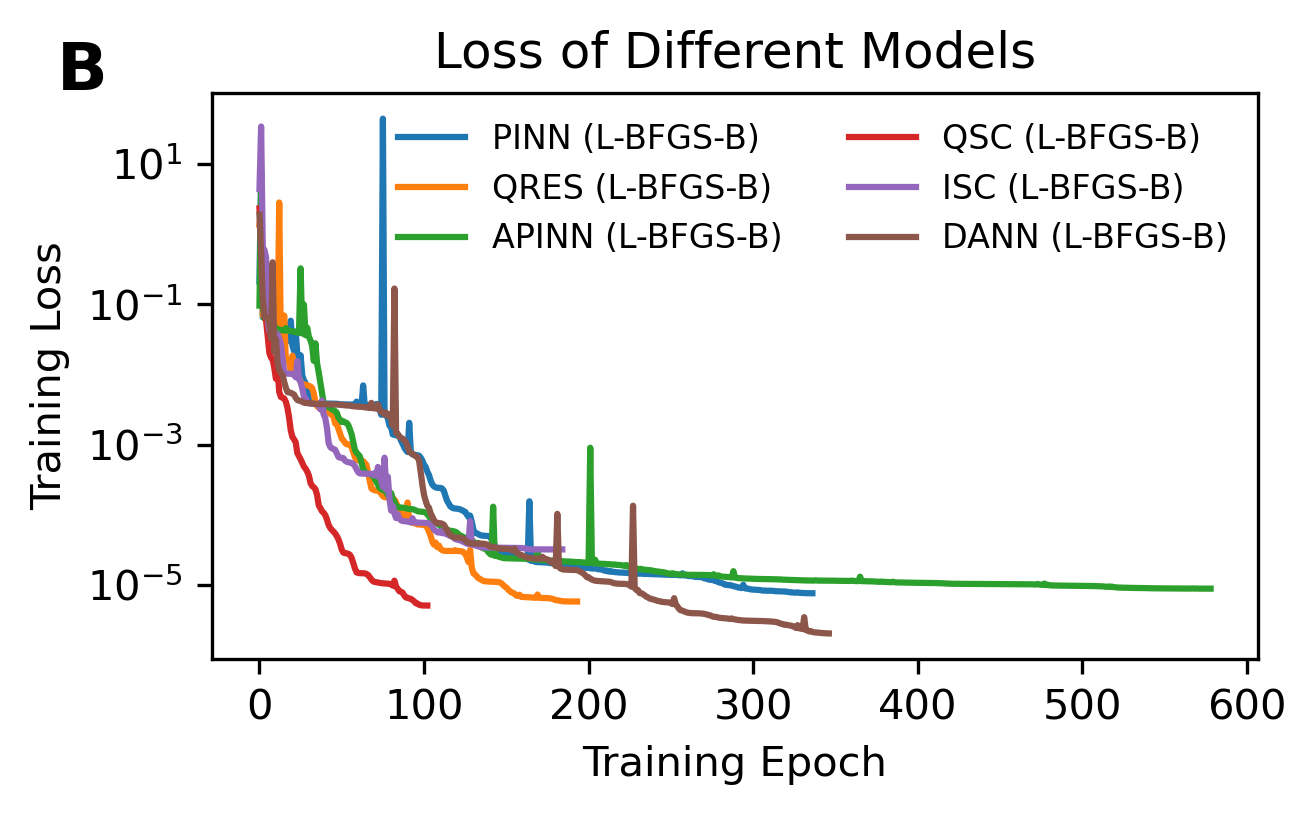}
	\end{minipage}
	\begin{minipage}{0.32\linewidth}
		\centering
		\includegraphics[width=\linewidth]{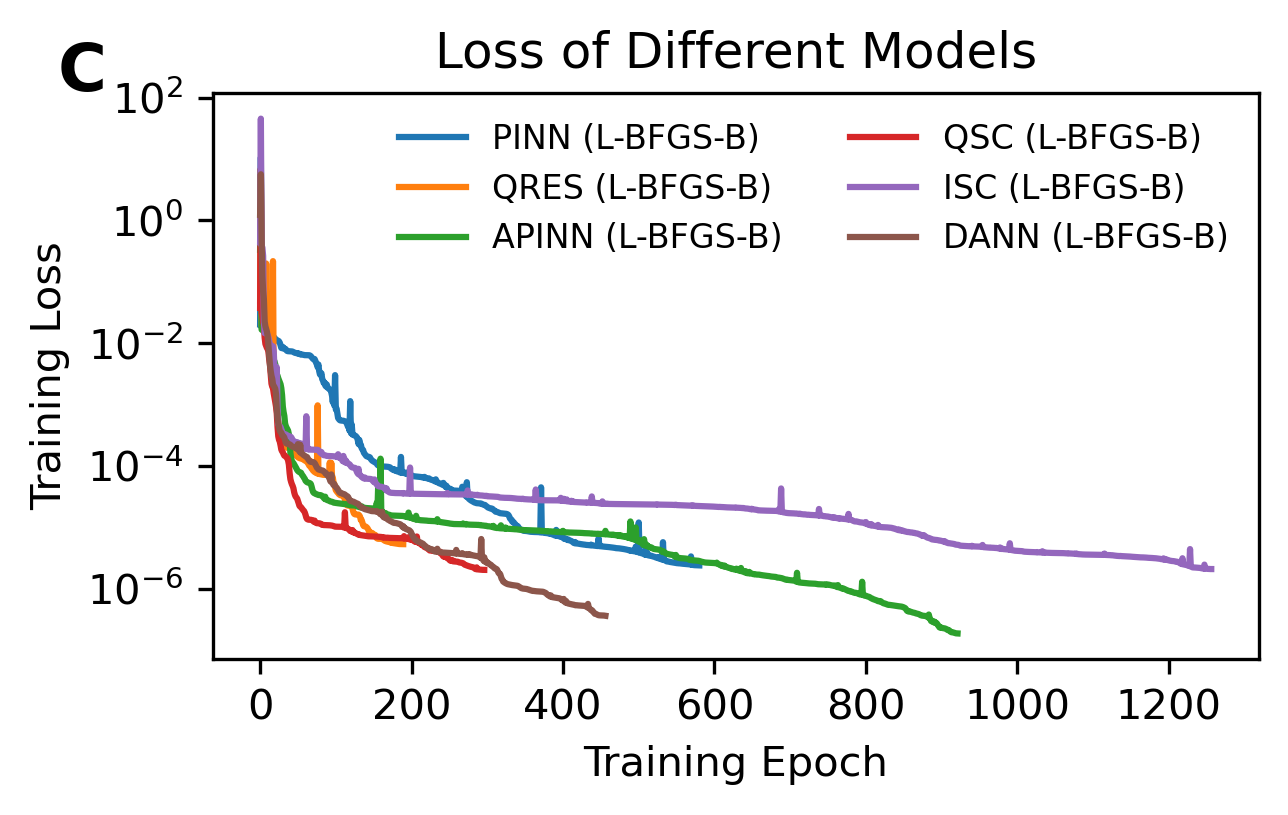}
	\end{minipage}
	\caption{Training loss curves for solving Example 2 based on global fitting (A), piecewise fitting (Subdomain 1) (B) and piecewise fitting (Subdomain 2) (C).}
	\vspace{-0.2cm}
	\label{delay2loss}
\end{figure}

\noindent the highest accuracy and PINN has the lowest accuracy. Additionally, according to Figure $\ref{delay2loss}$ (A), DANN showed the fastest convergence speed among these six models. In piecewise fitting, the approximate solutions accuracy obtained by these six models ranks from high to low as follows: DANN $>$ PINN $>$ QRES $>$ QSC $>$ APINN $>$ ISC, which means DANN has the highest accuracy and ISC has the lowest accuracy. As can be seen from Figure $\ref{delay2loss} $ (B) and (C), in terms of convergence speed, the top three from fast to slow are: QSC $>$  QRES $>$ DANN, all three converge faster than PINN. However, after approximately 300 steps, the loss functions of QSC and QRES stabilize, whereas the loss function of DANN continues to decrease.
\begin{table*}[t]
    \centering
    \caption{The relative $L^2$ errors of $u$ obtained by different models.}
    \begin{tabular}{ccccccc}
    \toprule
        \textbf{Model} & \textbf{PINN} & \textbf{APINN} & \textbf{QRES} & \textbf{ISC} & \textbf{QSC}  & \textbf{DANN} \\ \midrule
        \textbf{Global Fitting} & 9.07e-01 & 1.80e-01 &2.70e-01 &5.59e-01   &4.49e-01 & 1.65e-02 \\
        \textbf{Piecewise Fitting} & 9.78e-04 & 1.55e-03 &1.01e-03 &2.54e-03   &1.25e-03 & 7.53e-04 \\
    \bottomrule
    \end{tabular}
    \label{delay2table}
\end{table*}

\begin{figure}[H]
	\centering
	\subfigure{
			\begin{minipage}[t]{0.45\linewidth}
			\centering
			\includegraphics[width=\textwidth]{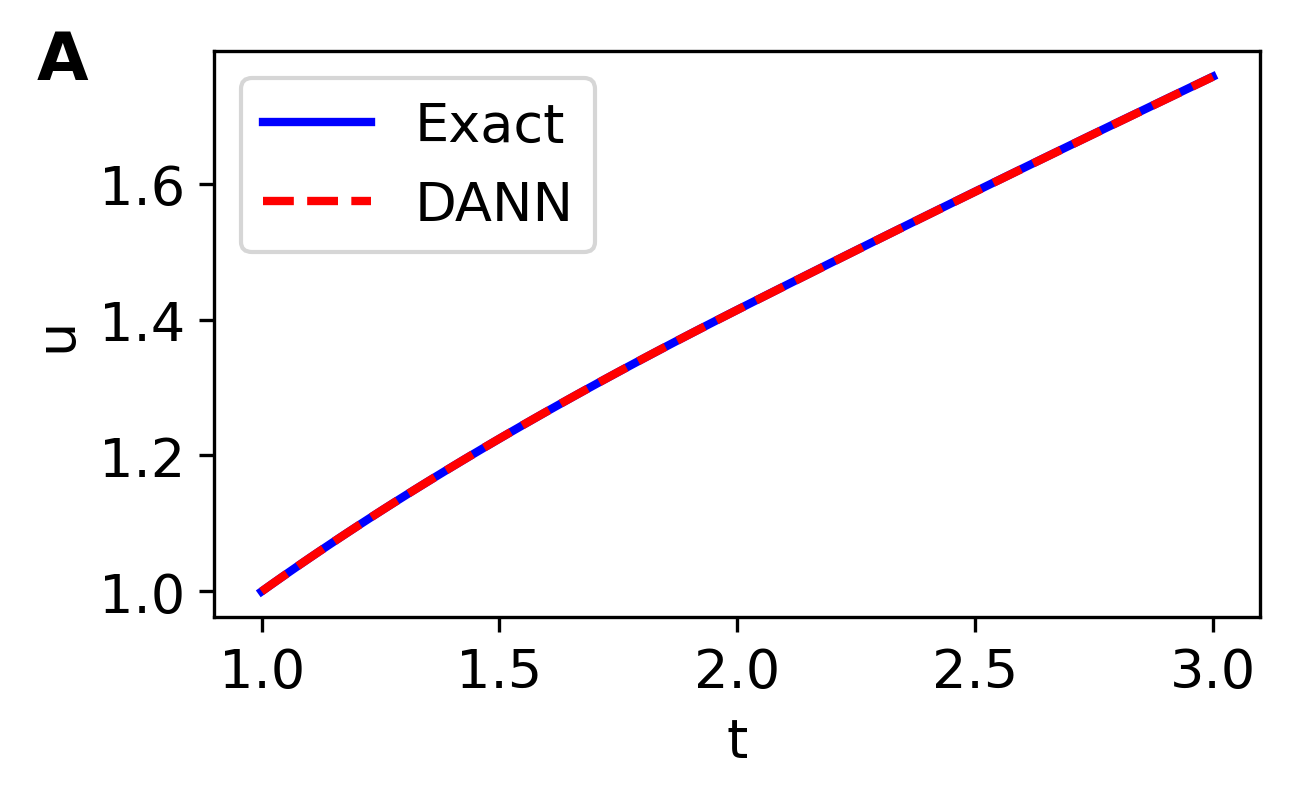}\\
			\vspace{0.02cm}
		\end{minipage}%
	}%
	\subfigure{
	\begin{minipage}[t]{0.45\linewidth}
			\centering
			\includegraphics[width=\textwidth]{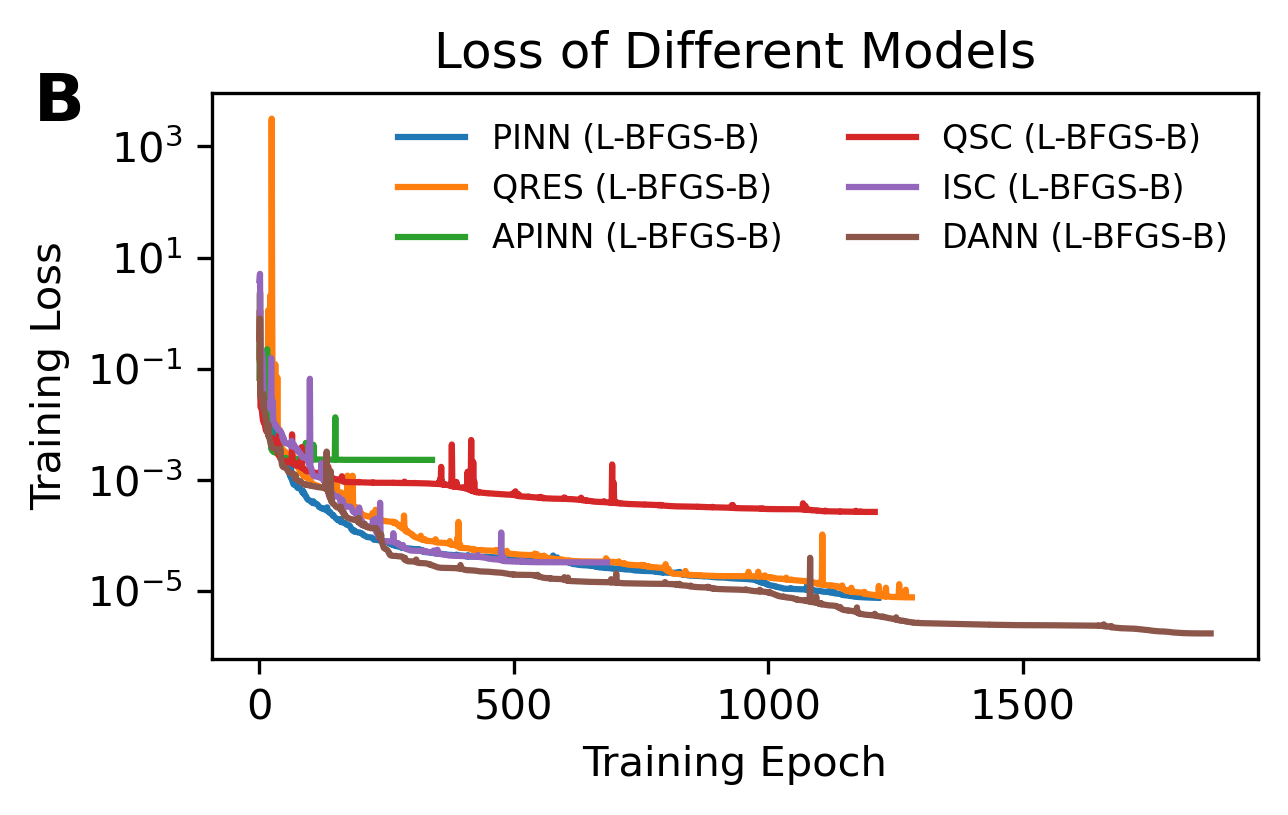}\\
			\vspace{0.02cm}
		\end{minipage}%
	}%

	\centering
	\caption{The approximate solution obtained by DANN and the exact solution (A). The loss curves for solving Example 3 (B).}
	\vspace{-0.2cm}
 \label{delay3}
\end{figure}

\begin{exm}
Consider the following DDE with state-dependent delay
\begin{eqnarray}
        \begin{cases}
         {u}^\prime(t) =\frac{1}{2}\frac{1}{\sqrt{t}}u(u(t)-\sqrt{2}+1), \quad 1\le t \le 3,\\
         u(t)=1,\quad t \le 1.\\
     \end{cases}
		\end{eqnarray}
	Its exact solution is
		\begin{eqnarray}
u(t) =
\begin{cases}
\sqrt{t}       ,\quad 1\le t \le 2, \\
\frac{t}{4}+\frac{1}{2}+(1-\frac{\sqrt{2}}{2}) \sqrt{t}  ,\quad 2\le t \le 3. \\
\end{cases}
\end{eqnarray}
\end{exm}

This problem involves two primary discontinuity points at $t = 0$ and $t = 1$. However, without prior knowledge of the exact solution, it is not possible to determine the exact locations of the primary discontinuity points in DDE with state-dependent delay. Therefore, we can only solve the problem using global fitting approach. 

\begin{table*}[t]
    \centering
    \caption{The relative $L^2$ errors of $u$ obtained by different models.}
    \label{tab:univ-compa}
    \begin{tabular}{ccccccc}
    \toprule
        \textbf{Model} & \textbf{PINN} & \textbf{APINN} & \textbf{QRES} & \textbf{ISC} & \textbf{QSC}  & \textbf{DANN} \\ \midrule
      \textbf{Global Fitting}&  1.56e-04 & 2.04e-02 &2.08e-04 &5.46e-04   &3.55e-03 & 9.18e-05 \\
    \bottomrule
    \end{tabular}
    \label{delay3table}
\end{table*}

To improve the fitting accuracy, we apply a special treatment to the expression of the approximate solution by setting $\hat{u}(t)=1+(t-1)\mathcal{N}(t;\theta)$, which guarantees the exact satisfaction for $t=1$. Table $\ref{delay3table}$ provides the relative $L^2$ errors obtained by different models. It can be observed that the approximate solution accuracy obtained by these six models ranks from high to low as follows: DANN $>$ PINN $>$ QRES $>$ ISC $>$ QSC $>$ APINN, with DANN achieves the highest accuracy among these models. The approximate solution obtained by DANN is presented in Figure $\ref{delay3}$ (A), demonstrating an excellent fit with the exact solution. 
Figure $\ref{delay3}$ (B) depicts the loss curves, illustrating that initially, PINN exhibits faster convergence than DANN. However, after approximately 250 iterations, DANN converges more rapidly and ultimately achieves a lower loss value compared to PINN.

\begin{exm}		
Consider the following $d$-dimensional DPDE with nonvanishing delay
\begin{eqnarray}
\begin{cases}
    u_t=\Delta u+du-u(t-1, \boldsymbol{x}),\quad (t, \boldsymbol{x}) \in[0,2] \times \Omega, \\
u(t, \boldsymbol{x})=\left(1-t+\frac{(t-1)^2}{2}\right) Q(\boldsymbol{x}),\quad (t, \boldsymbol{x}) \in[-1,0] \times \Omega, \\
u(t,\boldsymbol{x})|_{\partial \Omega}=0,\quad  t \in[0,2].
     \end{cases}
		\end{eqnarray}
Its exact solution is
\begin{eqnarray}
u(t, \boldsymbol{x})=\begin{cases}
\left(-\frac{1}{3}-t+\frac{(t-1)^2}{2}-\frac{(t-2)^3}{6}\right) Q(\boldsymbol{x}),\quad (t, \boldsymbol{x}) \in[0,1] \times \Omega, \\
\left(-\frac{7}{3}+\frac{t}{3}+\frac{(t-1)^2}{2}-\frac{(t-2)^3}{6}+\frac{(t-3)^4}{24}\right) Q(\boldsymbol{x}),\quad (t, \boldsymbol{x}) \in[1,2] \times \Omega,
\end{cases}
		\end{eqnarray}
where $\boldsymbol{x}=(x_1,x_2,\ldots,x_{d})$, $\Omega=[0,\pi]^{d}$ and $Q(\boldsymbol{x})=\prod \limits_{i=1}^{d} \sin(x_i)$.
\end{exm}

This problem involves two primary discontinuity points at $t=0$ and $t=1$. We represent its approximate solution as
$$\hat{u}(t, \boldsymbol{x})=\prod \limits_{i=1}^{d} x_i \prod \limits_{i=1}^{d} \sin(x_i-\pi)\mathcal{N}(t,\boldsymbol{x};\theta),$$
where $\mathcal{N}(t,\boldsymbol{x};\theta)$ represents a neural network with parameter $\theta$.

\begin{table*}[t]
    \centering
    \caption{The relative $L^2$ errors of $u$ obtained by different models.}
    \scalebox{0.87}{
    \begin{tabular}{c|ccccccc}
    \toprule
        \textbf{Dimension} &\textbf{Model} & \textbf{PINN} & \textbf{APINN} & \textbf{QRES} & \textbf{ISC} & \textbf{QSC}  & \textbf{DANN} \\ \midrule
        \multirow{2}{*}{d=3}&\textbf{Global Fitting} & 3.61e-02 & 1.24e+00 &2.02e-02 &3.49e-02   &5.36e-02 & 1.30e-02 \\
        
        \multirow{2}{*}{ }&\textbf{Piecewise Fitting} & 3.38e-04 & 5.90e-01 &7.59e-05 &3.39e-04   & 6.17e-05& 5.37e-05 \\
        \midrule
        \multirow{2}{*}{d=8}&\textbf{Global Fitting} & 1.35e-01 & 1.04e+00 &1.83e-01 &1.34e-01   &1.81e-01 & 1.73e-01 \\
        
        \multirow{2}{*}{ }&\textbf{Piecewise Fitting} & 3.89e-02 & 1.00e+00 &5.56e-03 &4.42e-02   & 4.90e-03& 3.54e-03 \\

    \bottomrule
    \end{tabular}}
\label{delay5table}
\end{table*}

For $d=3$ and $d=8$, we employ PINN, APINN, QRES, ISC, QSC, and DANN models based on global and piecewise fitting to solve the problem. In piecewise fitting, the domain $(t, \boldsymbol{x}) \in[0,2] \times \Omega$ is divided into $(t, \boldsymbol{x}) \in$ $[0,1] \times\Omega$ (Subdomain 1) and $(t, \boldsymbol{x}) \in[1,2] \times\Omega$ (Subdomain 2). Table $\ref{delay5table}$ provides the relative $L^2$ errors between the exact solution and the approximate solutions obtained by these six models when $d=3$ and $d=8$. 
The results in Table $\ref{delay5table}$ indicate that piecewise fitting provides superior fitting performance compared to global fitting, and the approximate solutions obtained by DANN, when utilizing piecewise fitting, demonstrate the highest accuracy among all models evaluated.
Figures $\ref{delay5wucha}$ and $\ref{delay6wucha}$ display the absolute errors for six models based on piecewise fitting at $t=0, 1.0, 2.0$ when $d=3$ and $d=8$, respectively. It is evident that DANN significantly reduces absolute errors.
 In summary, piecewise fitting significantly improves the accuracy of the approximate solution compared to global fitting, and DANN based on piecewise fitting consistently outperforms other mentioned models in terms of accuracy. Figure $\ref{delay5loss}$ presents the loss curves of di-
 \begin{figure}[H]
	\centering
	\begin{minipage}{0.27\linewidth}
		\centering
		\includegraphics[width=\linewidth]{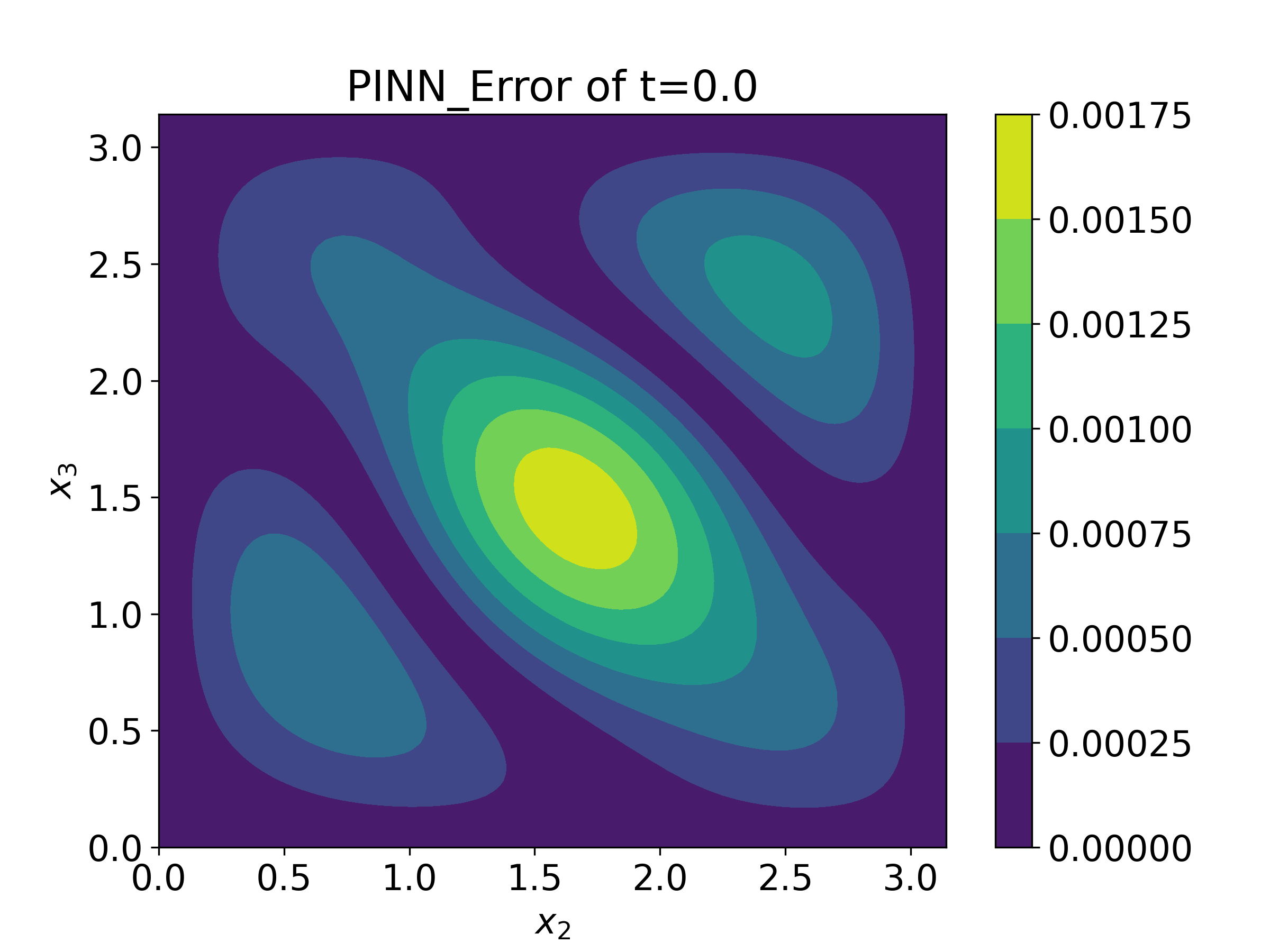}
		
		\label{chutian3}
	\end{minipage}
	\begin{minipage}{0.27\linewidth}
		\centering
		\includegraphics[width=\linewidth]{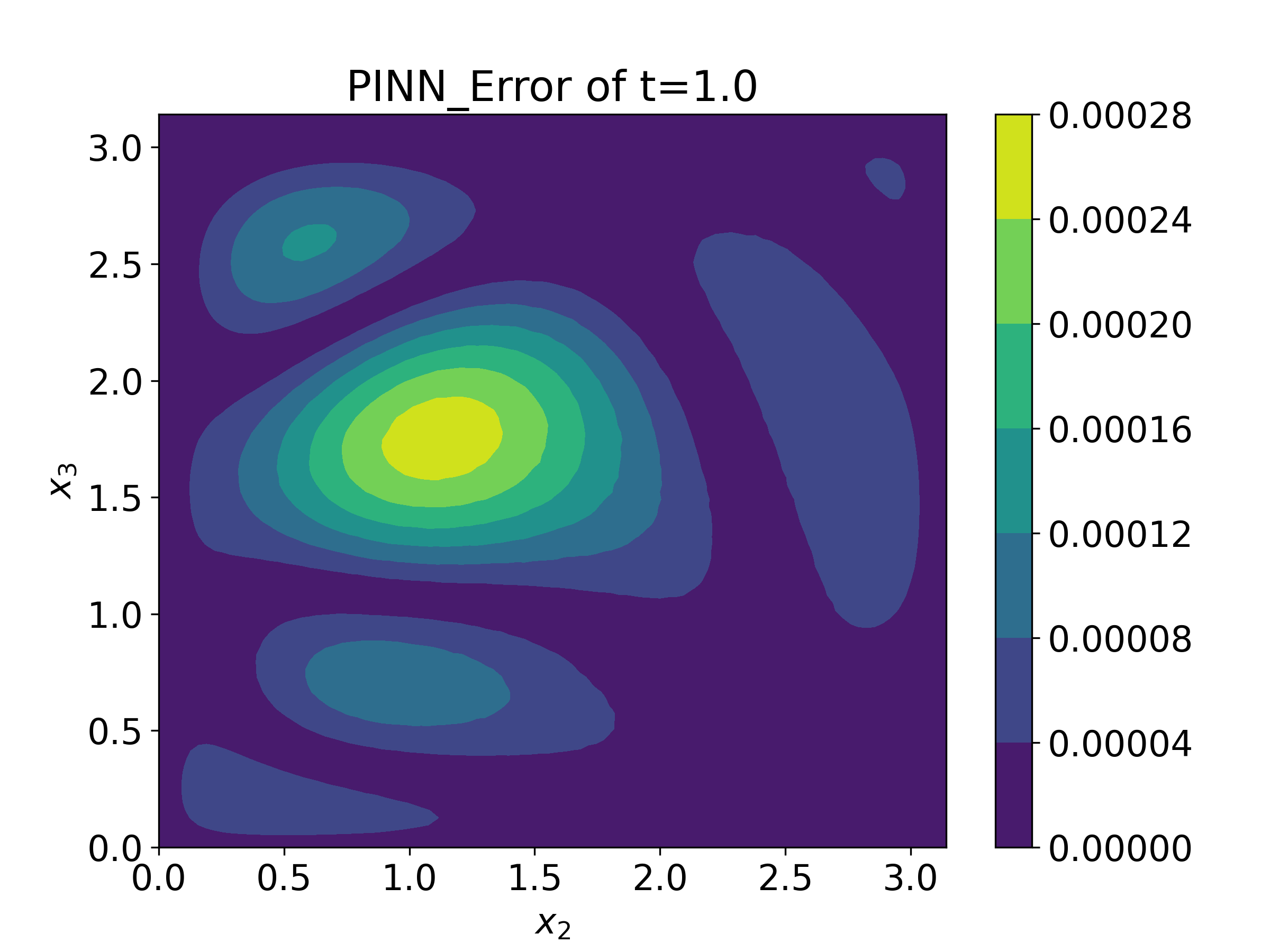}
		
		\label{chutian4}
	\end{minipage}
 \begin{minipage}{0.27\linewidth}
		\centering
		\includegraphics[width=\linewidth]{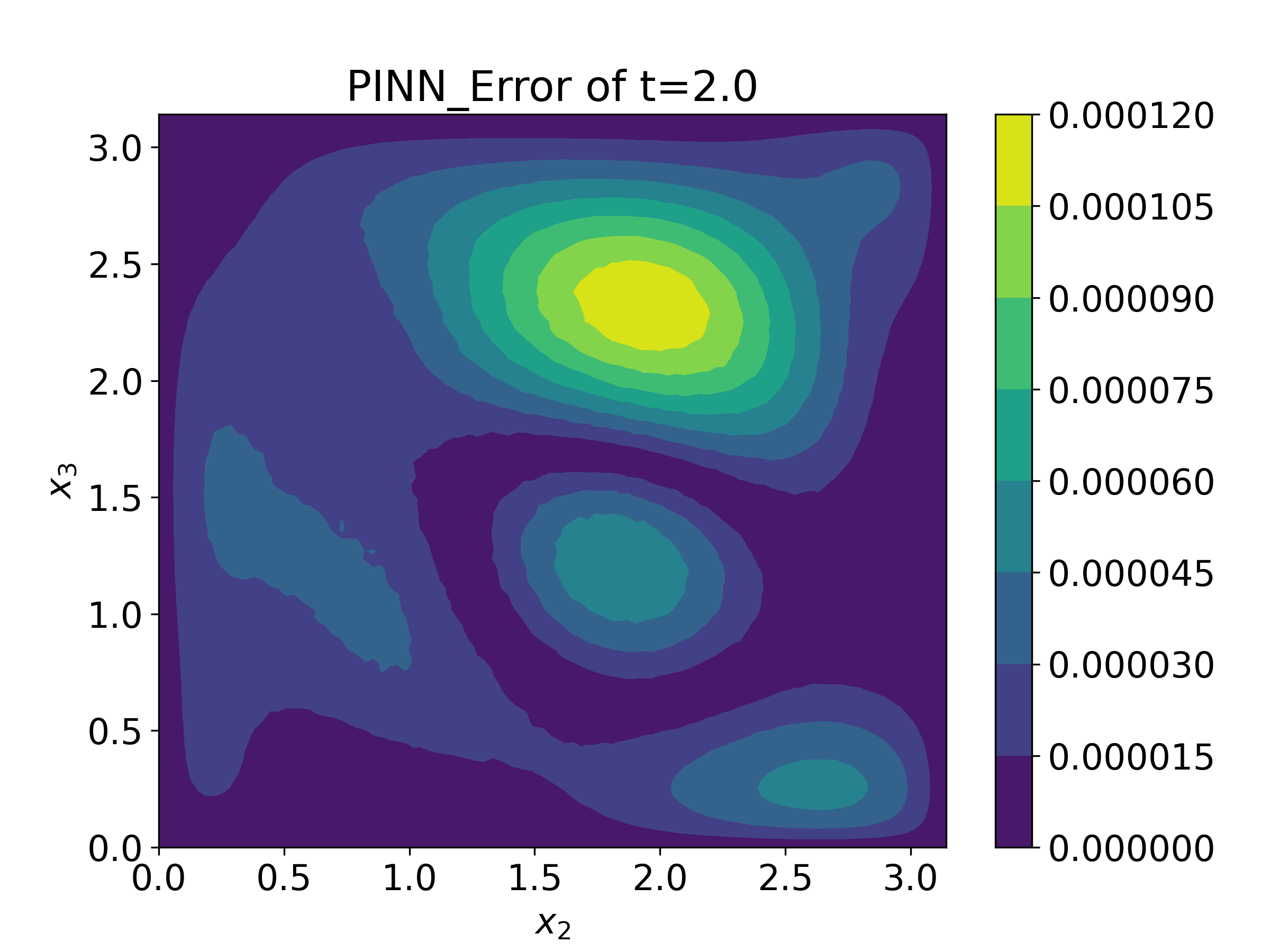}
		
		\label{chutian2}
	\end{minipage}

 \begin{minipage}{0.27\linewidth}
		\centering
		\includegraphics[width=\linewidth]{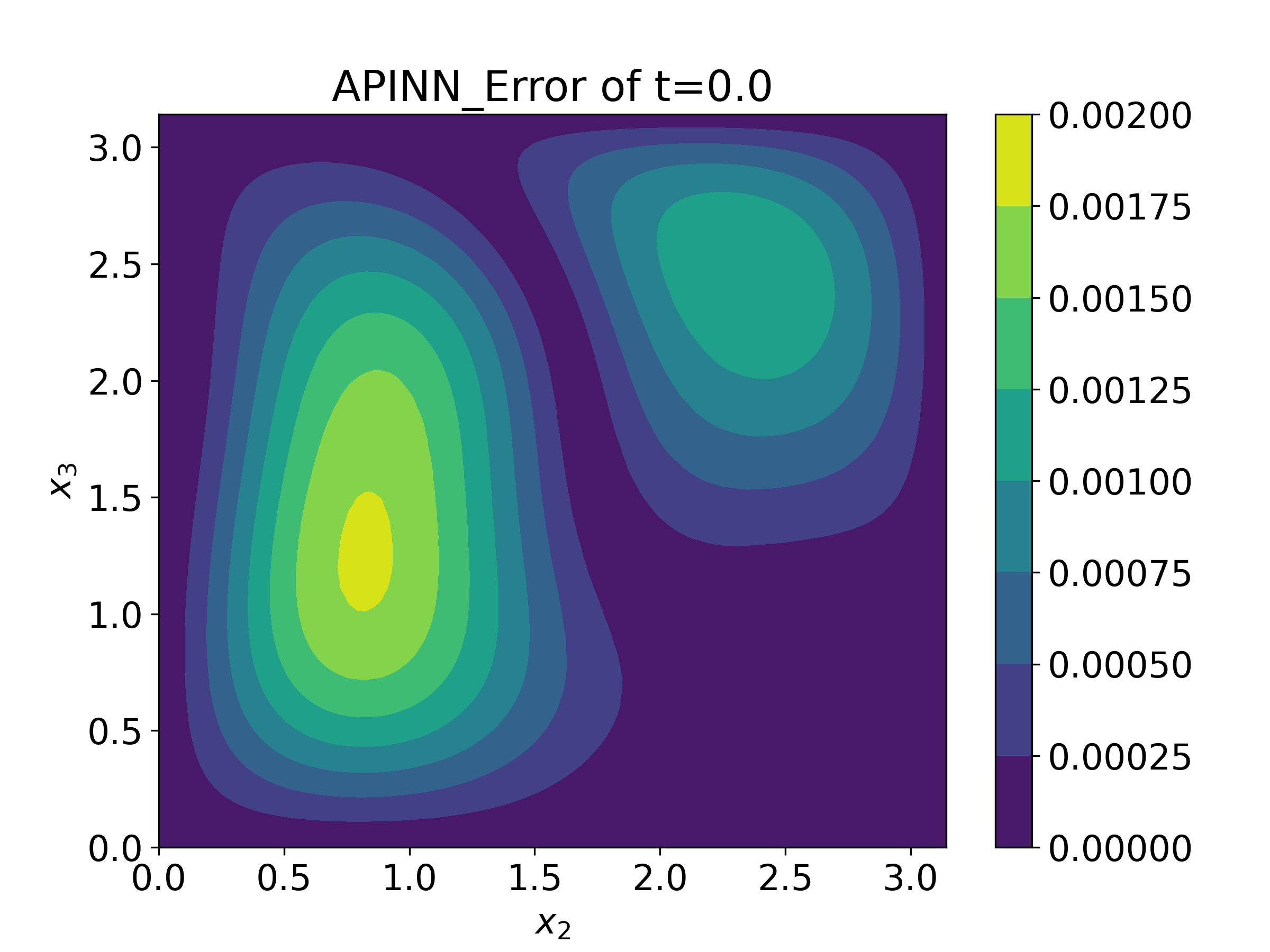}
		
		\label{chutian3}
	\end{minipage}
	\begin{minipage}{0.27\linewidth}
		\centering
		\includegraphics[width=\linewidth]{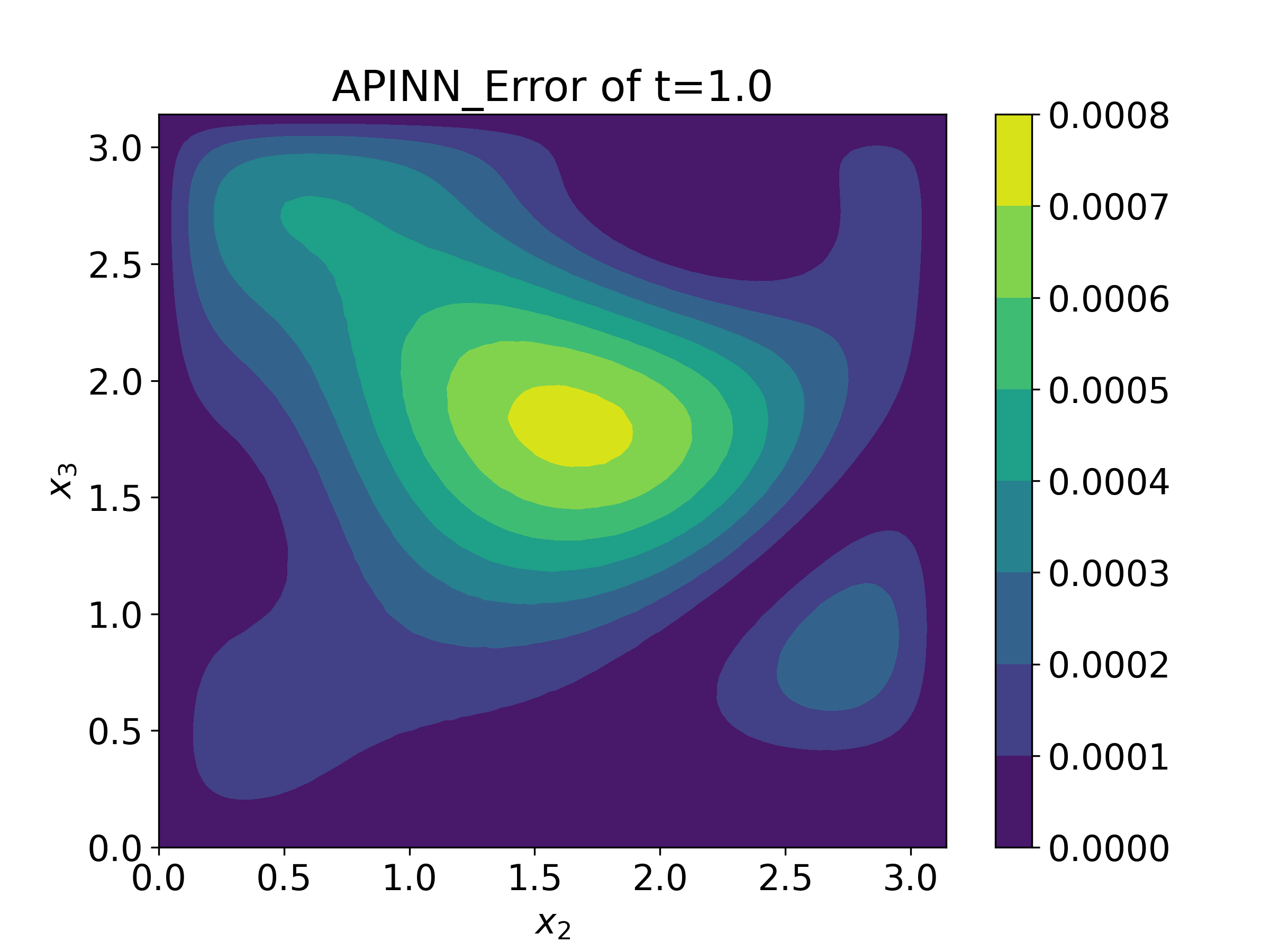}
		
		\label{chutian4}
	\end{minipage}
 \begin{minipage}{0.27\linewidth}
		\centering
		\includegraphics[width=\linewidth]{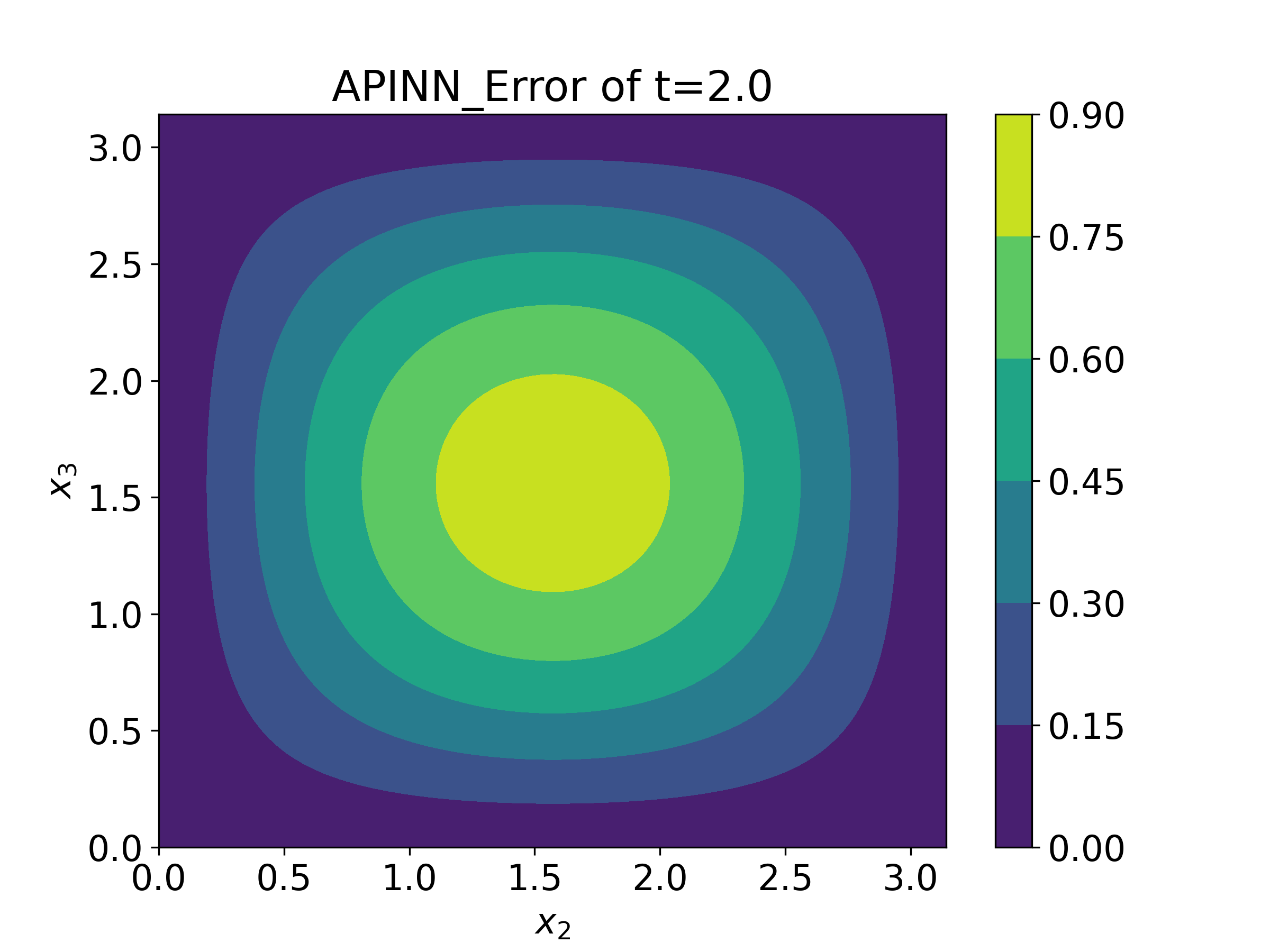}
		
		\label{chutian2}
	\end{minipage}

 \begin{minipage}{0.27\linewidth}
		\centering
		\includegraphics[width=\linewidth]{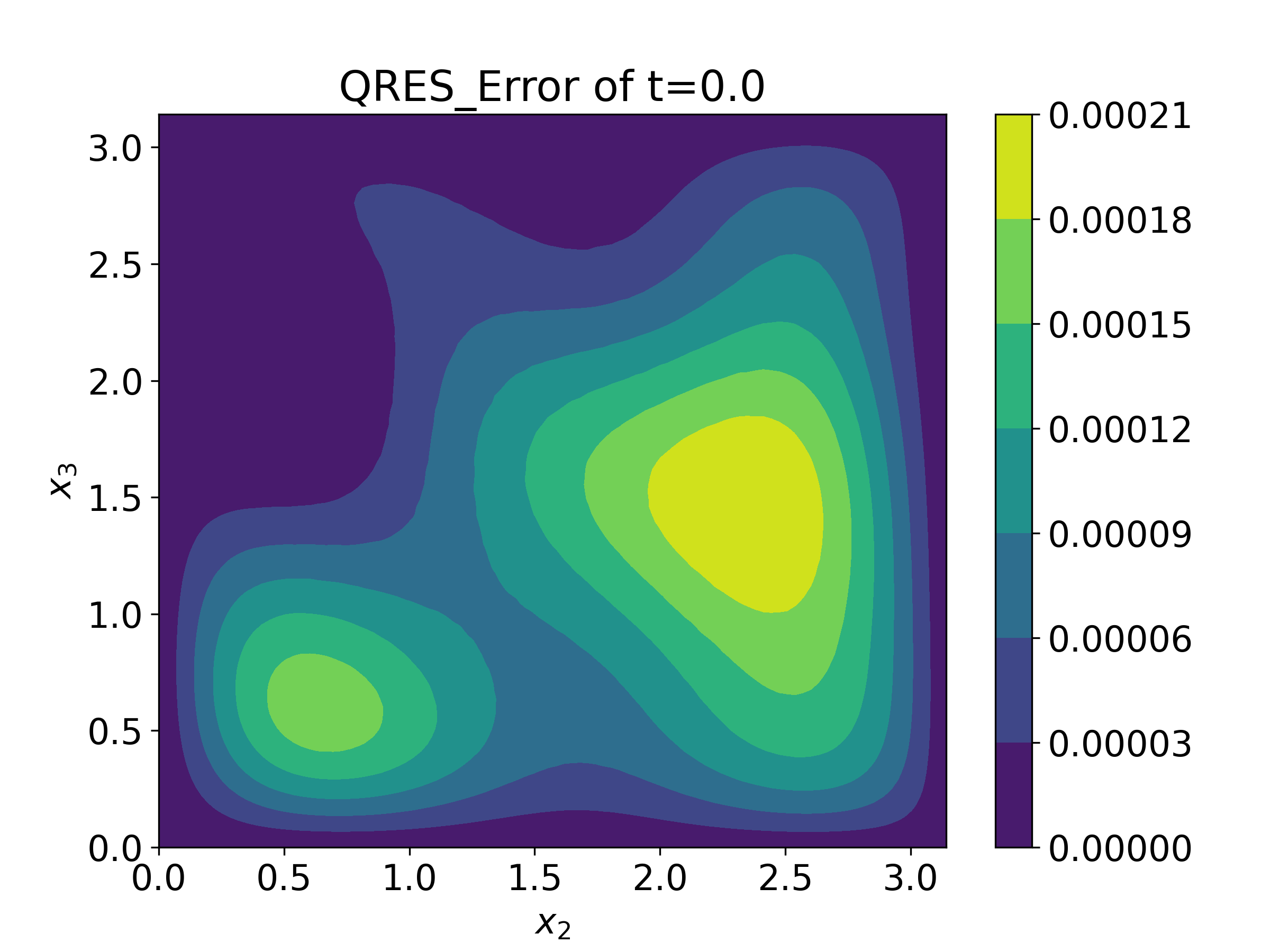}
		
		\label{chutian3}
	\end{minipage}
	\begin{minipage}{0.27\linewidth}
		\centering
		\includegraphics[width=\linewidth]{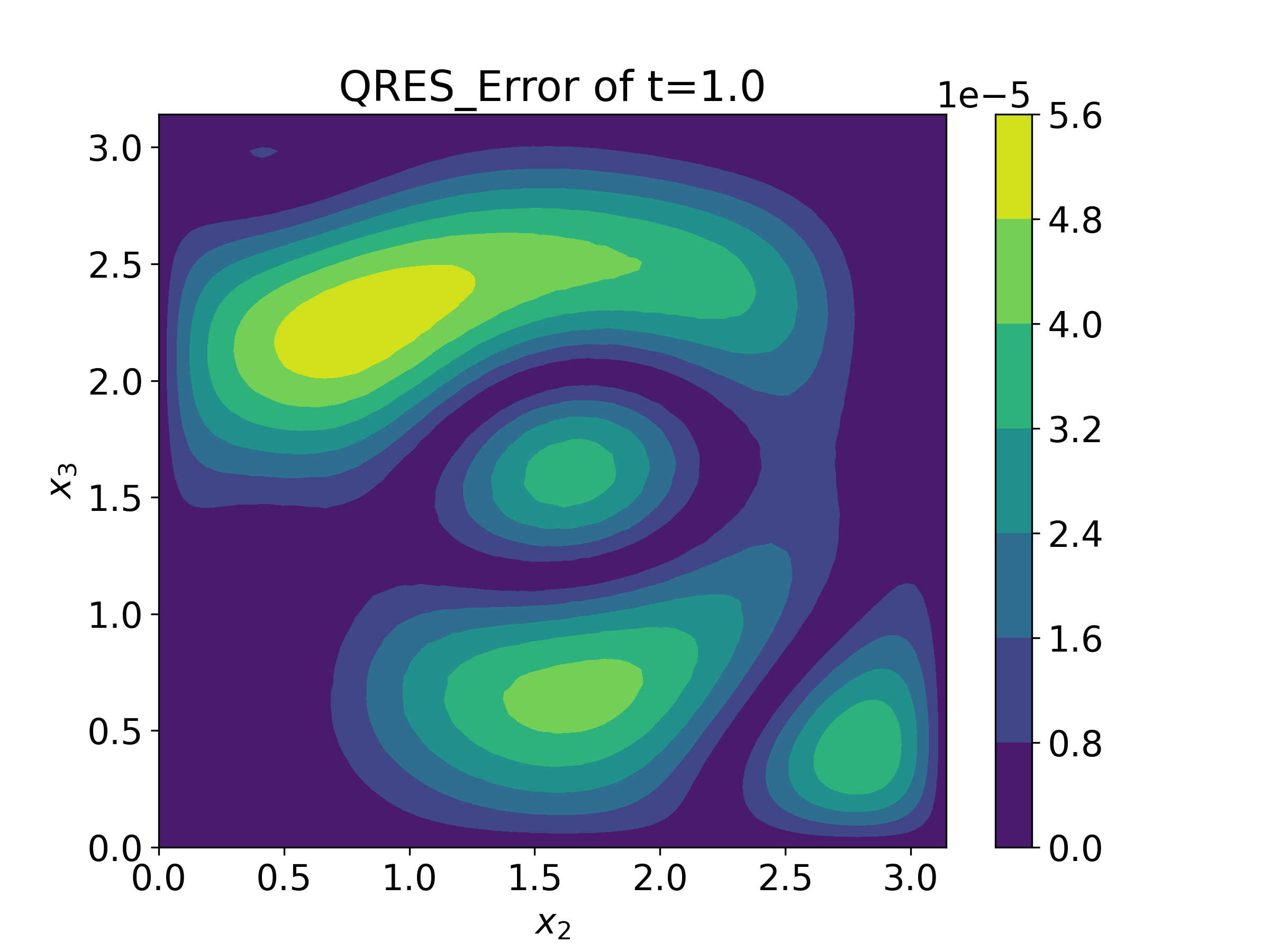}
		
		\label{chutian4}
	\end{minipage}
 \begin{minipage}{0.27\linewidth}
		\centering
		\includegraphics[width=\linewidth]{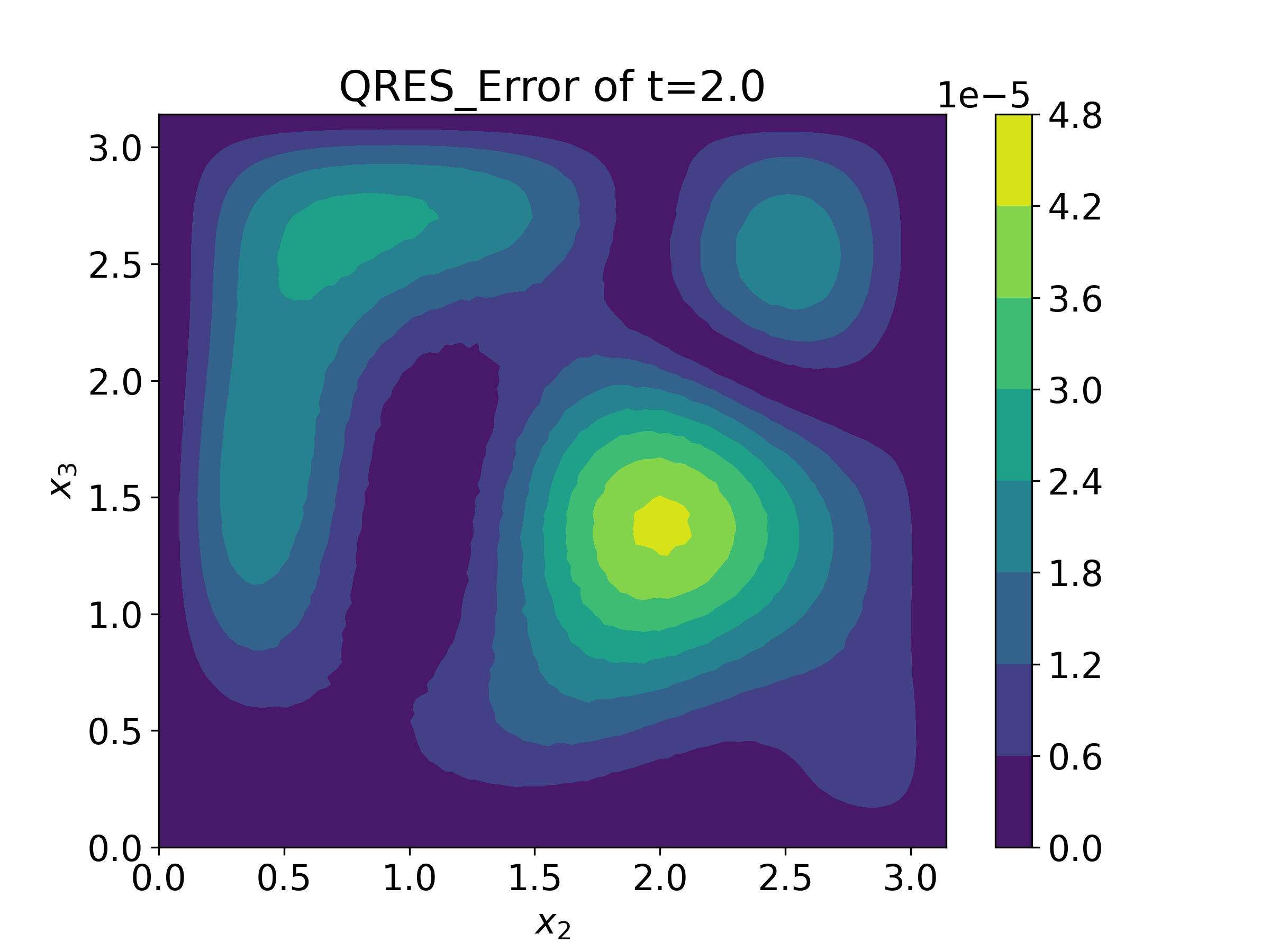}
		
		\label{chutian2}
	\end{minipage}

 \begin{minipage}{0.27\linewidth}
		\centering
		\includegraphics[width=\linewidth]{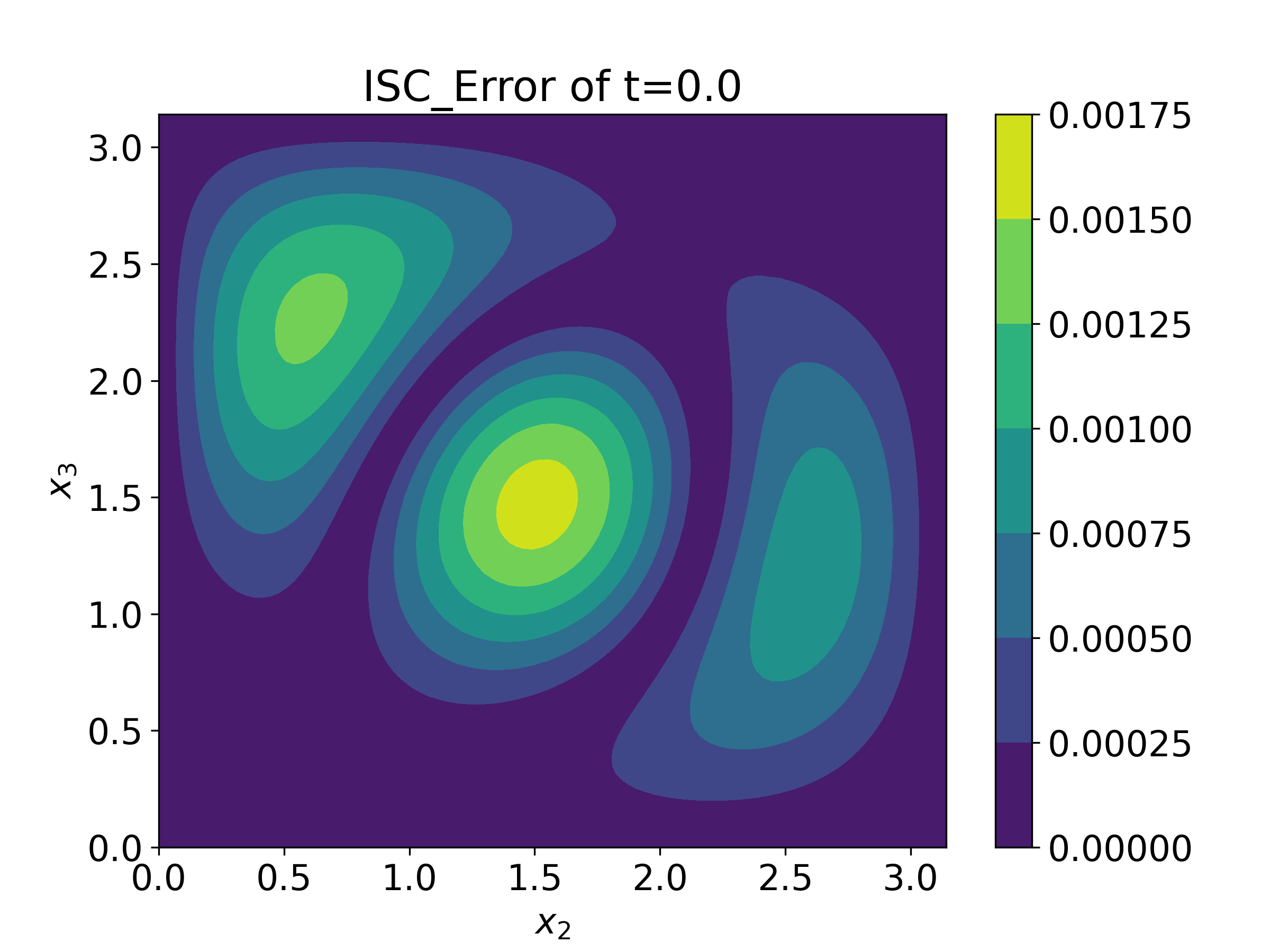}
		
		\label{chutian3}
	\end{minipage}
	\begin{minipage}{0.27\linewidth}
		\centering
		\includegraphics[width=\linewidth]{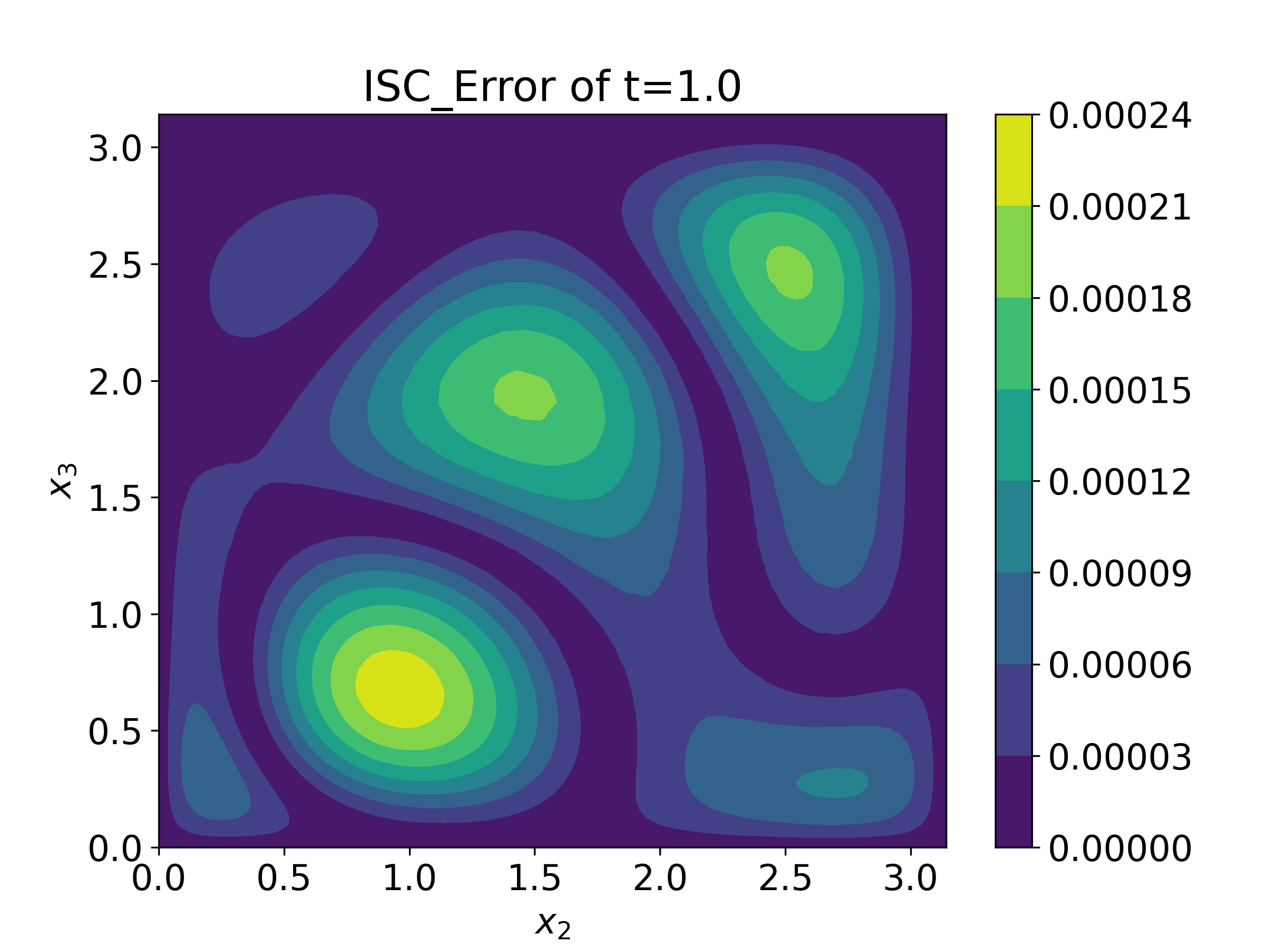}
		
		\label{chutian4}
	\end{minipage}
 \begin{minipage}{0.27\linewidth}
		\centering
		\includegraphics[width=\linewidth]{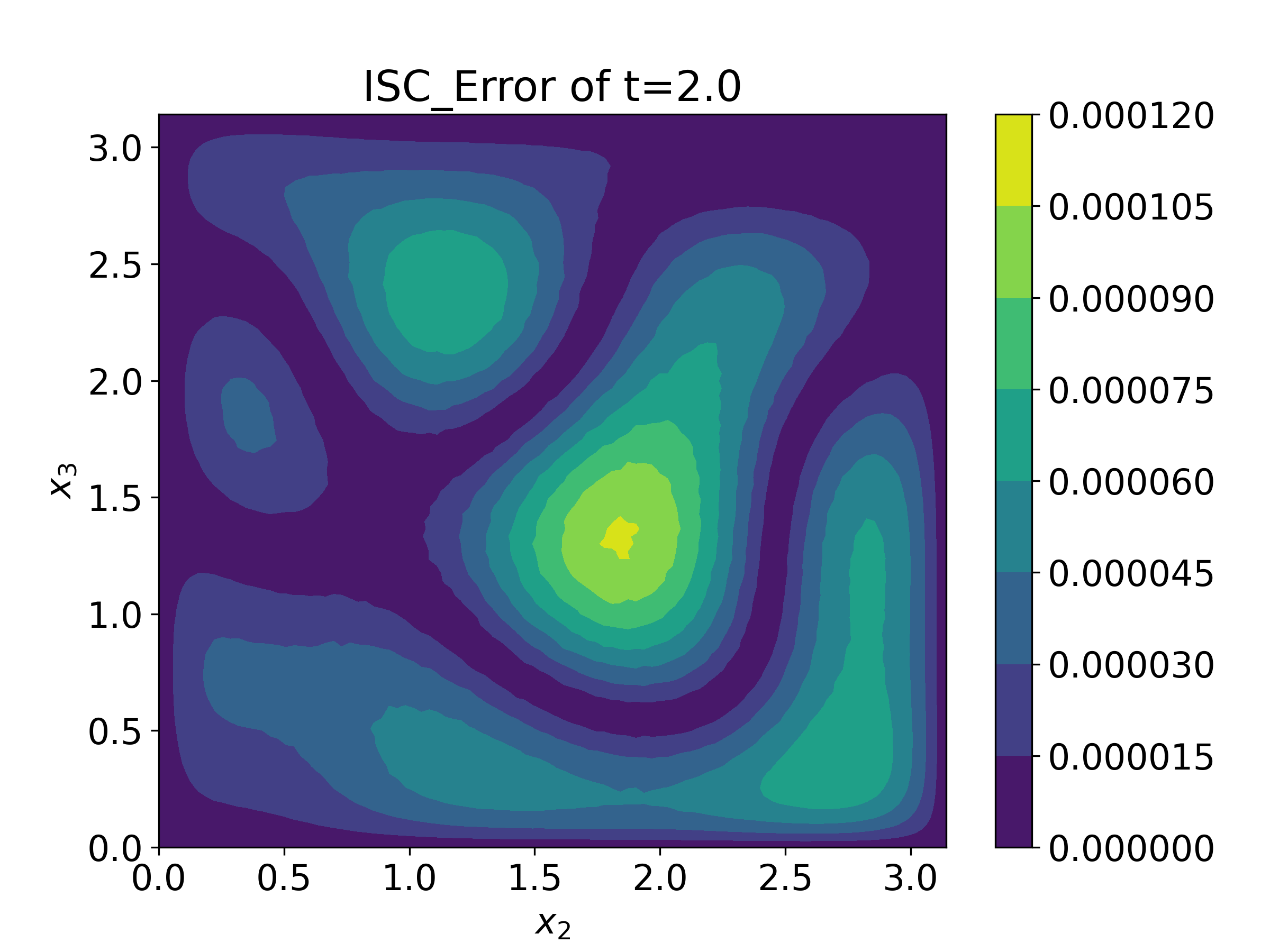}
		
		\label{chutian2}
	\end{minipage}

 \begin{minipage}{0.27\linewidth}
		\centering
		\includegraphics[width=\linewidth]{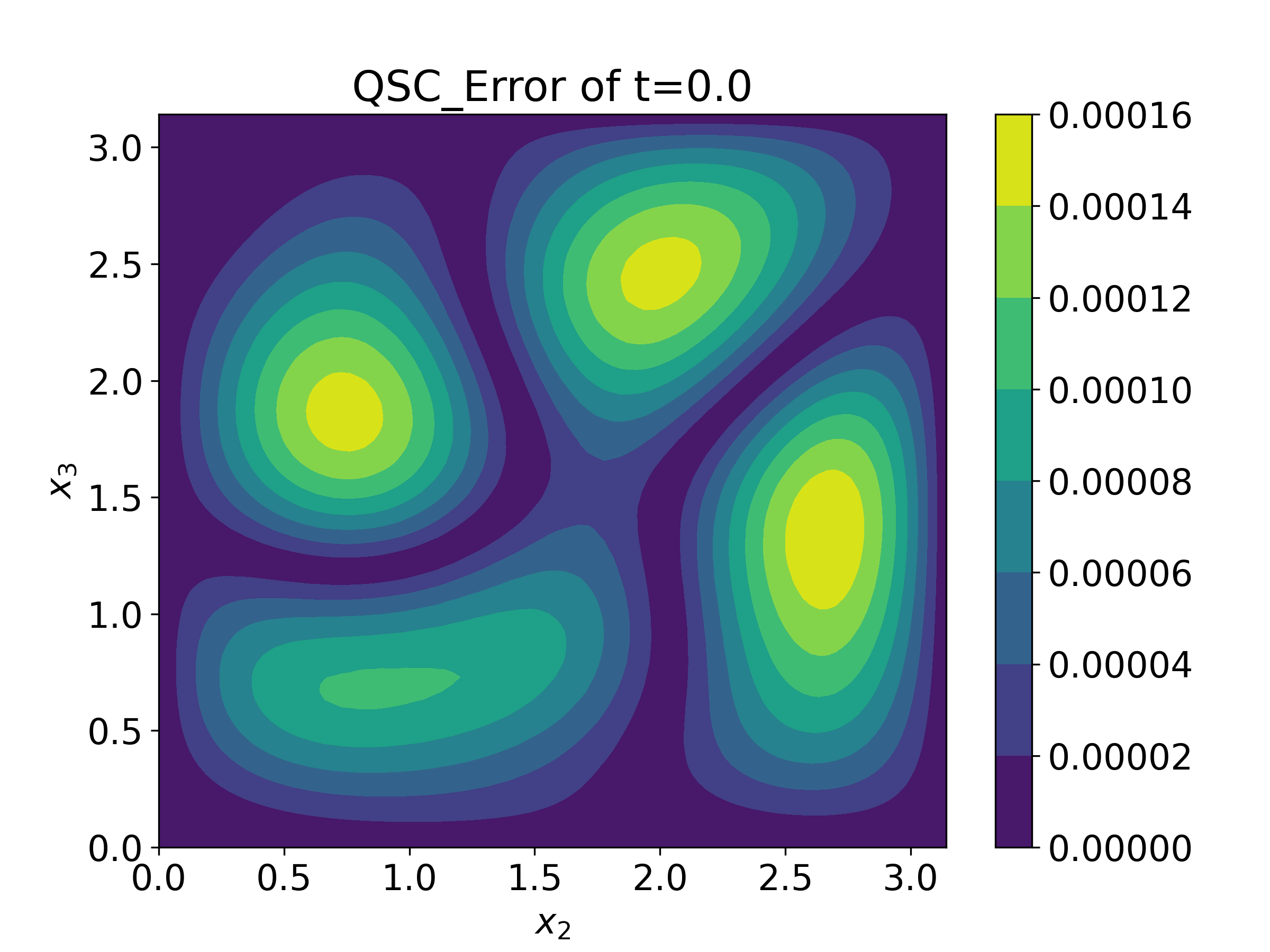}
		
		\label{chutian3}
	\end{minipage}
	\begin{minipage}{0.27\linewidth}
		\centering
		\includegraphics[width=\linewidth]{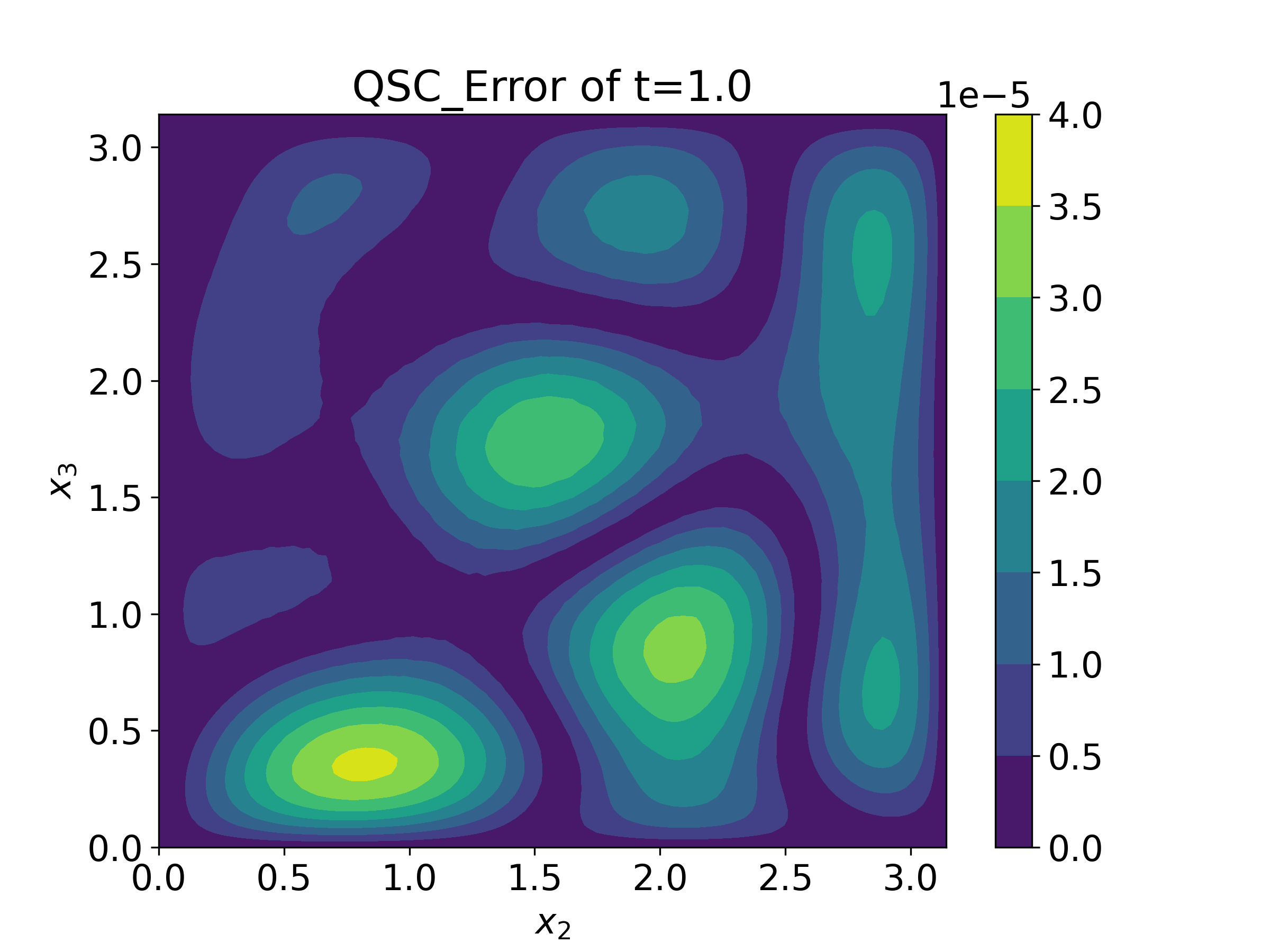}
		
		\label{chutian4}
	\end{minipage}
 \begin{minipage}{0.27\linewidth}
		\centering
		\includegraphics[width=\linewidth]{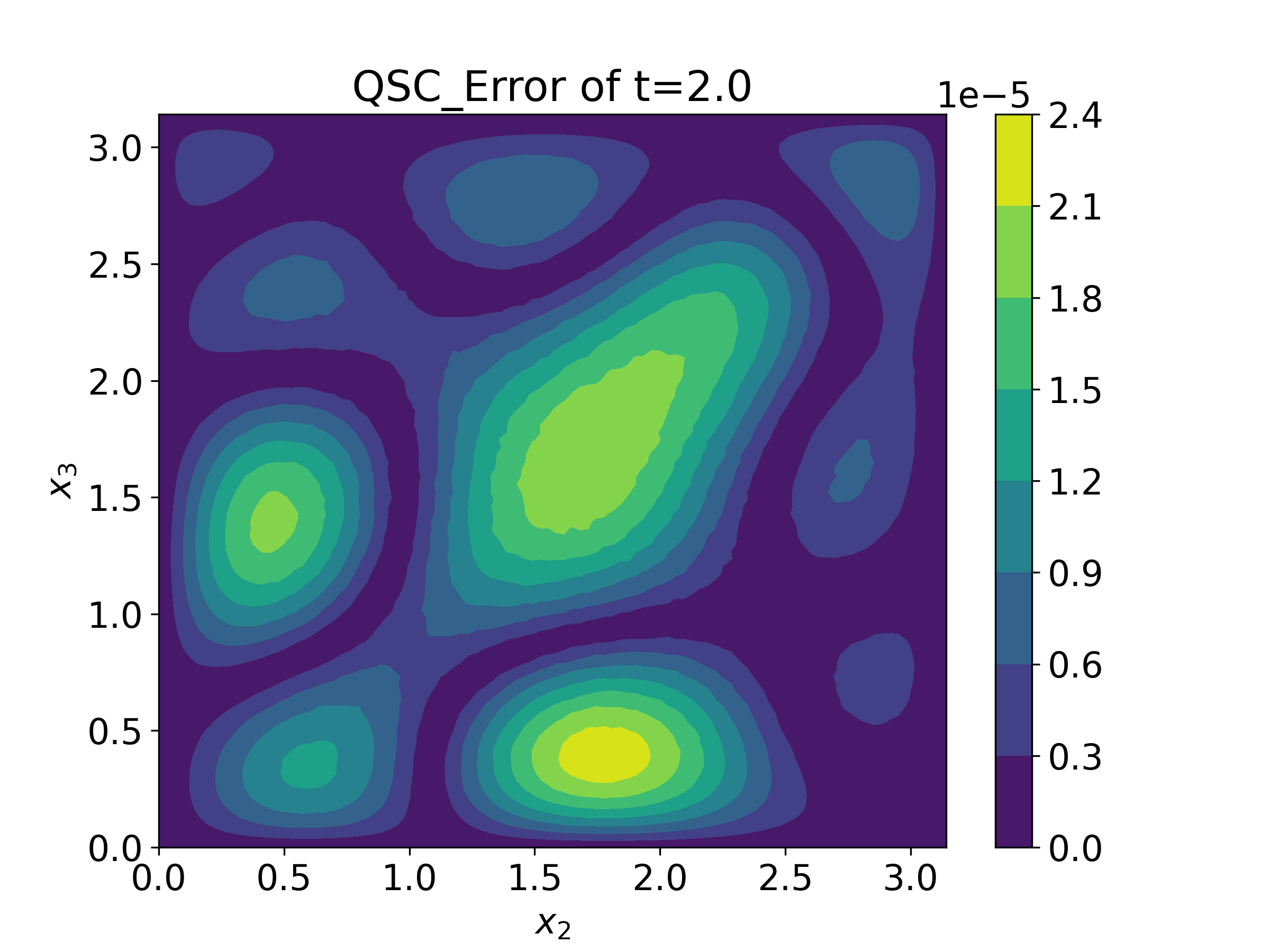}
		
		\label{chutian2}
	\end{minipage}

 \begin{minipage}{0.27\linewidth}
		\centering
		\includegraphics[width=\linewidth]{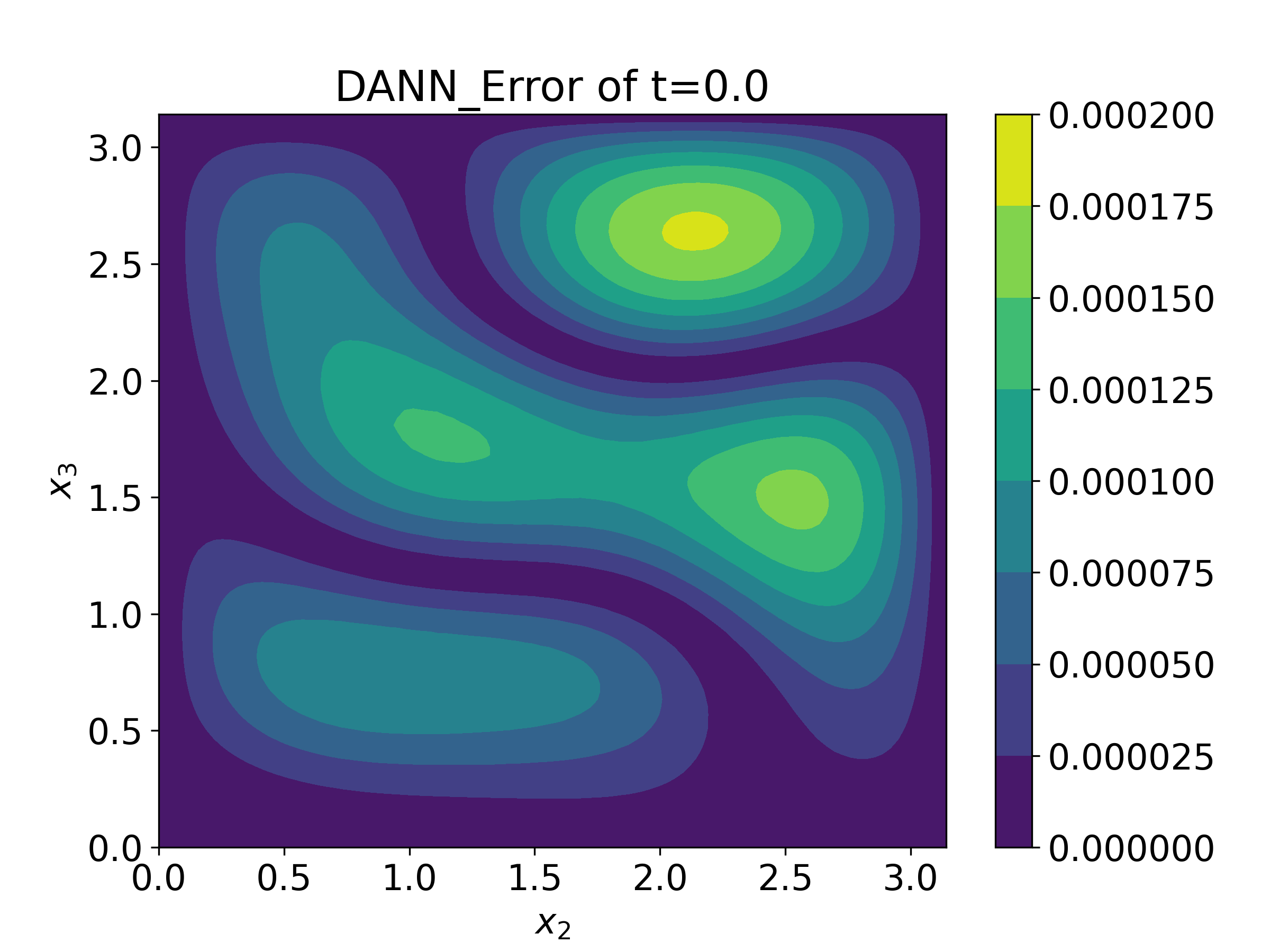}
		
		\label{chutian3}
	\end{minipage}
	\begin{minipage}{0.27\linewidth}
		\centering
		\includegraphics[width=\linewidth]{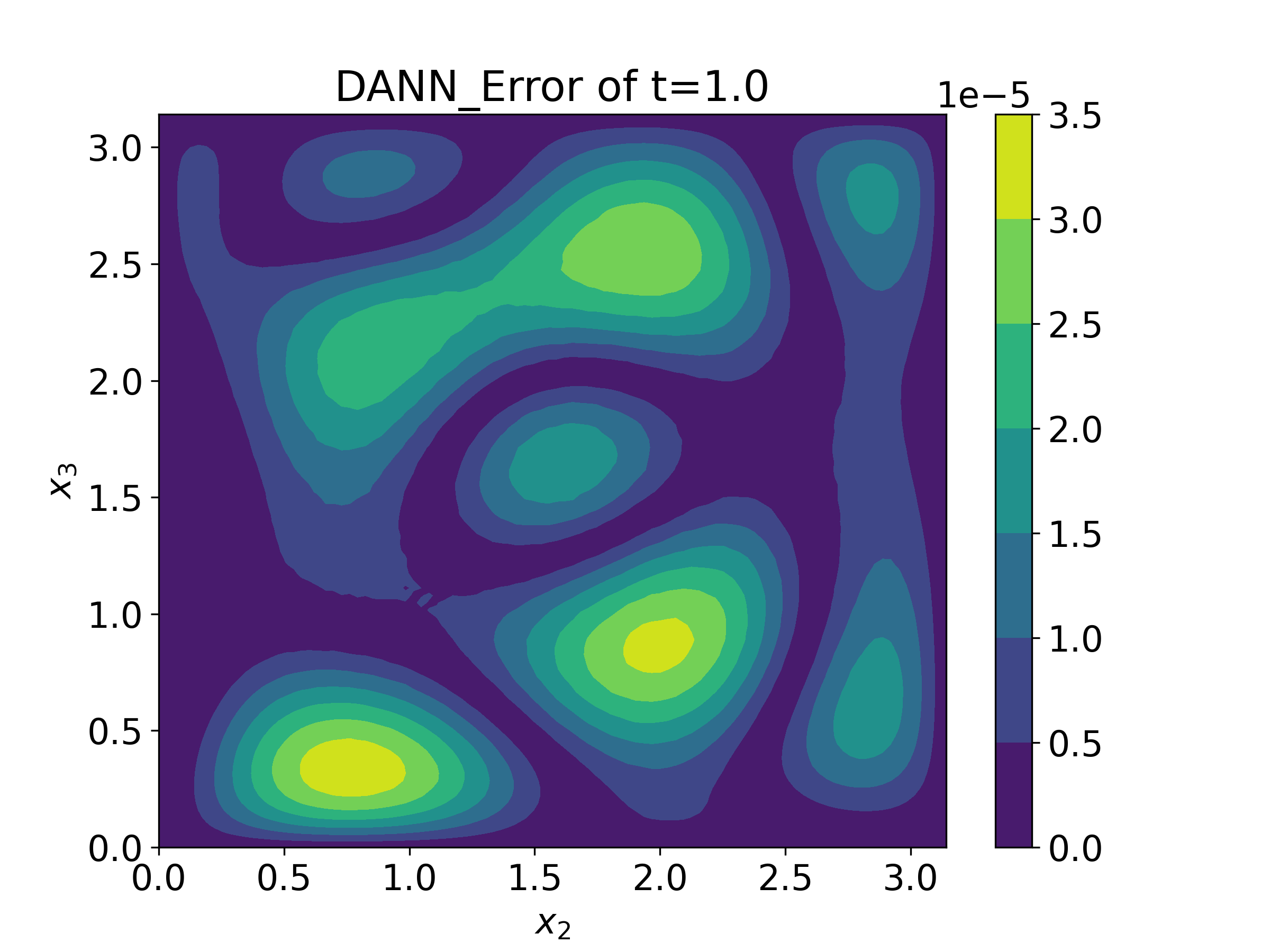}
		
		\label{chutian4}
	\end{minipage}
 \begin{minipage}{0.27\linewidth}
		\centering
		\includegraphics[width=\linewidth]{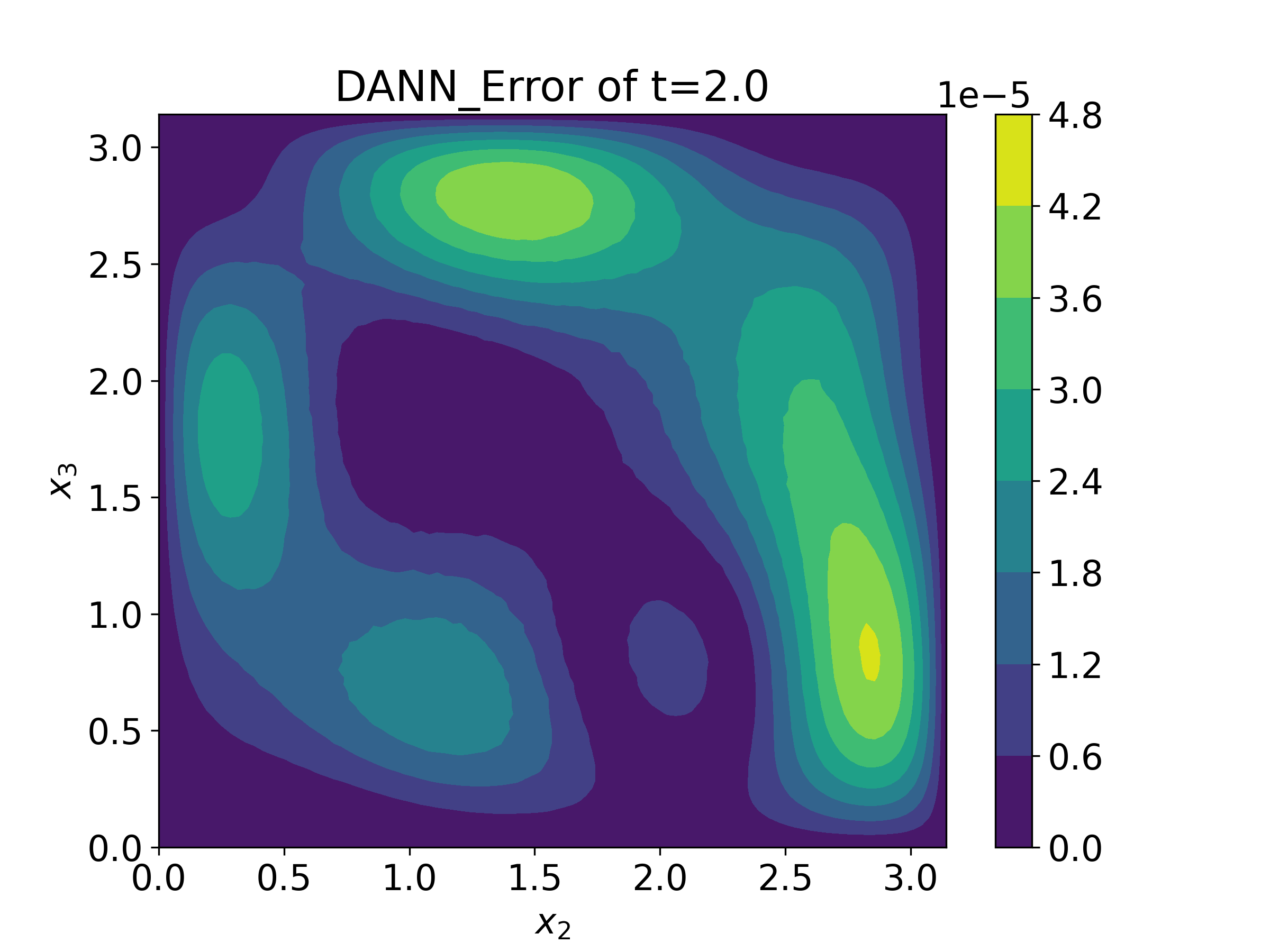}
		
		\label{chutian2}
	\end{minipage}
\caption{The absolute errors obtained by different models based on piecewise fitting at $t=0.0, 1.0, 2.0$ when $d=3$.}
\label{delay5wucha}
\end{figure}
\begin{figure}[H]
	\centering
 \begin{minipage}{0.27\linewidth}
		\centering
		\includegraphics[width=\linewidth]{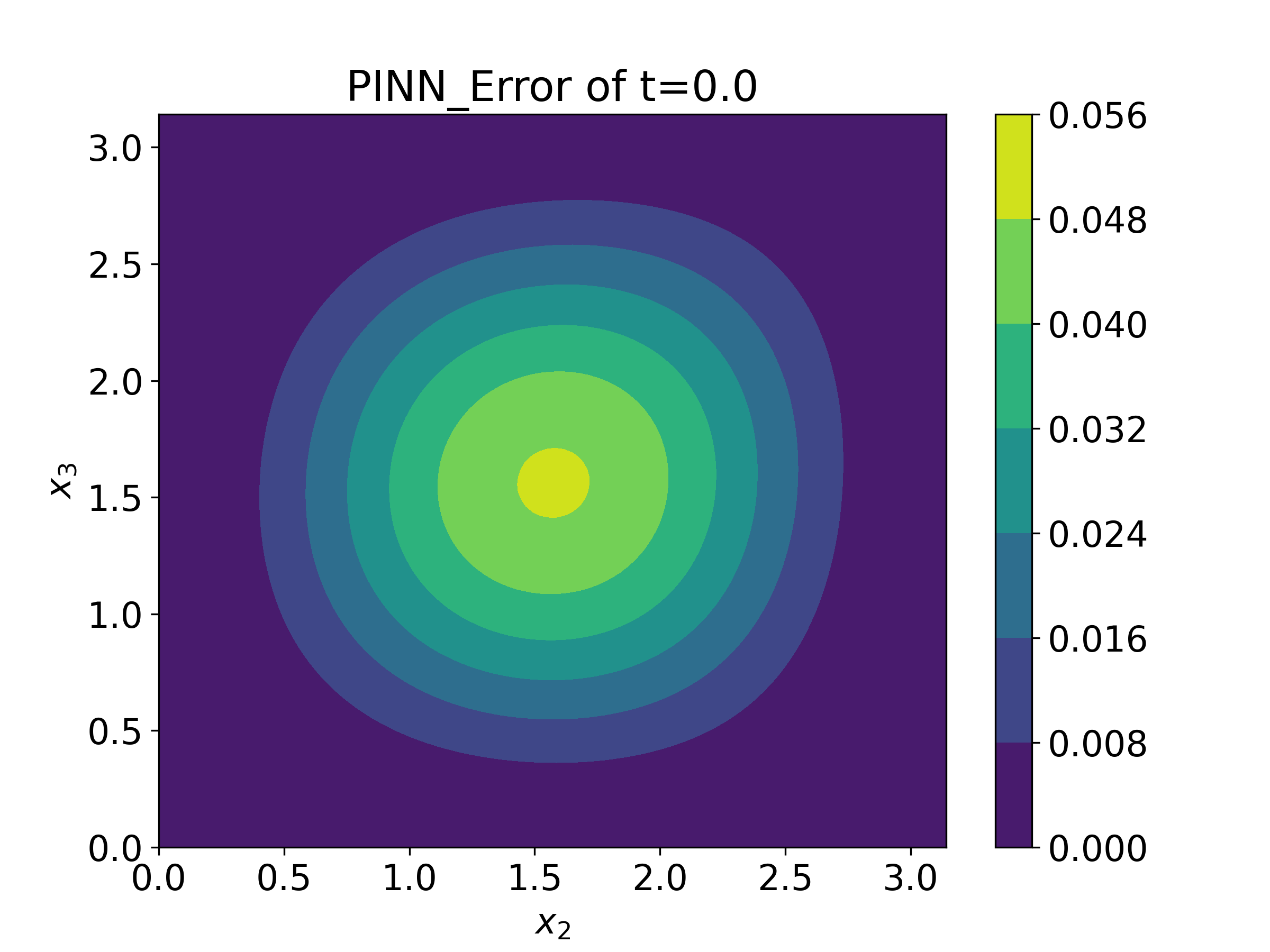}
		
		\label{chutian3}
	\end{minipage}
	\begin{minipage}{0.27\linewidth}
		\centering
		\includegraphics[width=\linewidth]{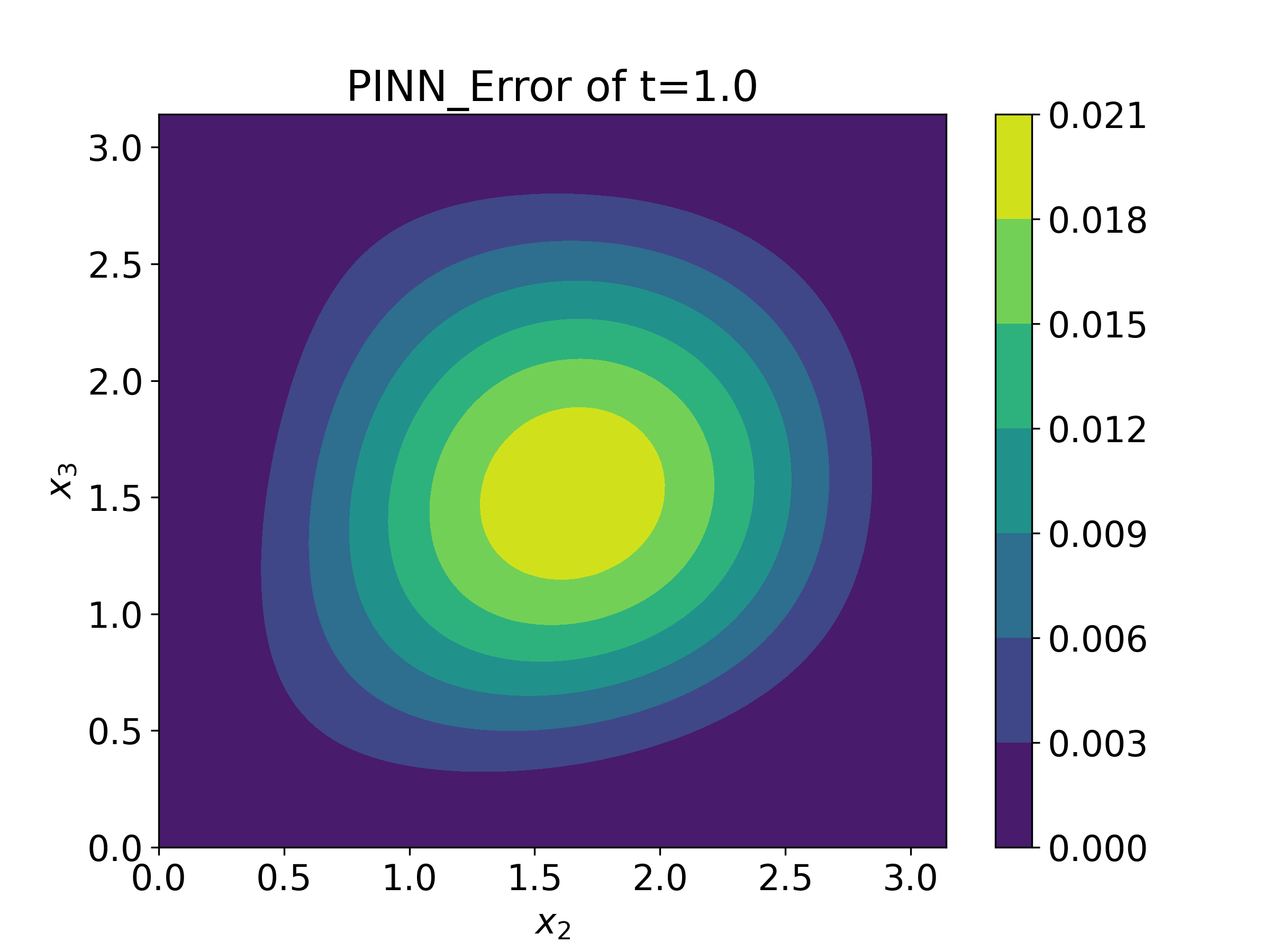}
		
		\label{chutian4}
	\end{minipage}
 \begin{minipage}{0.27\linewidth}
		\centering
		\includegraphics[width=\linewidth]{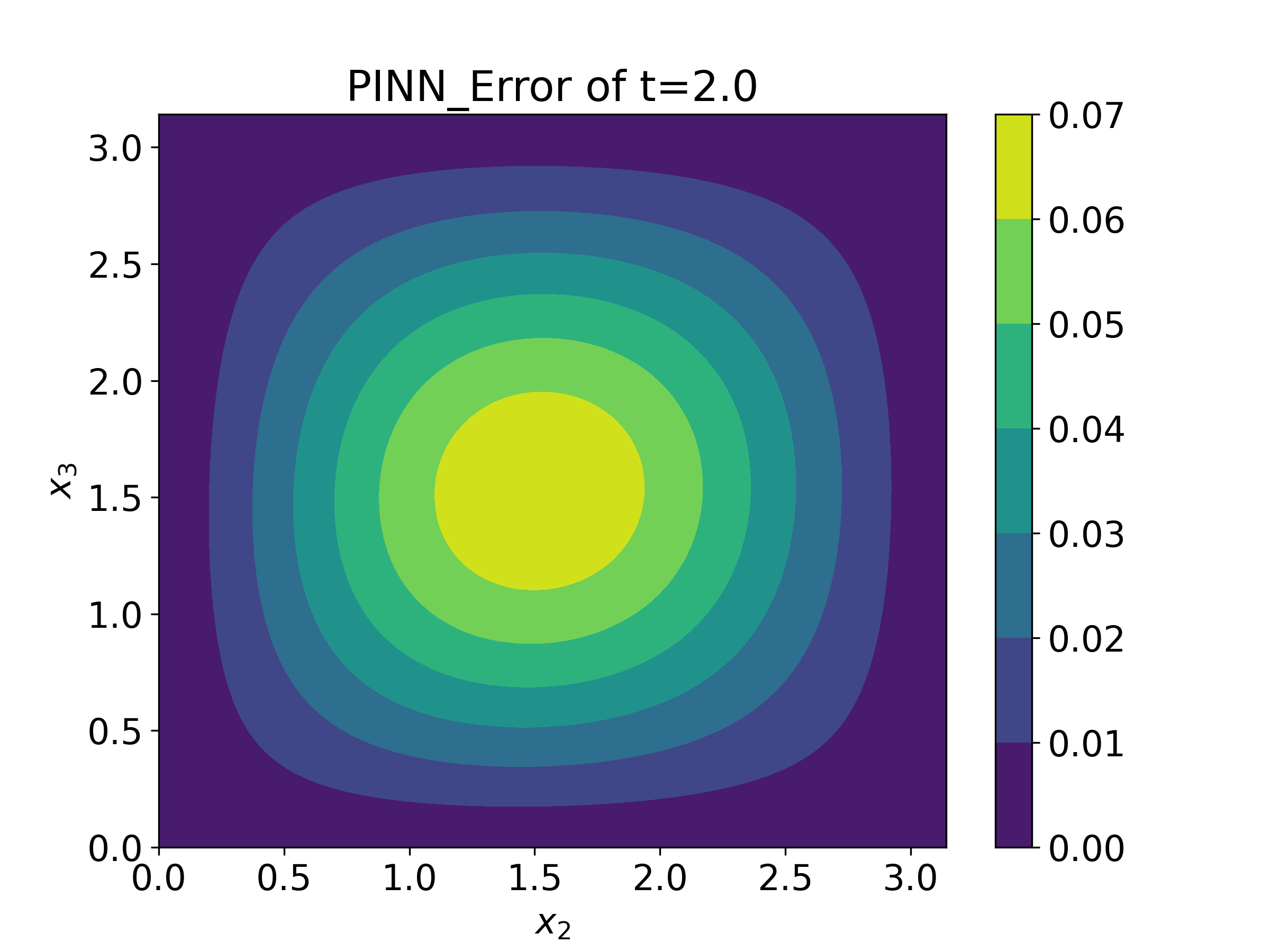}
		
		\label{chutian2}
	\end{minipage}

 \begin{minipage}{0.27\linewidth}
		\centering
		\includegraphics[width=\linewidth]{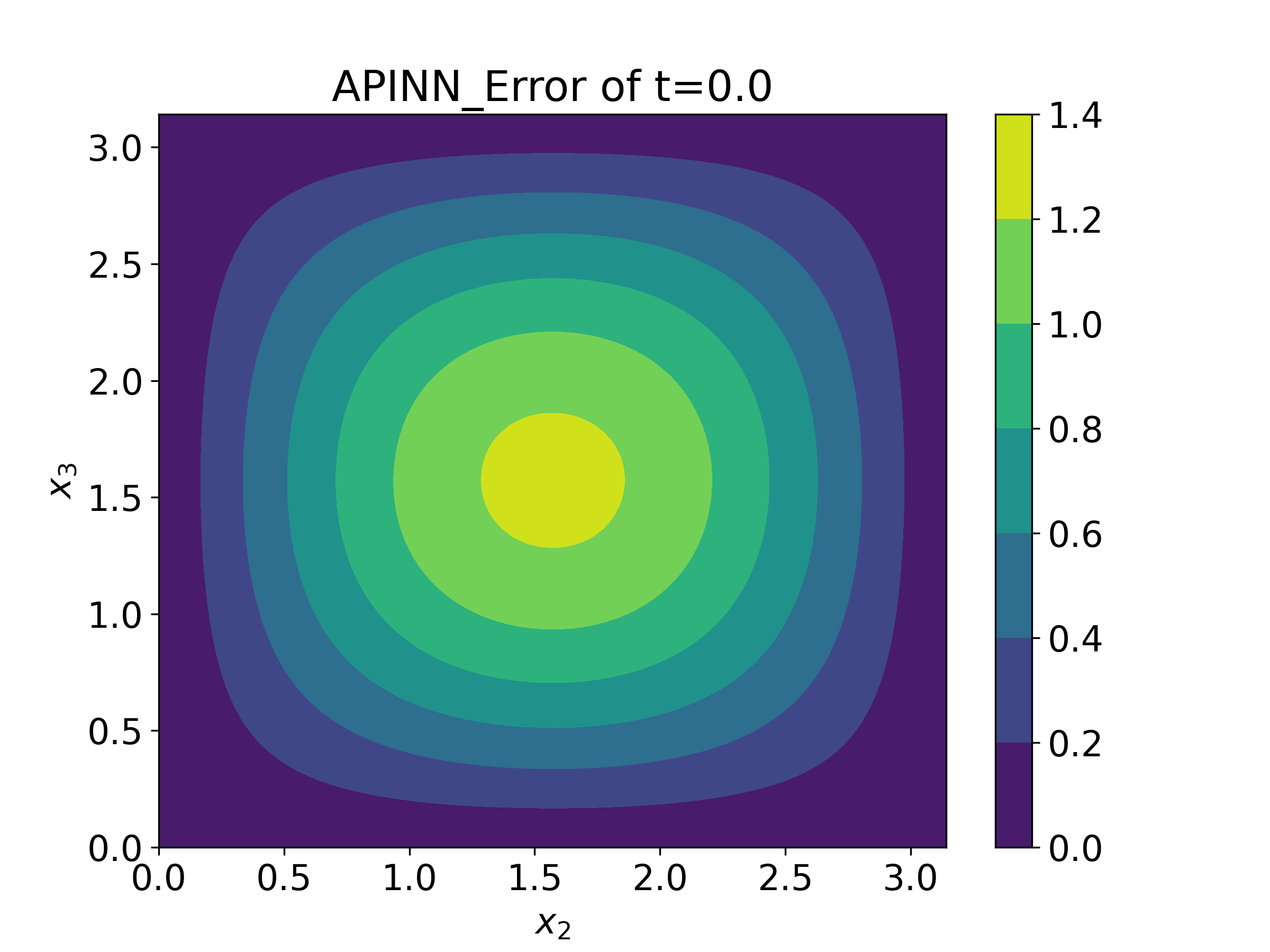}
		
		\label{chutian3}
	\end{minipage}
	\begin{minipage}{0.27\linewidth}
		\centering
		\includegraphics[width=\linewidth]{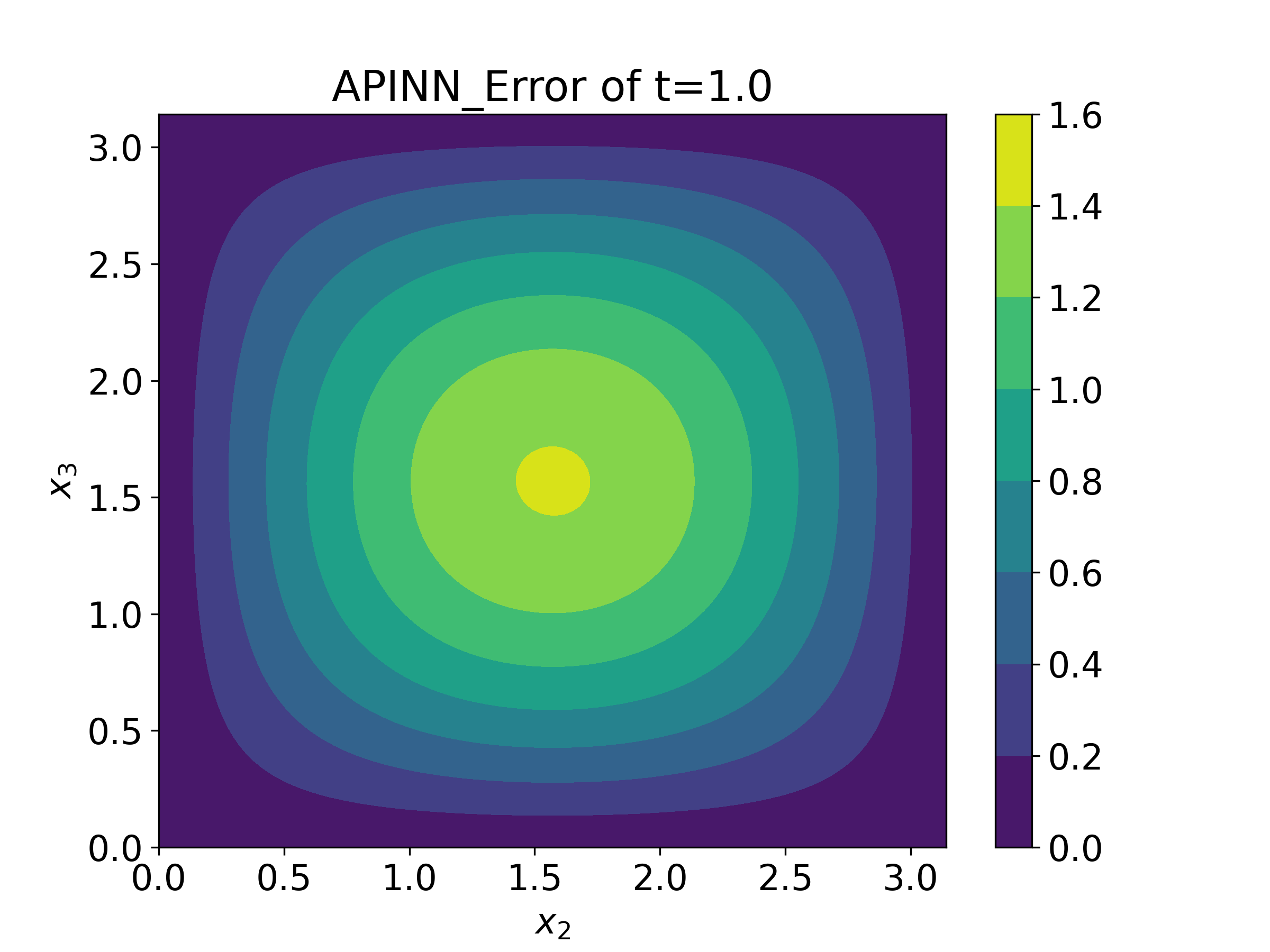}
		
		\label{chutian4}
	\end{minipage}
 \begin{minipage}{0.27\linewidth}
		\centering
		\includegraphics[width=\linewidth]{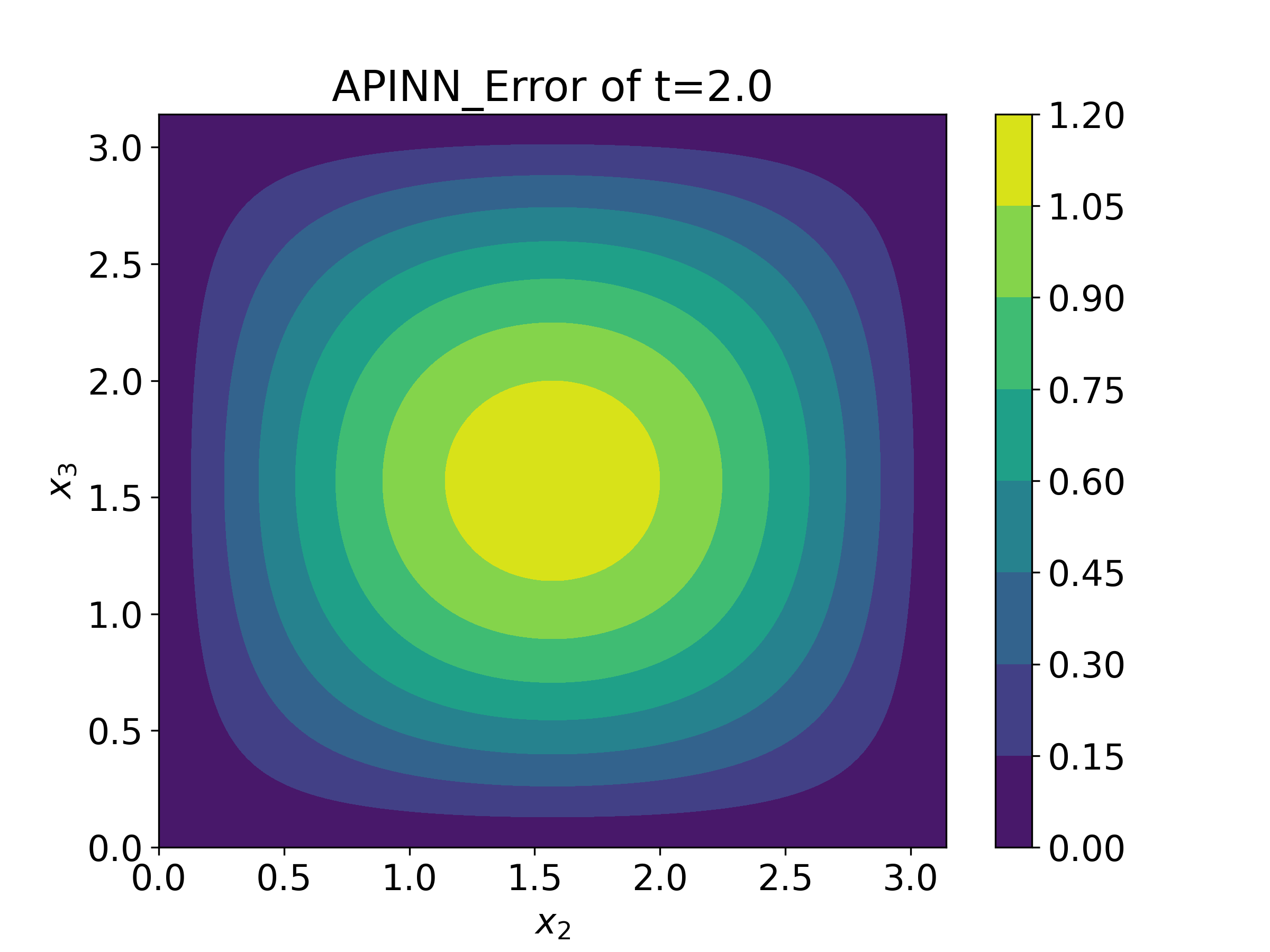}
		
		\label{chutian2}
	\end{minipage}

 \begin{minipage}{0.27\linewidth}
		\centering
		\includegraphics[width=\linewidth]{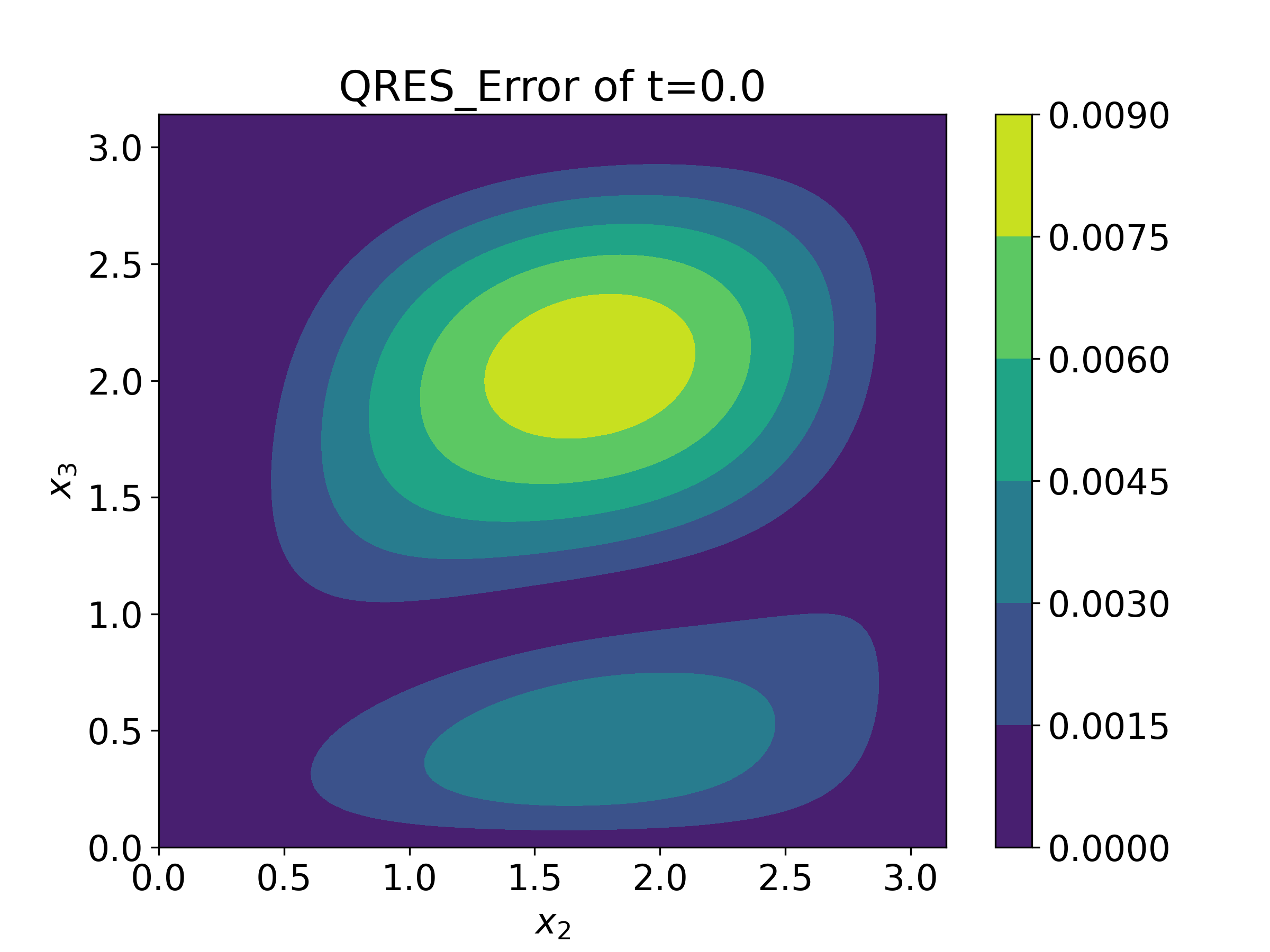}
		
		\label{chutian3}
	\end{minipage}
	\begin{minipage}{0.27\linewidth}
		\centering
		\includegraphics[width=\linewidth]{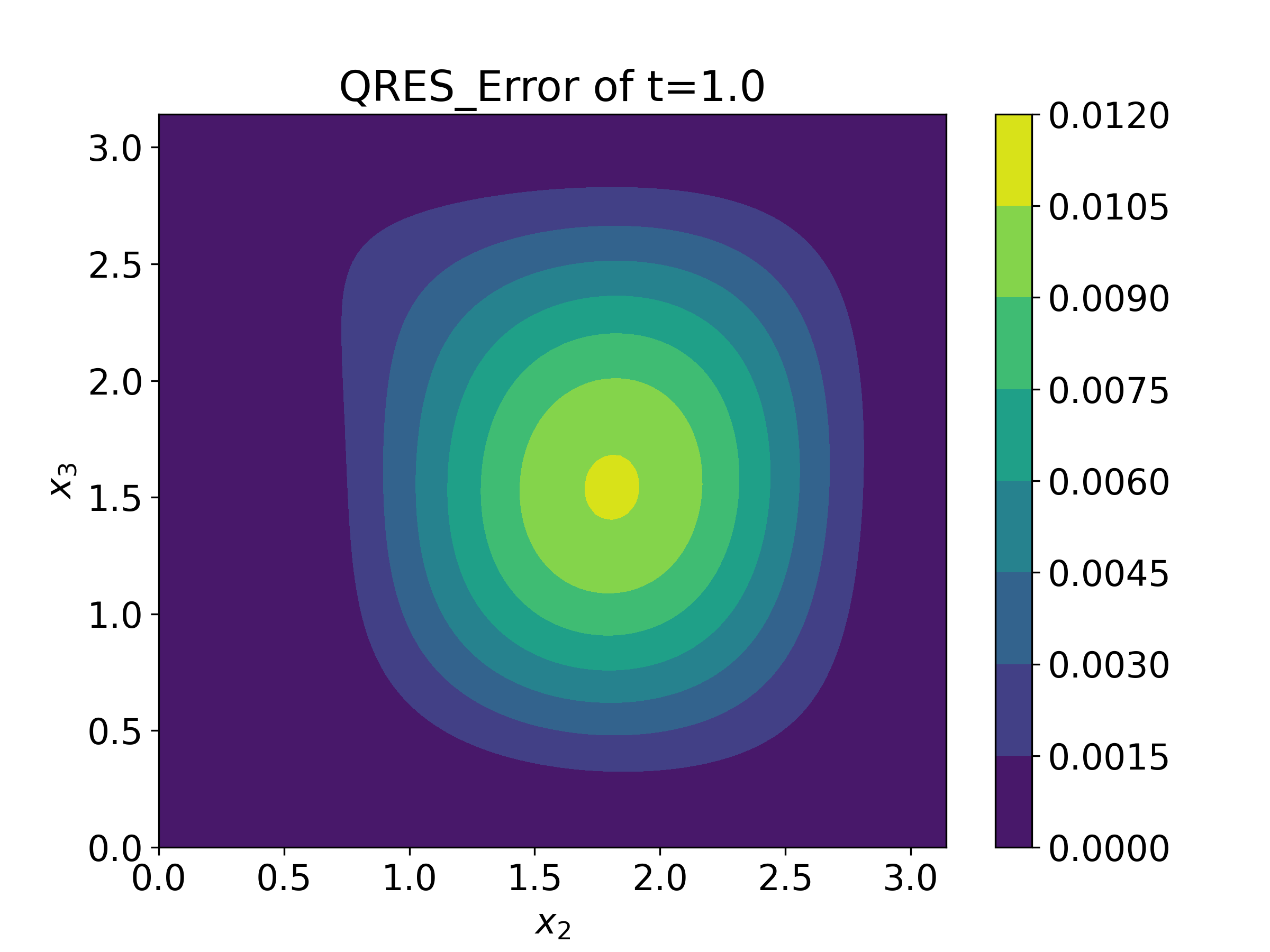}
		
		\label{chutian4}
	\end{minipage}
 \begin{minipage}{0.27\linewidth}
		\centering
		\includegraphics[width=\linewidth]{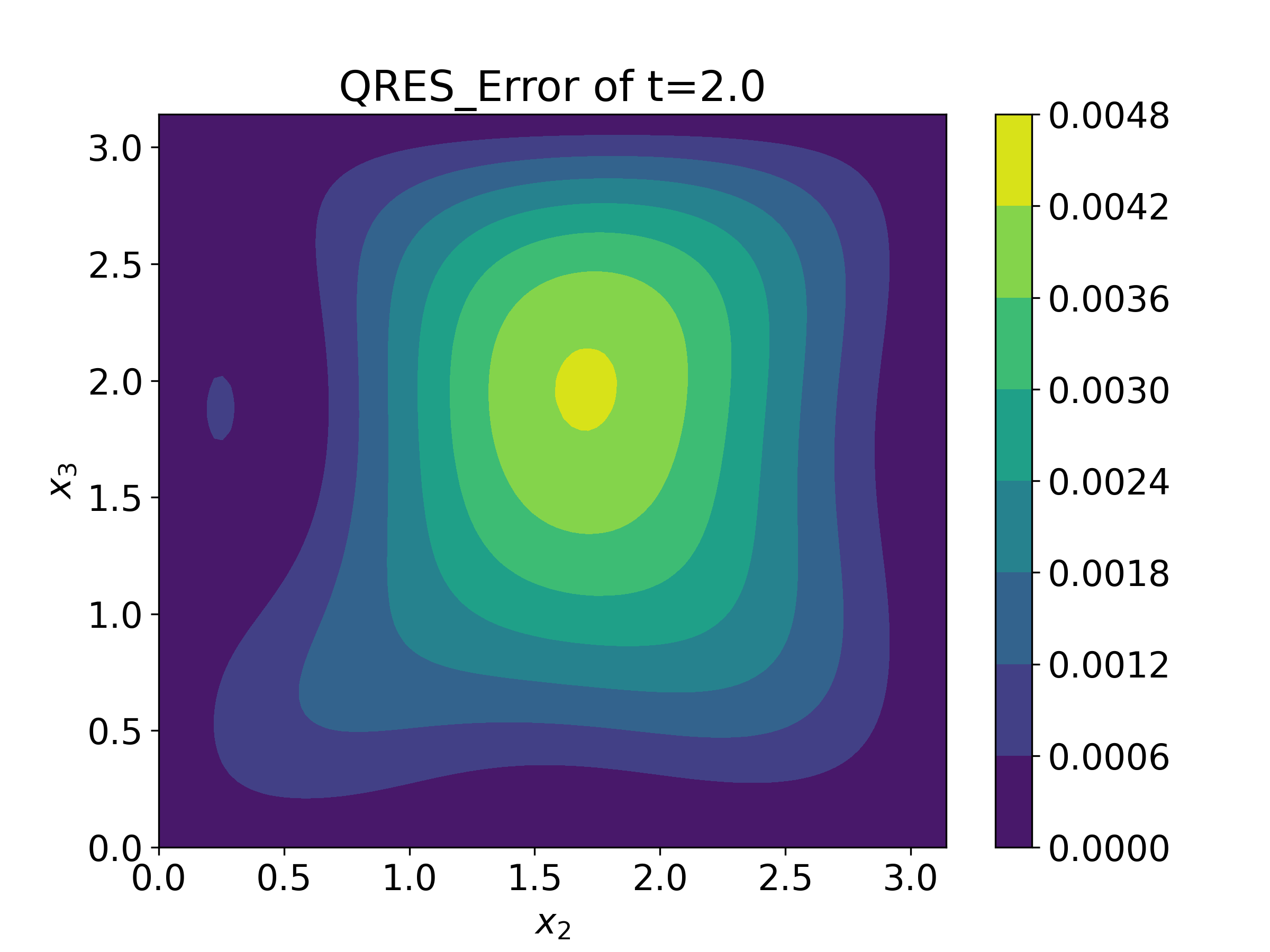}
		
		\label{chutian2}
	\end{minipage}

 \begin{minipage}{0.27\linewidth}
		\centering
		\includegraphics[width=\linewidth]{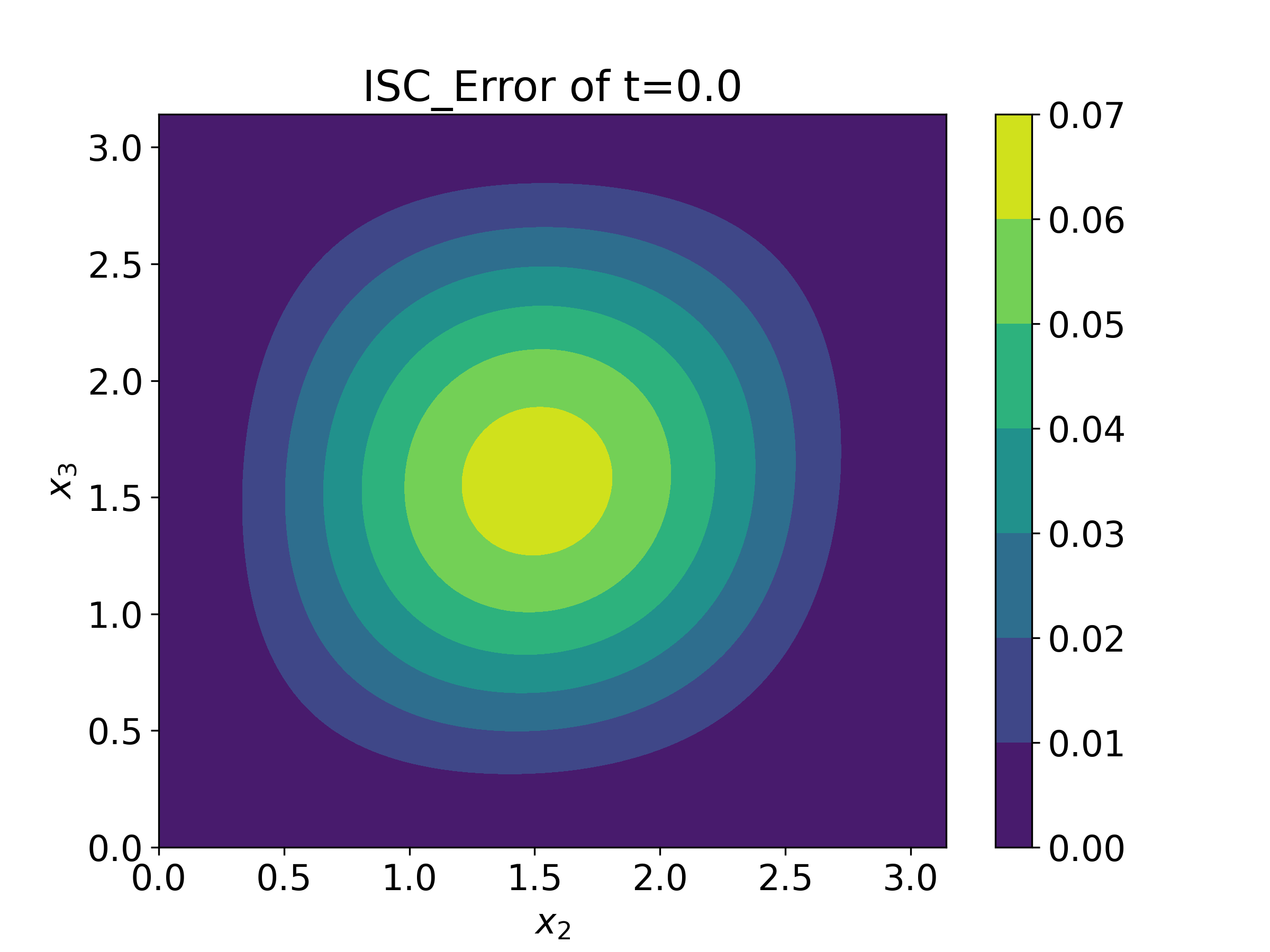}
		
		\label{chutian3}
	\end{minipage}
	\begin{minipage}{0.27\linewidth}
		\centering
		\includegraphics[width=\linewidth]{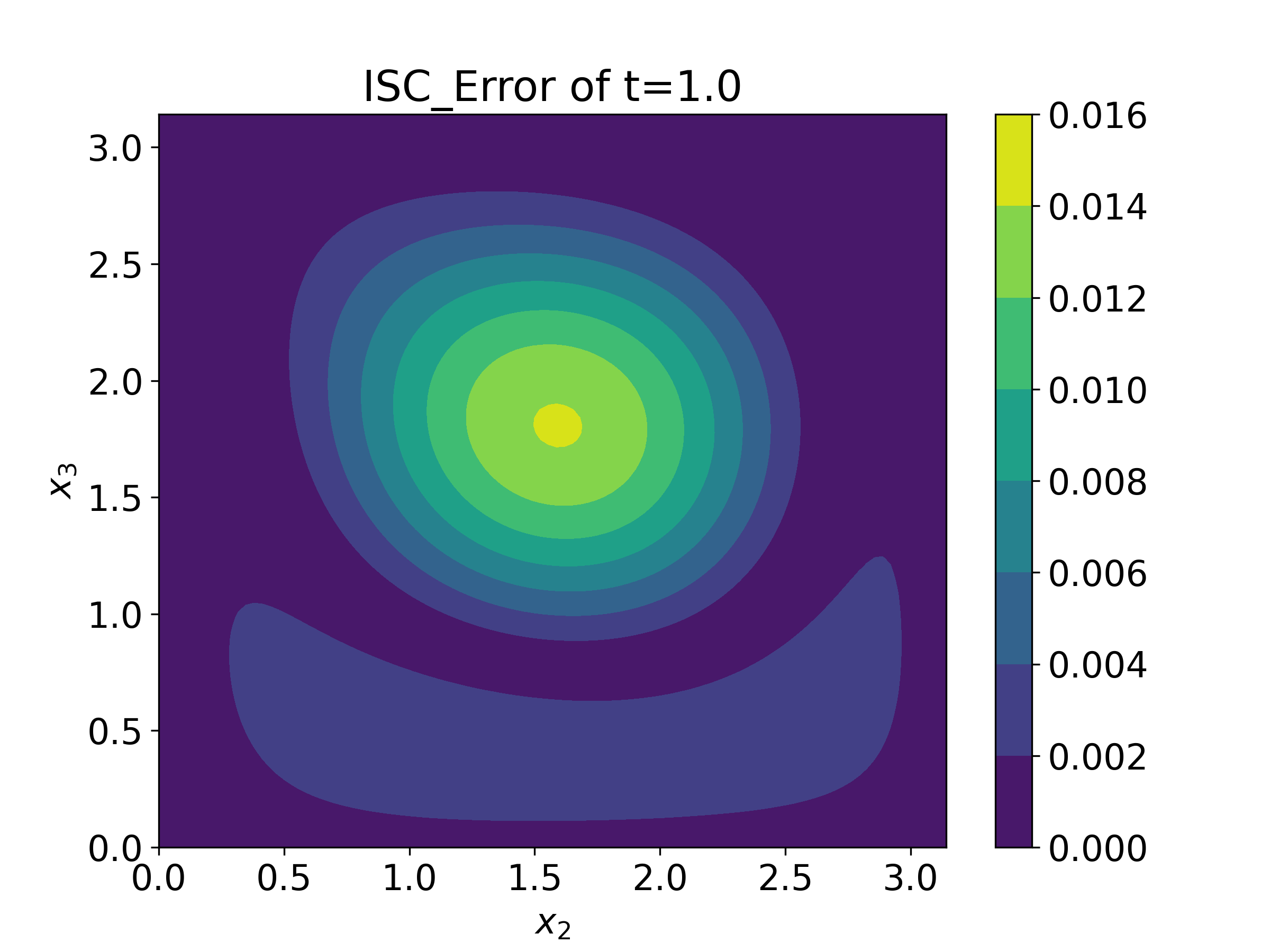}
		
		\label{chutian4}
	\end{minipage}
 \begin{minipage}{0.27\linewidth}
		\centering
		\includegraphics[width=\linewidth]{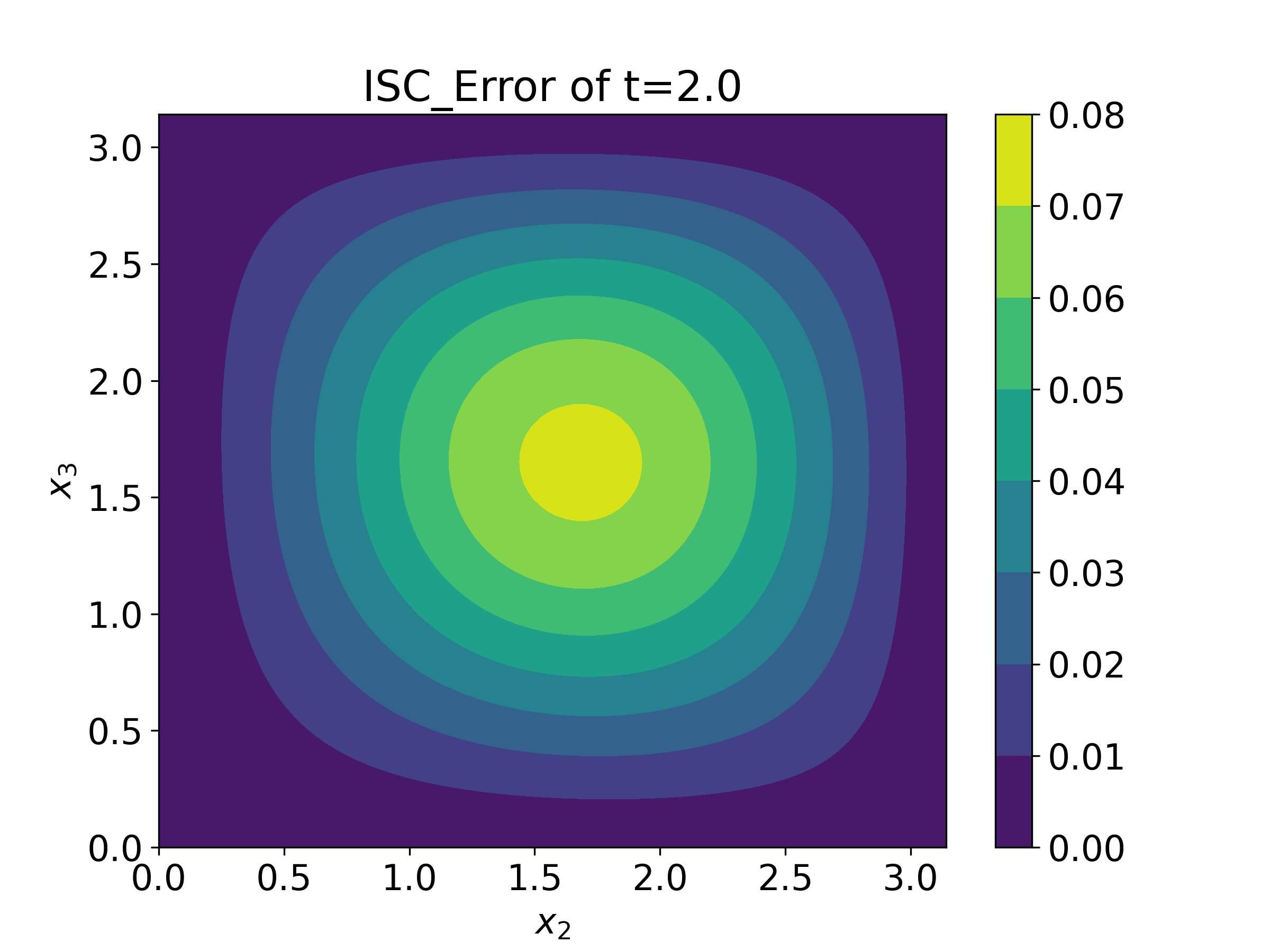}
		
		\label{chutian2}
	\end{minipage}

 \begin{minipage}{0.27\linewidth}
		\centering
		\includegraphics[width=\linewidth]{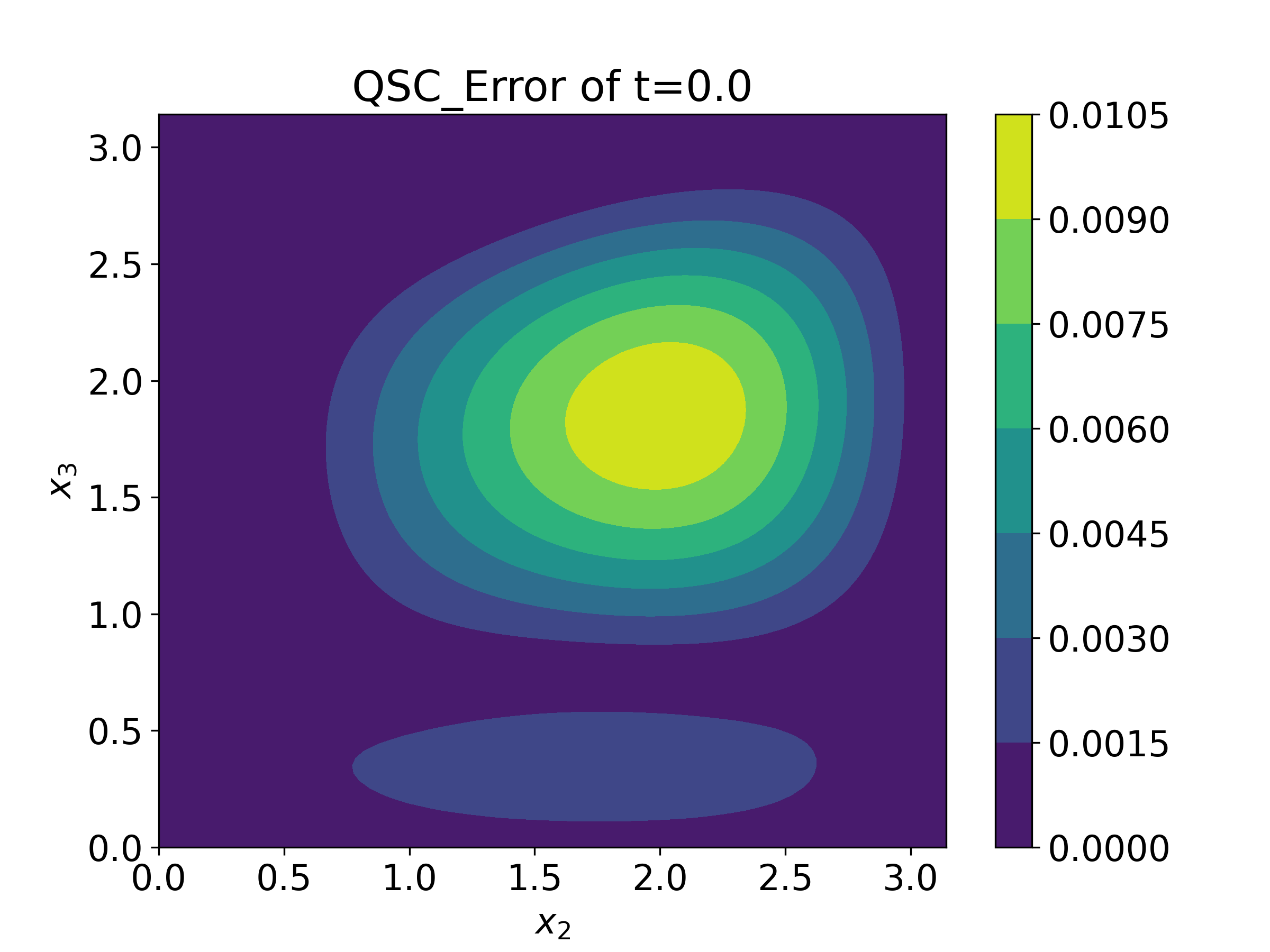}
		
		\label{chutian3}
	\end{minipage}
	\begin{minipage}{0.27\linewidth}
		\centering
		\includegraphics[width=\linewidth]{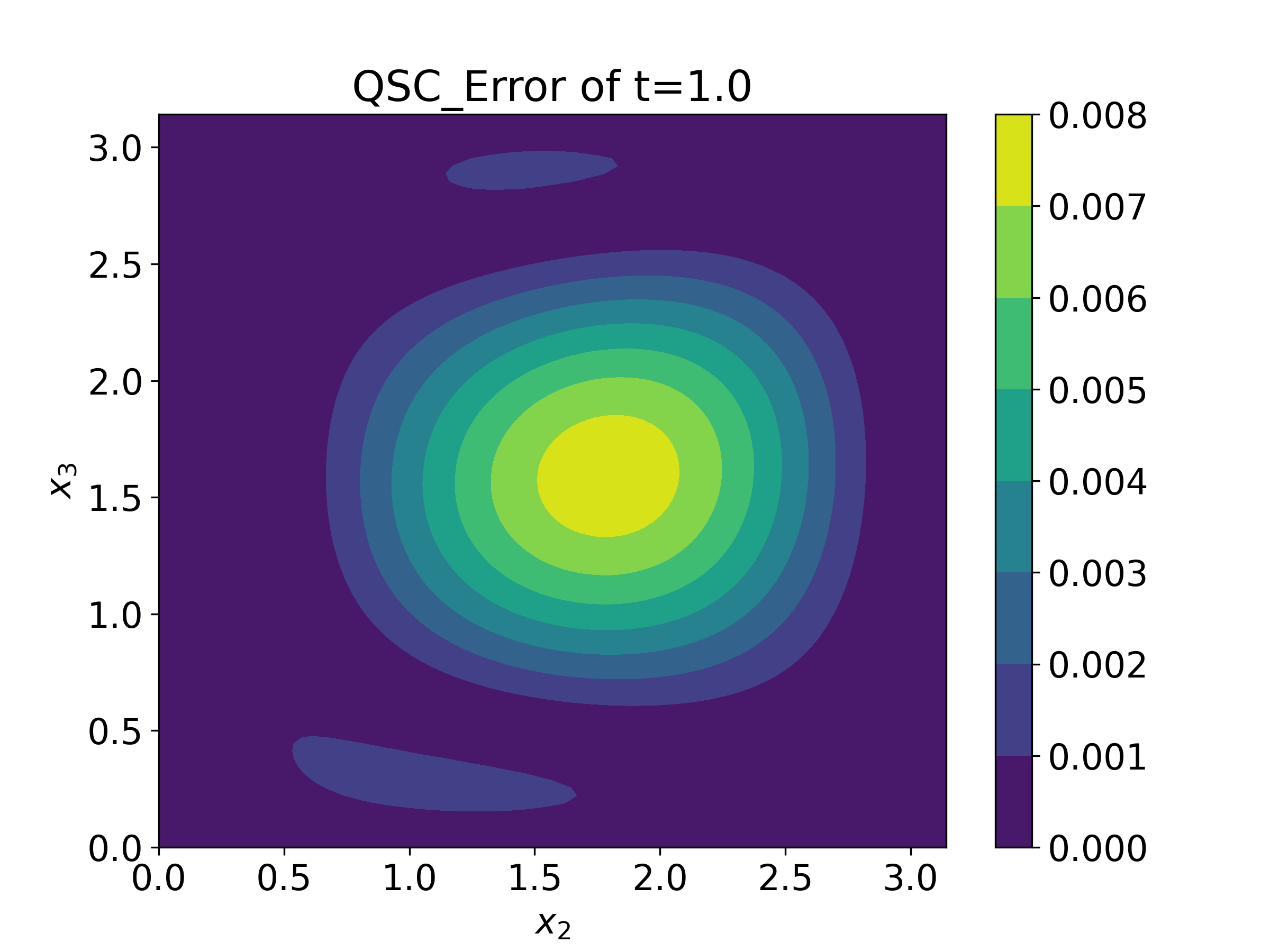}
		
		\label{chutian4}
	\end{minipage}
 \begin{minipage}{0.27\linewidth}
		\centering
		\includegraphics[width=\linewidth]{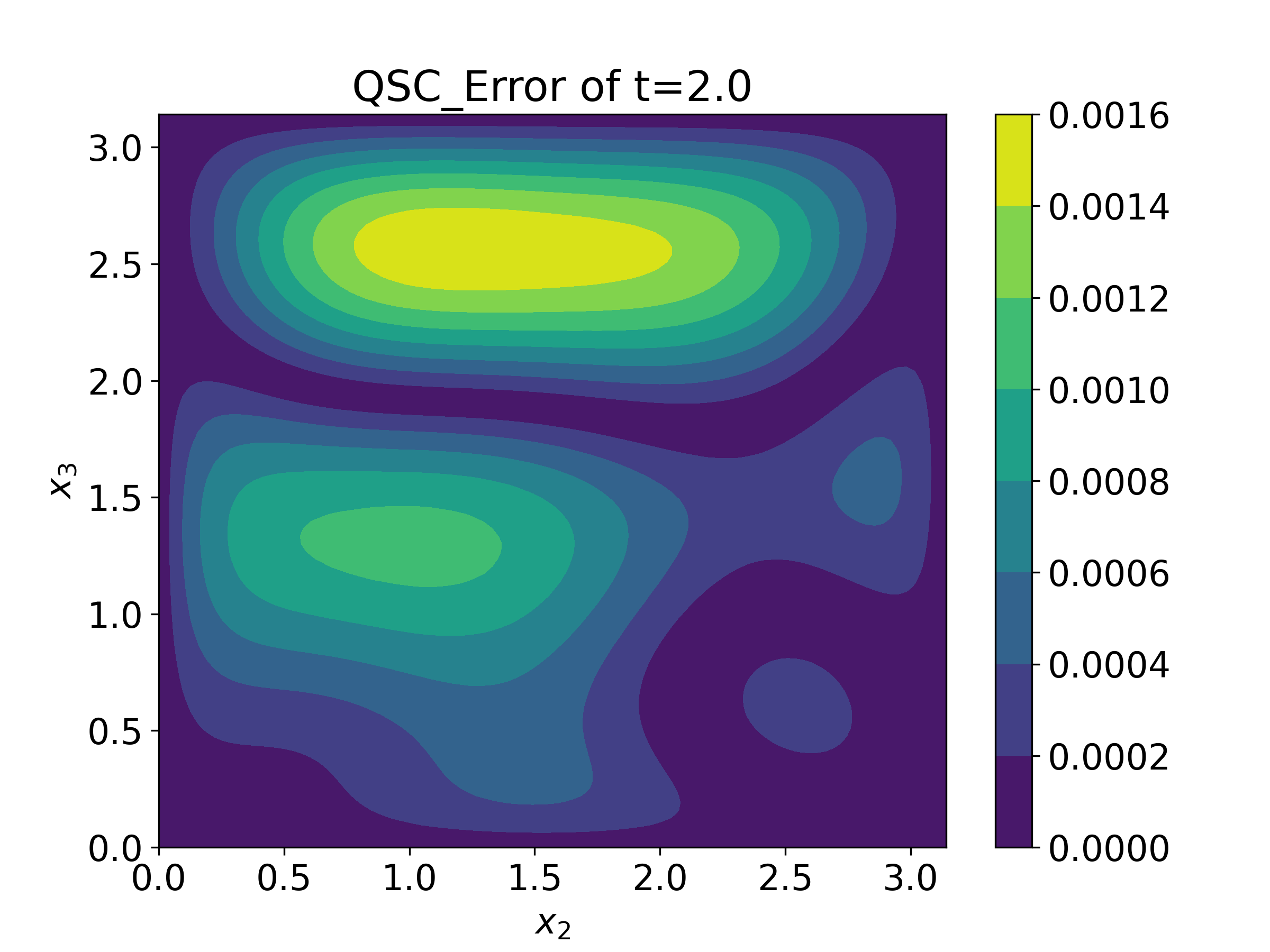}
		
		\label{chutian2}
	\end{minipage}

 \begin{minipage}{0.27\linewidth}
		\centering
		\includegraphics[width=\linewidth]{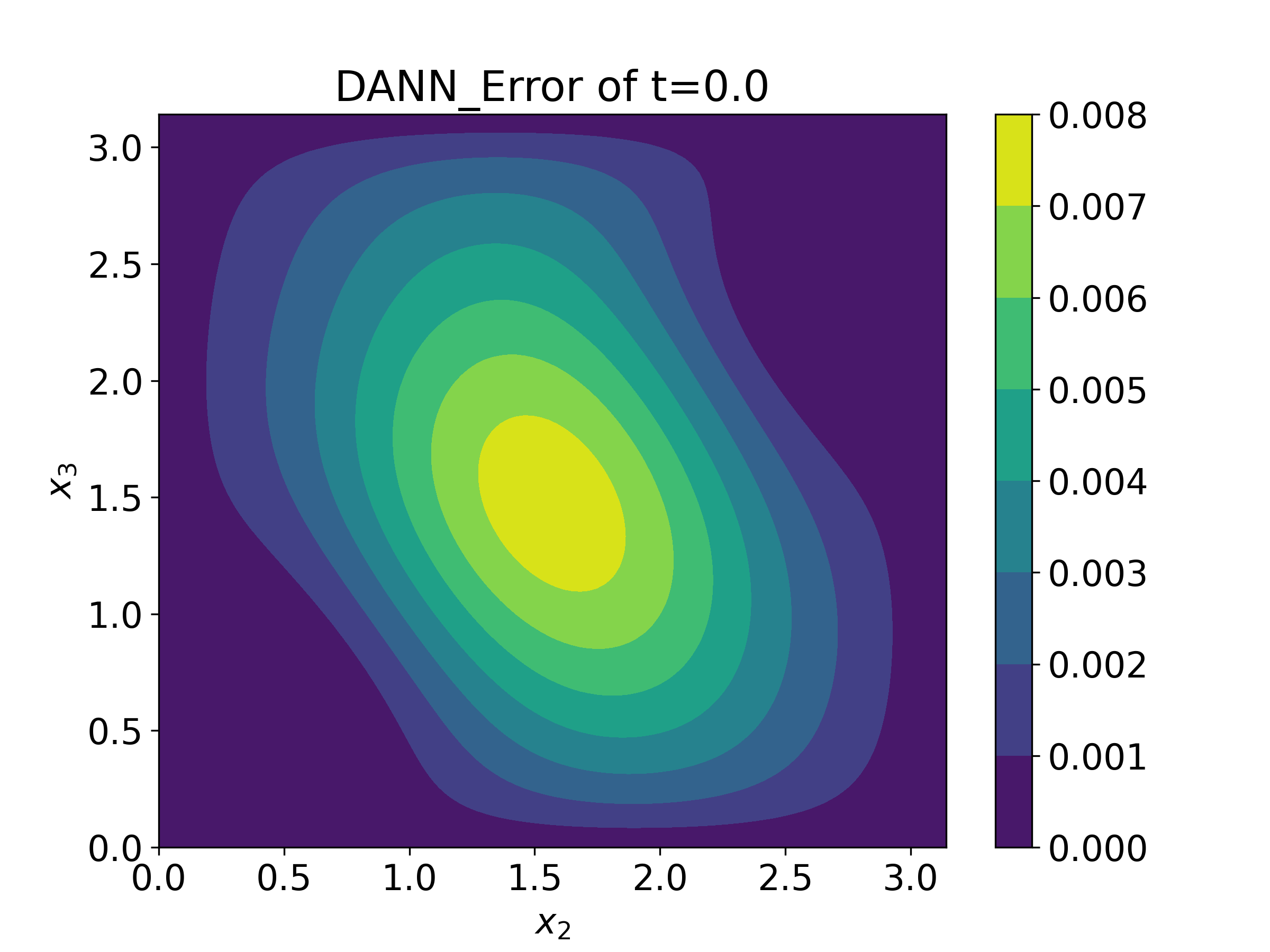}
		
		\label{chutian3}
	\end{minipage}
	\begin{minipage}{0.27\linewidth}
		\centering
		\includegraphics[width=\linewidth]{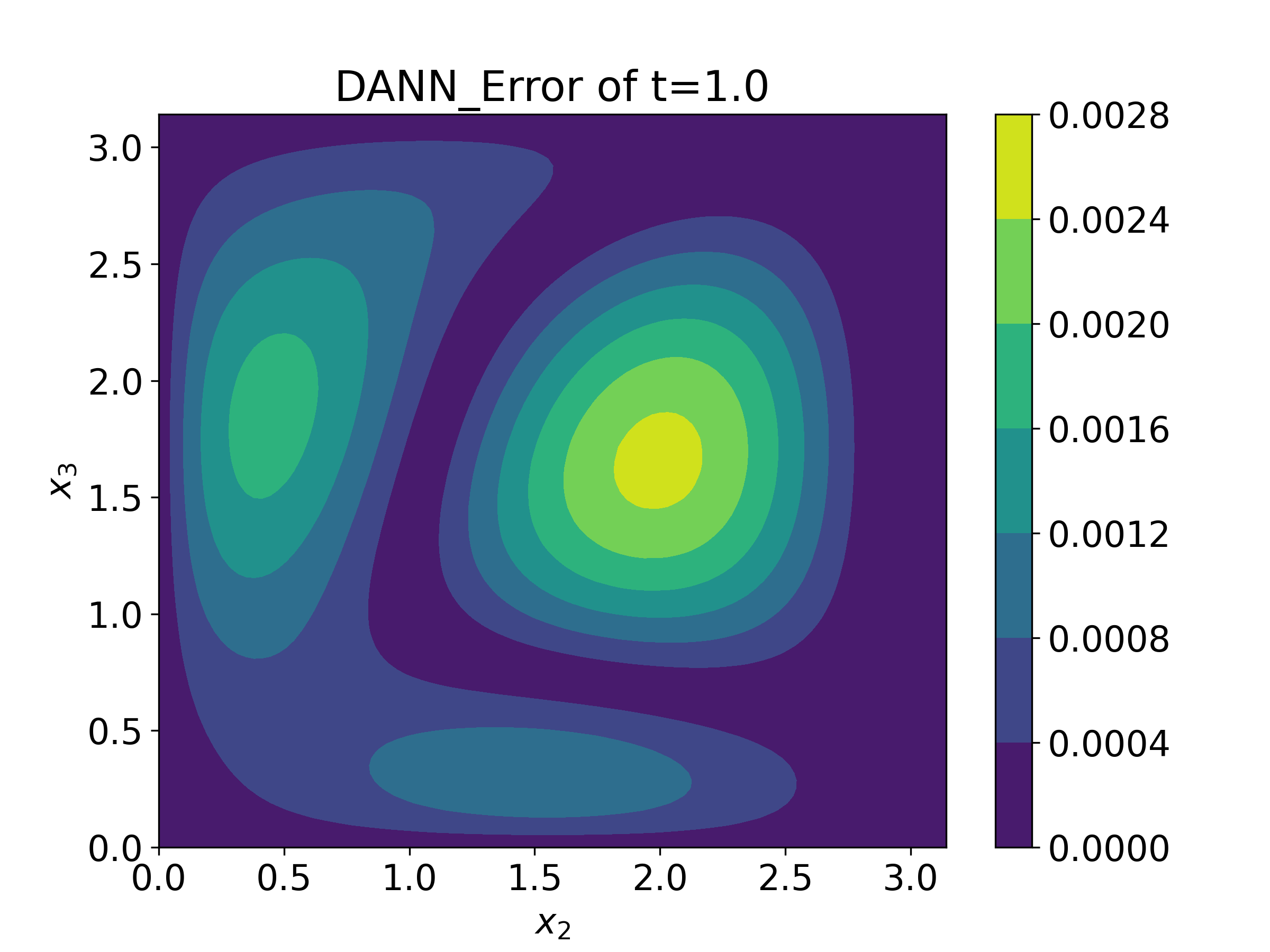}
		
		\label{chutian4}
	\end{minipage}
 \begin{minipage}{0.27\linewidth}
		\centering
		\includegraphics[width=\linewidth]{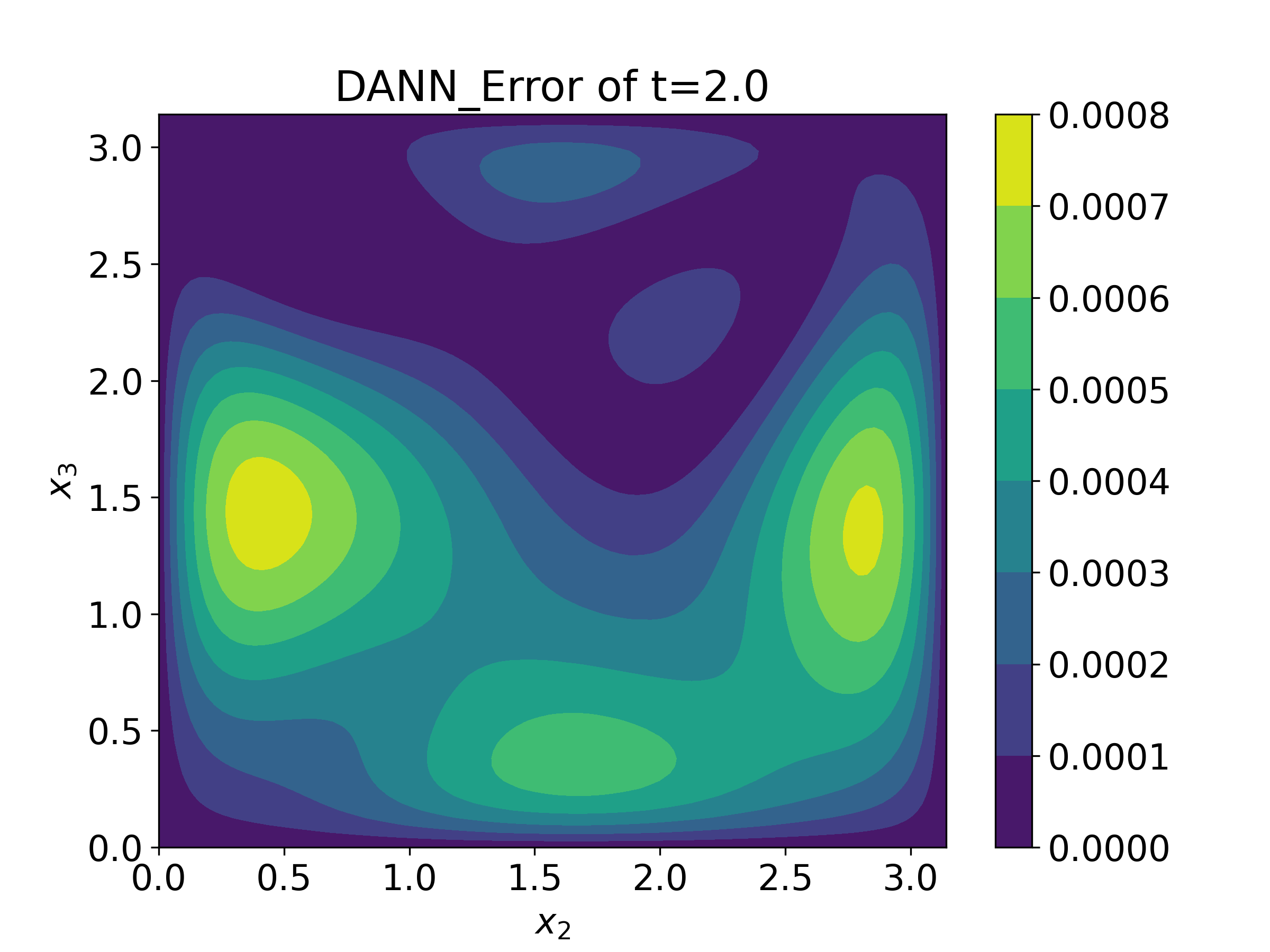}
		
		\label{chutian2}
	\end{minipage}
\caption{The absolute errors obtained by different models based on piecewise fitting at $t=0.0, 1.0, 2.0$ when $d=8$.}
\label{delay6wucha}
\end{figure}
\noindent fferent models based on piecewise fitting when $d=3$ and $d=8$, which shows that DANN is faster than PINN, APINN and ISC. Although the convergence speed of QSC and QRES is similar to that of DANN, the approximate solution obtained by DANN has higher accuracy.

\begin{figure}[H]
	\centering
 \begin{minipage}{0.4\linewidth}
		\centering
		\includegraphics[width=\linewidth]{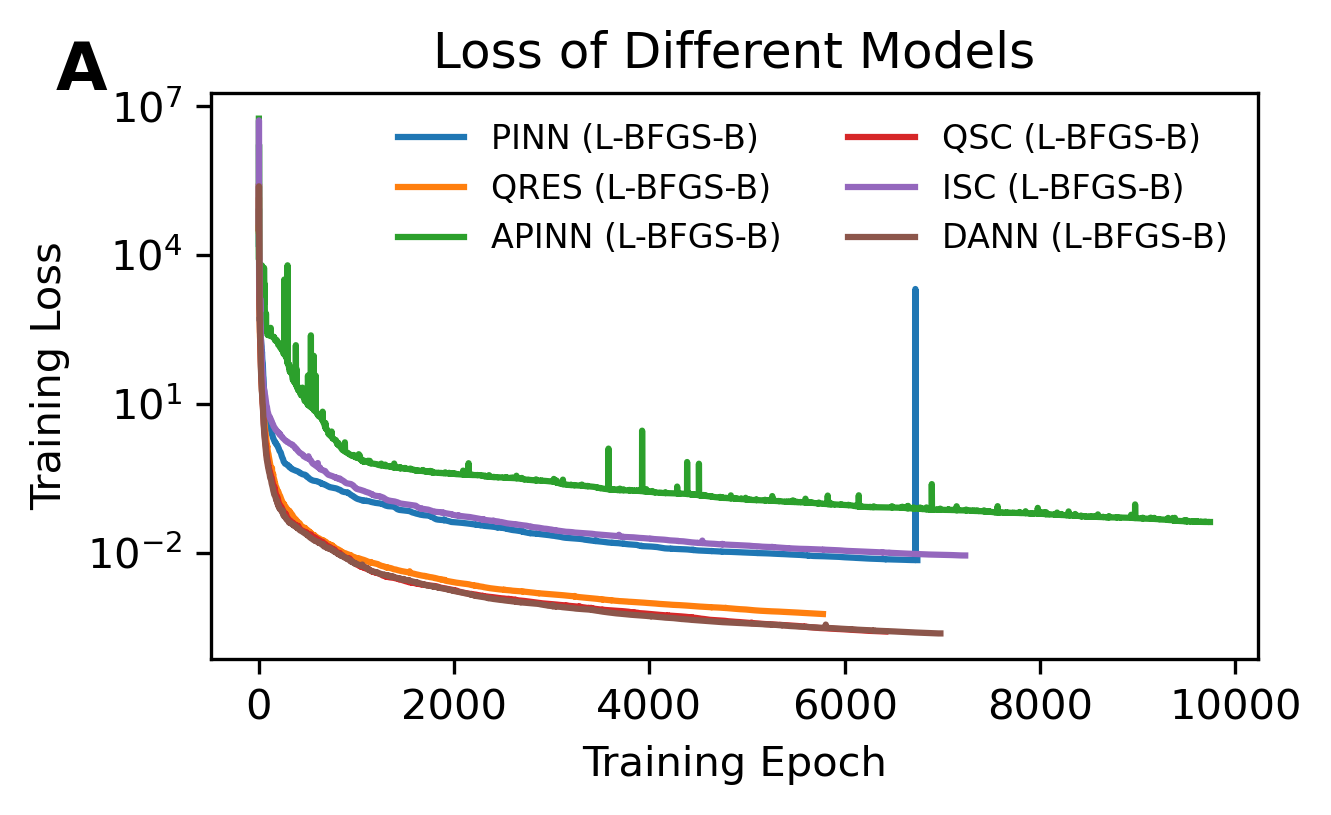}
		
		\label{chutian3}
	\end{minipage}
	\begin{minipage}{0.4\linewidth}
		\centering
		\includegraphics[width=\linewidth]{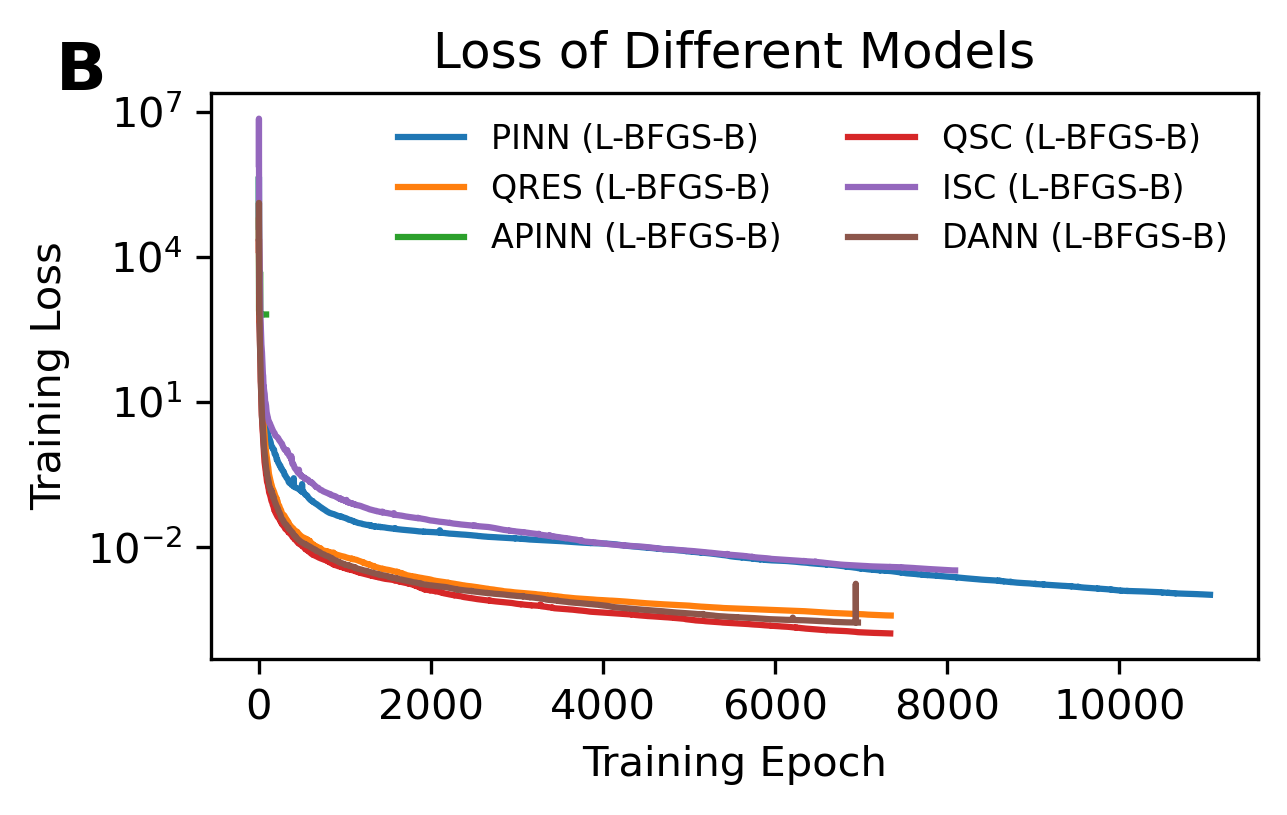}
		
		\label{chutian3}
	\end{minipage}

	\begin{minipage}{0.4\linewidth}
		\centering
		\includegraphics[width=\linewidth]{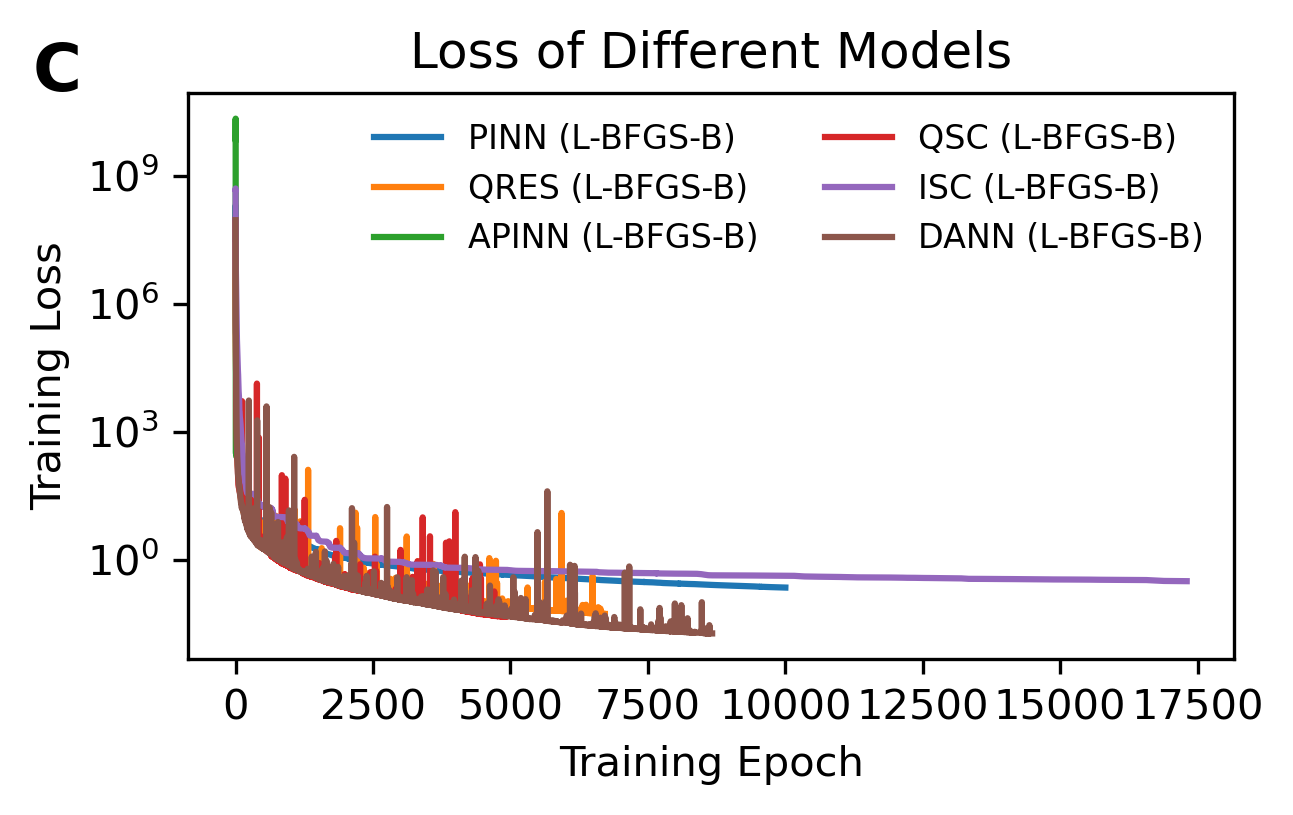}
		
		\label{chutian4}
	\end{minipage}
 \begin{minipage}{0.4\linewidth}
		\centering
		\includegraphics[width=\linewidth]{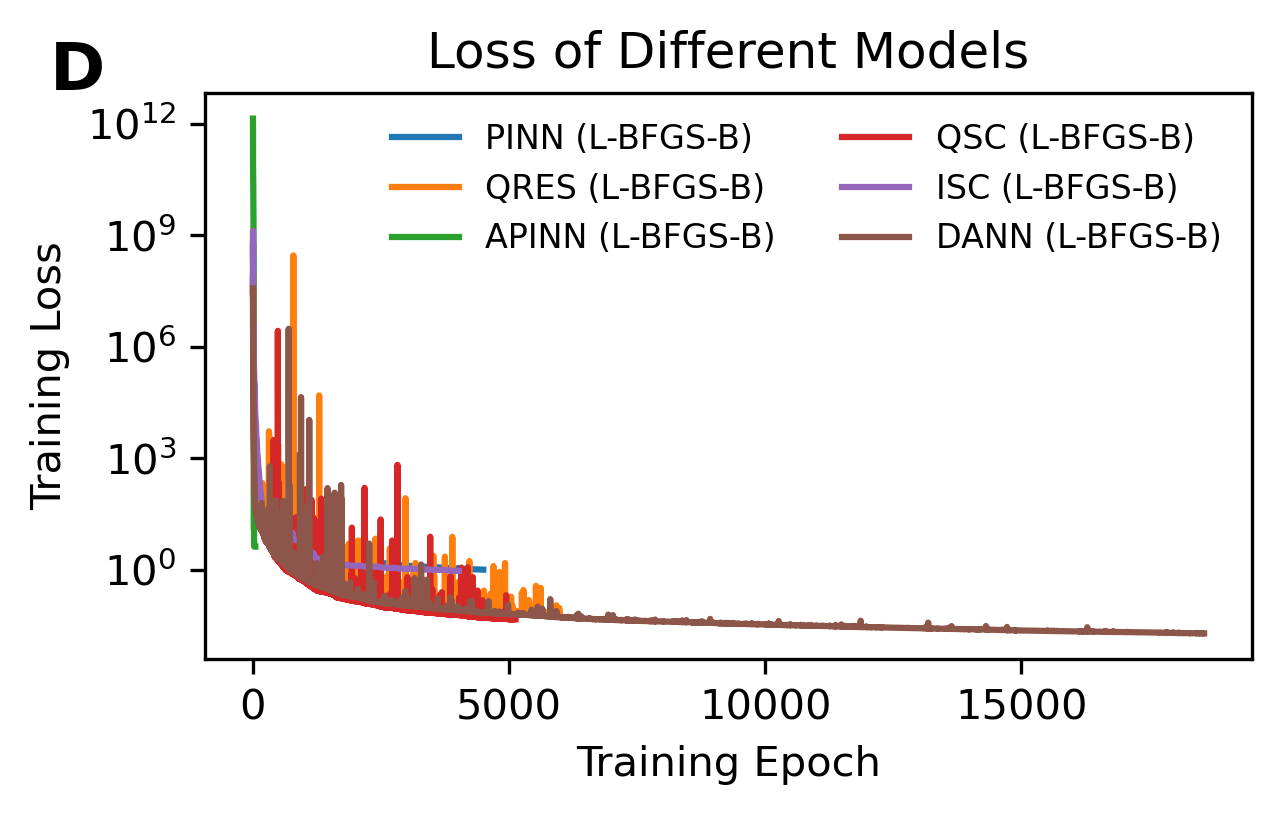}
		
		\label{chutian2}
	\end{minipage}
	\caption{Training loss curves for solving Example 4 based on piecewise fitting when $d=3$ (Subdomain 1) (A) and (Subdomain 2) (B) and $d=8$ (Subdomain 1) (C) and (Subdomain 2) (D).}
	\vspace{-0.2cm}
\label{delay5loss}
\end{figure}

\section{Conclusion}
In this paper, we proposed double-activation neural network (DANN), a new network architecture with two activation functions and an additional parameter  to augment the network's fitting performance for solving parabolic equations with time delay. We demonstrated the effectiveness of DANN through four numerical examples, including DDE with state-dependent and parabolic DPDEs with nonvanishing delay. Moreover, a piecewise fitting approach is proposed to address the issue of low fitting accuracy caused by the discontinuity of solution’s derivative. The numerical results demonstrated that DANN significantly outperforms PINN in terms of fitting accuracy and convergence speed, particularly when employing piecewise fitting. For comparison, we utilized APINN, QRES, ISC, and QSC models to solve these problems, and the results indicated that DANN achieves highest fitting accuracy.
Additionally, the convergence analysis of the proposed method was given, it was proved theoretically that the loss function obtained by PINN converges to 0 for the problem discussed in this paper. In this study, we did not assign weights to each item in the loss function, and it is an interesting research topic in the future to automatically determine an optimal weight. Moreover, it is possible to combine DANN with residual-based adaptive refinement (RAR) \cite{lu2021deepxde} to further improve the performance.

\bibliographystyle{abbrv}
\bibliography{xiugaidelay.bib}
\end{document}